\documentclass[a4paper,11pt]{amsart}
\usepackage{amssymb,amsmath,amsthm}
\numberwithin{equation}{section}
\usepackage{natbib}
\usepackage{paralist}               
\usepackage[margin=1.0in]{geometry}
\usepackage[titletoc]{appendix}     
\usepackage{mathrsfs}
\usepackage{setspace}
\usepackage[final]{graphicx}
\usepackage{epstopdf}
\usepackage{subcaption}
\usepackage{float}
\usepackage[bookmarks,bookmarksnumbered,bookmarksopen,breaklinks,linktocpage]{hyperref}
\newcommand{\ds}{\displaystyle}

\newcommand{\dd}{\mathrm{d}}

\newtheorem{Theorem}{Theorem}[section]
\newtheorem{Proposition}{Proposition}[section]

\theoremstyle{definition}           
\newtheorem{defn}{Definition}[section]
\theoremstyle{remark}
\newtheorem{Remark}{Remark}[section]
\usepackage{multirow}
\usepackage{color,soul}
\usepackage{soul}
\usepackage{array}
\newcolumntype{L}[1]{>{\raggedright\let\newline\\\arraybackslash\hspace{0pt}}m{#1}}
\newcolumntype{C}[1]{>{\centering\let\newline\\\arraybackslash\hspace{0pt}}m{#1}}
\newcolumntype{R}[1]{>{\raggedleft\let\newline\\\arraybackslash\hspace{0pt}}m{#1}}

\begin{document}
	\title[Mixed $3$ \& $4$-wave kinetic equation]{An energy cascade finite volume scheme for a mixed 3- and 4-wave kinetic equation arising from the theory of finite-temperature trapped Bose gases}
	
	\author[Arijit Das]{Arijit Das}
	
	\address[Arijit Das]{Department of Mathematics \\
		Thapar Institute of Engineering and Technology, Patiala-147004, Punjab, INDIA}
	
	\email[Arijit Das]{\href{mailto:arijit.das@thapar.edu}{arijit.das@thapar.edu}}
	
	\author[Minh-Binh Tran]{Minh-Binh Tran}
	\address[Minh-Binh Tran]{Department of Mathematics \\
		Texas A\&M University, College Station, TX, 77843 USA}
	
	\email[Minh-Binh Tran]{\href{mailto:minhbinh@tamu.edu}{minhbinh@tamu.edu}}
	\thanks{ M.-B. T is  funded in part by  a   Humboldt Fellowship,   NSF CAREER  DMS-2303146, and NSF Grants DMS-2204795, DMS-2305523,  DMS-2306379. }
	
	\begin{abstract}
Building on recent developments in numerical schemes designed to capture energy cascades for 3-wave kinetic equations~\cite{das2024numerical, walton2022deep, walton2023numerical, walton2024numerical}, we construct in this work a finite-volume algorithm for a significantly more complex wave kinetic equation whose collision operator incorporates both 3-wave and 4-wave interactions. This model arises in the context of finite-temperature Bose–Einstein condensation. We establish theoretical properties of the proposed scheme, and our numerical experiments demonstrate that it successfully captures the energy cascade behavior predicted by the equation.

	\end{abstract}
	
	\maketitle
	\allowdisplaybreaks
	\tableofcontents
	\section{Introduction}
	
Wave turbulence theory, which addresses the dynamics of nonlinear wave fields far from thermal equilibrium, traces its origins to the seminal contributions of several early pioneers. Foundational insights were introduced by Peierls~\cite{Peierls:1993:BRK,Peierls:1960:QTS}, followed by major developments from Brout and Prigogine~\cite{brout1956statistical}, Zaslavskii and Sagdeev~\cite{zaslavskii1967limits}, and Hasselmann~\cite{hasselmann1962non,hasselmann1974spectral}. Later advances by Benney, Saffman, and Newell~\cite{benney1966nonlinear,benney1969random}, together with the influential work of Zakharov~\cite{zakharov2012kolmogorov}, consolidated the framework.

In a recent series of works, we have developed numerical schemes for the isotropic 3-wave kinetic equations \cite{das2024numerical,walton2022deep,walton2023numerical,walton2024numerical} that accurately capture the associated energy cascade. This cascade manifests as a progressive loss of energy from the solutions over time. Continuing this research program, we consider in the present work a much more sophisticated wave kinetic equation whose collision operator incorporates both 3-wave and 4-wave interaction processes.
 We develop a finite-volume scheme capable of capturing the corresponding loss of energy in the solution. The equation reads \begin{equation}
	\label{wavekinetic}
	\partial_t f(t,k) = \mathbb{Q}[f](t,k)
	:= C_{12}[f](t,k) + C_{22}[f](t,k), 
	\qquad 
	f(0,k) = f^{\mathrm{in}}(k),
\end{equation}
where the collision operator \( \mathbb{Q}[f] \) is composed of contributions from both 3-wave and 4-wave interaction processes, detailed below.

\medskip
\noindent\textit{3-wave interaction term:}
\begin{align}
	\label{C12}
	\begin{split}
		C_{12}[f] :=\ &{c}_{12}
		\iint_{\mathbb{R}^3 \times \mathbb{R}^3}
		\mathrm{d}k_1\, \mathrm{d}k_2 \;
		\mathcal{W}_{12}(|k|,|k_1|,|k_2|)\,
		\delta(\omega - \omega_1 - \omega_2)\,
		\delta(k - k_1 - k_2)
		\big[ f_1 f_2 - (f_1 + f_2)f \big] \\[0.3em]
		&\ - 2
		\iint_{\mathbb{R}^3 \times \mathbb{R}^3}
		\mathrm{d}k_1\, \mathrm{d}k_2 \;
		\mathcal{W}_{12}(|k_1|,|k|,|k_2|)\,
		\delta(\omega_1 - \omega - \omega_2)\,
		\delta(k_1 - k - k_2)
		\big[ f f_2 - (f + f_2)f_1 \big].
	\end{split}
\end{align}

\medskip
\noindent\textit{4-wave interaction term:}
\begin{equation}
	\label{C22}
	\begin{aligned}
		C_{22}[f] :=\ &{c}_{22}
		\iiint_{\mathbb{R}^{3 \times 3}}
		\mathrm{d}k_1\, \mathrm{d}k_2\, \mathrm{d}k_3 \;
		\mathcal{W}_{22}(|k|,|k_1|,|k_2|,|k_3|)\,
		\delta(k + k_1 - k_2 - k_3)\,
		\delta(\omega + \omega_1 - \omega_2 - \omega_3) \\
		&\qquad\qquad \times
		\big[ f_2 f_3 (f_1 + f) - f f_1 (f_2 + f_3) \big].
	\end{aligned}
\end{equation}

In the expressions above, the parameters ${c}_{12}$ and ${c}_{22}$ are positive constants, $t \in \mathbb{R}_+$ denotes time, and $k \in \mathbb{R}^3$ represents the three-dimensional momentum variable. The function $\omega(k)=\omega(|k|)$ specifies the dispersion relation of the waves, while $f_0(k)=f_0(|k|)\ge 0$ denotes the initial particle density of the thermal cloud. Throughout, we assume that both the dispersion relation $\omega$ and the solution $f$ are radial; precise conditions on $\omega$ will be given   below. Under this radial symmetry, we may identify $f(k)$ with $f(|k|)$ and equivalently with $f(\omega)$.

For convenience, we employ the shorthand notation
\begin{equation}\label{Shorthand}
	\begin{aligned}
		& f = f(k) = f(|k|) = f(\omega), \qquad
		f_1 = f(k_1) = f(|k_1|) = f(\omega_1), \qquad
		f_2 = f(k_2) = f(|k_2|) = f(\omega_2), \\
		&\text{and similarly for } \omega = \omega(k) = \omega(|k|), \ \text{etc.}
	\end{aligned}
\end{equation}
The kernels are given as follows:
\begin{equation}\label{Kernel}
	\begin{aligned}
		\mathcal{W}_{12}(|k|,|k_1|,|k_2|) \ &= (\omega\,\omega_1\,\omega_2)^\sigma,\\
		\mathcal{W}_{22}(|k|,|k_1|,|k_2|,|k_3|) \ &= (\omega\,\omega_1\,\omega_2\,\omega_3)^\gamma,
	\end{aligned}
\end{equation}
where \( \sigma \ge 0 \) and \( \gamma \ge 0 \) are prescribed constants. In our numerical tests, we assume
\begin{equation}
	\label{omega}
	\omega(k) = |k|^\rho, \qquad \rho \ge 1.
\end{equation}

Below, we recall the concept of the \emph{energy cascade}, which was rigorously established in the recent works \cite{soffer2018energy,staffilani2024energy,staffilani2024condensation,staffilani2025energyfinite}. Initially, the energy is defined as
\begin{equation}
	\int_{\mathbb{R}^3} \mathrm{d}k\, f^{\mathrm{in}}(k)\,\omega(k)
	= 2\pi \int_{\mathbb{R}_+} \mathrm{d}\omega\, f^{\mathrm{in}}(\omega)\,\frac{|k|^2}{\omega'(|k|)} .
\end{equation}
As time evolves, the energy is expected to leave any bounded domain, causing the quantity
\begin{equation}
	\int_{[0,R)} \mathrm{d}\omega\, f(t,\omega)\,\frac{|k|^2}{\omega'(|k|)},
\end{equation}
to decrease for every $R > 0$ and eventually vanish.
 Numerical schemes for the isotropic 3-wave kinetic equations that accurately capture the associated energy cascade have been developed in \cite{das2024numerical, walton2022deep, walton2023numerical, walton2024numerical}. In this work, we construct numerical schemes capable of capturing this phenomenon for isotropic solutions of \eqref{wavekinetic}.

 The experimental achievement of Bose--Einstein condensation (BEC) in dilute atomic gases~\cite{WiemanCornell, Ketterle, bradley1995evidence} represented a major breakthrough in quantum statistical physics and sparked an extensive body of theoretical and experimental research. In these landmark experiments, rapid evaporative cooling is used to lower the temperature of a Bose gas below its critical value, triggering the formation of a macroscopic condensate. A key issue in this area is understanding the nonequilibrium dynamics of the condensate at finite temperatures, where the thermal (non-condensed) component is governed by kinetic equations, as discussed in Subsection \ref{Subs:PhysicalContext} below. \textit{Therefore, numerical simulations of \eqref{wavekinetic} are essential, since reliable simulations can inform and guide experimental physicists before they carry out actual experiments. Our work represents a step toward this goal, especially in verifying the energy cascade phenomenon.}

	\subsection{Physical context }\label{Subs:PhysicalContext}

A systematic kinetic theory for Bose gases at finite temperature was first developed by Kirkpatrick and Dorfman~\cite{KD1, KD2}, who derived kinetic equations describing the behavior of particles outside the condensate. Building on their work, Zaremba, Nikuni, and Griffin introduced in~\cite{ZarembaNikuniGriffin:1999:DOT} a coupled model in which a quantum Boltzmann equation governs the thermal cloud while a Gross--Pitaevskii equation describes the evolution of the condensate. Independently, Pomeau, Brachet, Metens, and Rica proposed a closely related kinetic framework in~\cite{PomeauBrachetMetensRica}.

The kinetic descriptions proposed in these works feature two principal classes of collision mechanisms, each represented by a distinct operator whose precise definitions will be given later:
\begin{itemize}
	\item \( \mathscr{C}_{22} \): accounts for binary \(2 \leftrightarrow 2\) scattering events among thermal excitations;
	\item \( \mathscr{C}_{12} \): represents \(1 \leftrightarrow 2\) processes involving one condensate particle and two excited particles.
\end{itemize}

An important extension of this framework was introduced by Reichl and Gust~\cite{ReichlGust:2012:CII}, who identified a third operator, \( \mathscr{C}_{31} \), describing \(1 \leftrightarrow 3\) scattering events within the excitation cloud. The mathematical derivation of this term was established only recently by Tran and Pomeau~\cite{tran2020boltzmann}, and its physical significance has been reinforced by experimental evidence reported in~\cite{reichl2019kinetic}.
For comprehensive discussions of these kinetic frameworks and their significance in modeling the dynamics of Bose--Einstein condensation, we refer the reader to the reviews~\cite{GriffinNikuniZaremba:BCG:2009, PomeauBinh, tran2021thermal}.

The quantum kinetic description of a Bose gas incorporating the collision mechanisms \( \mathscr{C}_{22} \) and \( \mathscr{C}_{12} \) takes the form
\begin{equation}
	\label{KineticFinal}
	\partial_t f(t,k)
	= \mathscr{C}_{12}[f](t,k) + \mathscr{C}_{22}[f](t,k),
	\qquad
	f(0,k) = f^{\mathrm{in}}(k).
\end{equation}
Their explicit forms are as follows.

\paragraph{\textit{The \( \mathscr{C}_{12} \) collision operator.}}
\begin{equation}
	\label{C12Discrete}
	\begin{aligned}
		\mathscr{C}_{12}[f](t,k)
		=\;&
		4\pi g^2 n 
		\iiint_{\mathbb{R}^3} \mathrm{d}k_1 \,\mathrm{d}k_2\,\mathrm{d}k_3\;
		\big[ \delta(k - k_1) - \delta(k - k_2) - \delta(k - k_3) \big]
		\\
		&\times \delta\!\left( \omega(k_1) - \omega(k_2) - \omega(k_3) \right)
	\mathcal{W}_{12}
		\delta(k_1 - k_2 - k_3)
		\\
		&\times
		\big[
		f(k_2) f(k_3) (f(k_1)+1)
		- f(k_1)(f(k_2)+1)(f(k_3)+1)
		\big].
	\end{aligned}
\end{equation}

\paragraph{\textit{The \( \mathscr{C}_{22} \) collision operator.}}
\begin{equation}
	\label{C22Discrete}
	\begin{aligned}
		\mathscr{C}_{22}[f](t,k)
		=\;&
		\pi g^2 
		\iiiint_{\mathbb{R}^3}
		\mathrm{d}k_1\,\mathrm{d}k_2\,\mathrm{d}k_3\,\mathrm{d}k_4\;
		\big[ \delta(k - k_1) + \delta(k - k_2)
		- \delta(k - k_3) - \delta(k - k_4) \big]
		\\
		&\times
		\mathcal{W}_{22}\,
		\delta(k_1 + k_2 - k_3 - k_4)\,
		\delta\!\left( \omega(k_1) + \omega(k_2)
		- \omega(k_3) - \omega(k_4) \right)
		\\
		&\times
		\big[
		f(k_3) f(k_4) (f(k_1)+1)(f(k_2)+1)
		- f(k_1) f(k_2) (f(k_3)+1)(f(k_4)+1)
		\big].
	\end{aligned}
\end{equation}

Here \( n \) denotes the condensate density and \( g \) is the interaction strength. The functions
\(
\mathcal{W}_{12},\;
\mathcal{W}_{22}
\)
are the corresponding collision kernels and are known explicitly.
For treatments allowing time dependence of the condensate density, we refer the reader to the book~\cite{PomeauBinh}. We note that the global existence of classical solutions to \eqref{KineticFinal} was established in \cite{soffer2018dynamics}, and the question of condensate growth was also investigated in \cite{staffilani2025condensate}.

For convenience, we introduce the shorthand constants
\begin{equation} \label{Simpl1}
	c_{12} := 4\pi g^2 n, 
	\qquad 
	c_{22} := \pi g^2, 
	\qquad 
	{c}_{31} := \pi g^2.
\end{equation}

A standard simplification in kinetic models is to keep only the leading-order nonlinear contributions within the collision integrals while discarding lower-order terms (see, for instance \cite{spohn2010kinetics}). Under this approximation, the nonlinear factors appearing in the collision operators may be replaced by their dominant components. Specifically:
\begin{itemize}
	\item For the \( \mathscr{C}_{12} \) operator,
	\[
	f(k_2) f(k_3) 
	- f(k_1)\big( f(k_2) + f(k_3) + 1 \big)
	\quad \longrightarrow \quad
	f(k_2) f(k_3) 
	- f(k_1)\big( f(k_2) + f(k_3) \big),
	\]
	\item For the \( \mathscr{C}_{22} \) operator,
	\[
	f(k_3) f(k_4) (f(k_1)+1)(f(k_2)+1)
	- f(k_1) f(k_2) (f(k_3)+1)(f(k_4)+1)
	\]
	\[
	\longrightarrow\quad
	f(k_3) f(k_4)\big( f(k_1) + f(k_2) \big)
	\;-\;
	f(k_1) f(k_2)\big( f(k_3) + f(k_4) \big).
	\]
\end{itemize}

Applying these approximations yields the mixed 3-wave and 4-wave kinetic equation \eqref{wavekinetic}.

\begin{Remark}
The numerical simulation of $\mathscr{C}_{13}$ is highly complicated and will be investigated in a forthcoming work.

	\end{Remark}
Finally, we summarize the current analytical and numerical developments for both 4-wave and 3-wave kinetic equations:

\begin{itemize}
	
	\item \textit{Derivation of wave kinetic equations.}  
	A major breakthrough in the rigorous derivation of wave turbulence equations was achieved in the recent series of works by Deng and Hani~\cite{deng2021propagation, deng2022rigorous, deng2023long, deng2021full}, where the wave turbulence   description was justified for long-time dynamics.
	
	\item \textit{Numerical methods for wave kinetic equations.}  
	For 4-wave interactions, a fast Fourier spectral scheme was developed in~\cite{qi2025fast}.  
	For 3-wave interactions, a numerical method based on piecewise polynomial approximation of the resonant manifold was introduced in~\cite{banks2025new}.
	
	\item \textit{Analysis of 4-wave kinetic equations of \(2 \leftrightarrow 2\) type.}  
	Convergence rates for discrete approximations and local well-posedness for the one-dimensional MMT model have been obtained in~\cite{dolce2024convergence, germain2023local}.  
	The stability and instability of Kolmogorov–Zakharov (KZ) spectra, along with stability near equilibrium, have been investigated in~\cite{menegaki20222, escobedo2024instability, collot2024stability}.  
	Further results on local existence and propagation of moments in polynomially weighted \(L^\infty\) spaces appear in~\cite{GermainIonescuTran, ampatzoglou2025inhomogeneous}.  
	On spatially periodic domains, recent work has addressed entropy maximizers and spectral stability~\cite{escobedo2024entropy, germain2024stability}.
	
	\item \textit{Analysis of 3-wave kinetic equations.}  
	The 3-wave kinetic models arise  in a wide range of physical settings have been investigated, including stratified ocean flows~\cite{GambaSmithBinh,kim2025wave}, Bose–Einstein condensates~\cite{cortes2020system,EPV, escobedo2023linearized1,escobedo2023linearized, escobedo2025local,nguyen2017quantum}, phonon interactions in crystal lattices~\cite{CraciunBinh, EscobedoBinh,tran2020reaction}, capillary waves~\cite{ nguyen2017quantum}, and beam–wave interactions~\cite{rumpf2021wave}.  
	
\end{itemize}

	\section{Reformulating \eqref{wavekinetic}}

The goal of this section is to reformulate equation \eqref{wavekinetic} into the form \eqref{1_6}, which is more suitable for the construction of a numerical scheme.
First, following the standard angular integration procedure 
\cite{staffilani2024energy,staffilani2024condensation,staffilani2025condensate,staffilani2025energyfinite},
we can rewrite \eqref{wavekinetic} in the form
\begin{equation}\label{4wave1}
	\begin{aligned}
		\partial_t f(t,\omega)
		= \mathcal{Q}[f]
		=\;&
		C_1 \iiint_{\mathbb{R}_+^{3}}
		\mathrm{d}\omega_1 \frac{1}{\omega_1'}\,
		\mathrm{d}\omega_2 \frac{1}{\omega_2'}\,
		\mathrm{d}\omega_3 \frac{1}{\omega_3'}\;
		\delta(\omega + \omega_1 - \omega_2 - \omega_3)
		(\omega \omega_1 \omega_2 \omega_3)^\sigma
		\\
		&\times
		\big[
		f(\omega_2) f(\omega_3)\big( f(\omega_1) + f(\omega) \big)
		- f(\omega) f(\omega_1)\big( f(\omega_2) + f(\omega_3) \big)
		\big]\,\\
		&\qquad \times
		\frac{|k_1||k_2||k_3|\min\{|k_1|,|k_2|,|k_3|,|k|\}}{|k|}
		\\[0.4em]
		&\;+\;
		C_2 \iint_{\mathbb{R}_+^{2}}
		\mathrm{d}\omega_1 \frac{1}{\omega_1'}\,
		\mathrm{d}\omega_2 \frac{1}{\omega_2'}\;
		\delta(\omega - \omega_1 - \omega_2)
		(\omega \omega_1 \omega_2)^\gamma
		\\
		&\qquad \times
		\big[
		f(\omega_2) f(\omega_3)
		- f(\omega_1) f(\omega)
		- f(\omega_2) f(\omega)
		\big]\,
		\frac{|k_1||k_2|}{|k|}
		\\[0.4em]
		&\;-\;
		2 C_2 \iint_{\mathbb{R}_+^{2}}
		\mathrm{d}\omega_1 \frac{1}{\omega_1'}\,
		\mathrm{d}\omega_2 \frac{1}{\omega_2'}\;
		\delta(\omega_1 - \omega - \omega_2)
		(\omega \omega_1 \omega_2)^\gamma
		\\
		&\qquad \times
		\big[
		f(\omega) f(\omega_2)
		- f(\omega_1) f(\omega)
		- f(\omega_1) f(\omega)
		\big]\,
		\frac{|k_1||k_2|}{|k|},
	\end{aligned}
\end{equation}
where \(C_1\) and \(C_2\) are constants depending on the physical coefficients and $c_{12}, c_{22}$, and
\(\omega_1' = \omega'(|k_1|)\), 
\(\omega_2' = \omega'(|k_2|)\),
\(\omega_3' = \omega'(|k_3|)\). 

To simplify the expressions that follow, we set \(\omega_1 = \mu\), \(\omega_2 = \eta\), and \(\omega_3 = \nu\), and obtain from \eqref{4wave1}
\begin{equation}\label{4wave}
	\begin{aligned}
		\partial_t f = \mathcal{Q}[f]
		=\;&
		C_1 \iiint_{\mathbb{R}_+^{3}}
		\mathrm{d}\mu\,\frac{1}{\mu'}\,
		\mathrm{d}\eta\,\frac{1}{\eta'}\,
		\mathrm{d}\nu\,\frac{1}{\nu'}\;
		\delta(\omega + \mu - \eta - \nu)
		(\omega \mu \eta \nu)^\sigma
		\\
		&\times
		\big[
		f_\eta f_\nu (f_\mu + f_\omega)
		- f_\omega f_\mu (f_\eta + f_\nu)
		\big]\,
		\frac{|k_1||k_2||k_3|\min\{|k_1|,|k_2|,|k_3|,|k|\}}{|k|}
		\\[0.4em]
		&+ C_2
		\iint_{\mathbb{R}_+^{2}}
		\mathrm{d}\mu\,\frac{1}{\mu'}\,
		\mathrm{d}\eta\,\frac{1}{\eta'}\;
		\delta(\omega - \mu - \eta)
		(\omega \mu \eta)^\gamma
		\big[
		f_\mu f_\eta - f_\mu f_\omega - f_\eta f_\omega
		\big]\,
		\frac{|k_1||k_2|}{|k|}
		\\[0.4em]
		&- 2 C_2
		\iint_{\mathbb{R}_+^{2}}
		\mathrm{d}\mu\,\frac{1}{\mu'}\,
		\mathrm{d}\eta\,\frac{1}{\eta'}\;
		\delta(\mu - \omega - \eta)
		(\omega \mu \eta)^\gamma
		\big[
		f_\omega f_\eta - f_\mu f_\omega - f_\eta f_\mu
		\big]\,
		\frac{|k_1||k_2|}{|k|},
	\end{aligned}
\end{equation}
where \(f_\omega, f_\mu, f_\eta, f_\nu\) denote \(f(\omega)\), \(f(\mu)\), \(f(\eta)\), and \(f(\nu)\), respectively.

Since \(|k|\) is a function of \(\omega\), namely 
\(|k_1| = |k|(\mu)\), \(|k_2| = |k|(\eta)\), and \(|k_3| = |k|(\nu)\), 
the equation above can be rewritten as follows:\allowdisplaybreaks
\begin{align}\label{1_1}
	\begin{split}
		\partial_t f =\;&
		C_1 \iiint_{\mathbb{R}_+^{3}}
		|k|'(\mu)\, |k|'(\eta)\, |k|'(\nu)\;
		\delta(\omega + \mu - \eta - \nu)
		(\omega \mu \eta \nu)^\sigma
		\big[
		f_\eta f_\nu (f_\mu + f_\omega)
		\\
		&\left.
		{} - f_\omega f_\mu (f_\eta + f_\nu)
		\right]
		\frac{
			|k|(\mu)\, |k|(\eta)\, |k|(\nu)\,
			\min\{|k|(\mu), |k|(\eta), |k|(\nu), |k|(\omega)\}
		}{
			|k|(\omega)
		}
		\,\mathrm{d}\mu\,\mathrm{d}\eta\,\mathrm{d}\nu
		\\[0.4em]
		&\;+\;
		C_2 \iint_{\mathbb{R}_+^2}
		|k|'(\mu)\, |k|'(\eta)\;
		\delta(\omega - \mu - \eta)
		(\omega\mu\eta)^\gamma
		\left[
		f_\mu f_\eta - f_\mu f_\omega - f_\eta f_\omega
		\right]
		\frac{|k|(\mu)\, |k|(\eta)}{|k|(\omega)}
		\,\mathrm{d}\mu\,\mathrm{d}\eta
		\\[0.4em]
		&\;-\;
		2C_2 \iint_{\mathbb{R}_+^2}
		|k|'(\mu)\, |k|'(\eta)\;
		\delta(\mu - \omega - \eta)
		(\omega\mu\eta)^\gamma
		\left[
		f_\omega f_\eta - f_\mu f_\omega - f_\eta f_\mu
		\right]
		\frac{|k|(\mu)\, |k|(\eta)}{|k|(\omega)}
		\,\mathrm{d}\mu\,\mathrm{d}\eta.
	\end{split}
\end{align}

Using the properties of the Dirac delta distribution, the above equation reduces to:
\allowdisplaybreaks
\begin{equation}
\begin{aligned}\label{1_2}
	\partial_t f
	&=
	C_1 \int_{0}^{\infty} \int_{0}^{\omega+\mu}\,\mathrm{d}\eta\,\mathrm{d}\mu
	\Bigg[
	\frac{
		\min\{ |k|(\mu), |k|(\eta), |k|(\omega+\mu-\eta), |k|(\omega) \}
	}{
		|k|(\omega)
	}
	|k|'(\mu)\, |k|'(\eta)\, |k|'(\omega+\mu-\eta)
	\\
	&\qquad\qquad
	\times |k|(\mu)\, |k|(\eta)\, |k|(\omega+\mu-\eta)
	(\omega \mu \eta (\omega+\mu-\eta))^{\sigma}
	\\
	&\qquad\qquad
	\times \Big(
	f(\eta)f(\omega+\mu-\eta)\big(f(\mu)+f(\omega)\big)
	- f(\omega)f(\mu)\big(f(\eta)+f(\omega+\mu-\eta)\big)
	\Big)
	\Bigg]
	\\[0.5em]
	&\quad
	+ C_2 \int_{0}^{\omega}\,\mathrm{d}\mu
	\frac{|k|(\mu)\,|k|(\omega-\mu)}{|k|(\omega)}
	|k|'(\mu)\,|k|'(\omega-\mu)
	(\omega\mu(\omega-\mu))^\gamma
	\\
	&\qquad\qquad \times
	\big[
	f(\mu)f(\omega-\mu)
	- f(\mu)f(\omega)
	- f(\omega-\mu)f(\omega)
	\big]
	\\[0.5em]
	&\quad
	- 2C_2 \int_{\omega}^{\infty}	\,\mathrm{d}\mu
	\frac{|k|(\mu)\,|k|(\mu-\omega)}{|k|(\omega)}
	|k|'(\mu)\,|k|'(\mu-\omega)
	(\omega\mu(\mu-\omega))^\gamma
	\\
	&\qquad\qquad \times
	\big[
	f(\omega)f(\mu-\omega)
	- f(\mu)f(\omega)
	- f(\mu-\omega)f(\mu)
	\big].
\end{aligned}
\end{equation}
	
We now simplify the above integrals separately. For the first term, we have
\begin{align}
	I_1
	&=
	C_1 \int_{0}^{\omega} \int_{0}^{\mu}\,\mathrm{d}\eta\,\mathrm{d}\mu
	\frac{
		\min\{|k|(\mu), |k|(\eta), |k|(\omega+\mu-\eta), |k|(\omega)\}
	}{
		|k|(\omega)
	}
	|k|'(\mu)\,|k|'(\eta)\,|k|'(\omega+\mu-\eta)
	\notag \\
	&\hspace{1.4cm}
	\times |k|(\mu)\,|k|(\eta)\,|k|(\omega+\mu-\eta)
	(\omega\mu\eta(\omega+\mu-\eta))^{\sigma}
	\Big[
	f(\eta)f(\omega+\mu-\eta)
	\notag \\
	&\hspace{1.4cm}
	\times (f(\mu)+f(\omega))
	- f(\omega)f(\mu)\big(f(\eta)+f(\omega+\mu-\eta)\big)
	\Big]
	\notag \\
	&\quad
	+ C_1 \int_{0}^{\omega} \int_{\mu}^{\omega}\,\mathrm{d}\eta\,\mathrm{d}\mu
	\frac{
		\min\{|k|(\mu), |k|(\eta), |k|(\omega+\mu-\eta), |k|(\omega)\}
	}{
		|k|(\omega)
	}
	|k|'(\mu)\,|k|'(\eta)\,|k|'(\omega+\mu-\eta)
	\notag \\
	&\hspace{1.4cm}
	\times |k|(\mu)\,|k|(\eta)\,|k|(\omega+\mu-\eta)
	(\omega\mu\eta(\omega+\mu-\eta))^{\sigma}
	\Big[
	f(\eta)f(\omega+\mu-\eta)
	\notag \\
	&\hspace{1.4cm}
	\times (f(\mu)+f(\omega))
	- f(\omega)f(\mu)\big(f(\eta)+f(\omega+\mu-\eta)\big)
	\Big]
	\notag \\
	&\quad
	+ C_1 \int_{0}^{\omega} \int_{\omega}^{\omega+\mu}\,\mathrm{d}\eta\,\mathrm{d}\mu
	\frac{
		\min\{|k|(\mu), |k|(\eta), |k|(\omega+\mu-\eta), |k|(\omega)\}
	}{
		|k|(\omega)
	}
	|k|'(\mu)\,|k|'(\eta)\,|k|'(\omega+\mu-\eta)
	\notag \\
	&\hspace{1.4cm}
	\times |k|(\mu)\,|k|(\eta)\,|k|(\omega+\mu-\eta)
	(\omega\mu\eta(\omega+\mu-\eta))^{\sigma}
	\Big[
	f(\eta)f(\omega+\mu-\eta)
	\notag \\
	&\hspace{1.4cm}
	\times (f(\mu)+f(\omega))
	- f(\omega)f(\mu)\big(f(\eta)+f(\omega+\mu-\eta)\big)
	\Big].
	\notag
\end{align}

Since \(|k|\) is a monotonically increasing function, we observe that
\begin{align*}
	&\text{when } 0 \le \mu \le \omega \text{ and } 0 \le \eta \le \mu,
	\quad
	\min\{|k|(\mu), |k|(\eta), |k|(\omega+\mu-\eta), |k|(\omega)\}
	= |k|(\eta),\\[0.4em]
	&\text{when } 0 \le \mu \le \omega \text{ and } \mu \le \eta \le \omega,
	\quad
	\min\{|k|(\mu), |k|(\eta), |k|(\omega+\mu-\eta), |k|(\omega)\}
	= |k|(\mu),\\[0.4em]
	&\text{when } 0 \le \mu \le \omega \text{ and } \omega \le \eta \le \omega+\mu,
	\quad
	\min\{|k|(\mu), |k|(\eta), |k|(\omega+\mu-\eta), |k|(\omega)\}
	= |k|(\omega+\mu-\eta).
\end{align*}

Using these relations in the integral \(I_1\) yields
\begin{align}\label{1_3}
	\begin{split}
		I_1
		&= \frac{C_1}{|k|(\omega)}
		\int_{0}^{\omega} \int_{0}^{\mu}
		\,\mathrm{d}\eta\,\mathrm{d}\mu
		|k|'(\mu)\, |k|'(\eta)\, |k|'(\omega+\mu-\eta)\,
		|k|(\mu)\, |k|^2(\eta)\, |k|(\omega+\mu-\eta)
		\\
		&\hspace{2.2cm}
		\times (\omega\mu\eta(\omega+\mu-\eta))^{\sigma}
		\Big[
		f(\eta) f(\omega+\mu-\eta)(f(\mu)+f(\omega))
		\\
		&\hspace{2.2cm}
		- f(\omega) f(\mu)\big(f(\eta)+f(\omega+\mu-\eta)\big)
		\Big]
		\\[0.6em]
		&\quad + \frac{C_1}{|k|(\omega)}
		\int_{0}^{\omega} \int_{\mu}^{\omega}	\,\mathrm{d}\eta\,\mathrm{d}\mu
		|k|'(\mu)\, |k|'(\eta)\, |k|'(\omega+\mu-\eta)\,
		|k|^2(\mu)\, |k|(\eta)\, |k|(\omega+\mu-\eta)
		\\
		&\hspace{2.2cm}
		\times (\omega\mu\eta(\omega+\mu-\eta))^{\sigma}
		\Big[
		f(\eta) f(\omega+\mu-\eta)(f(\mu)+f(\omega))
		\\
		&\hspace{2.2cm}
		- f(\omega) f(\mu)\big(f(\eta)+f(\omega+\mu-\eta)\big)
		\Big]
		\\[0.6em]
		&\quad + \frac{C_1}{|k|(\omega)}
		\int_{0}^{\omega} \int_{\omega}^{\omega+\mu}\,\mathrm{d}\eta\,\mathrm{d}\mu
		|k|'(\mu)\, |k|'(\eta)\, |k|'(\omega+\mu-\eta)\,
		|k|(\mu)\, |k|(\eta)\, |k|^2(\omega+\mu-\eta)
		\\
		&\hspace{2.2cm}
		\times (\omega\mu\eta(\omega+\mu-\eta))^{\sigma}
		\Big[
		f(\eta) f(\omega+\mu-\eta)(f(\mu)+f(\omega))
		\\
		&\hspace{2.2cm}
		- f(\omega) f(\mu)\big(f(\eta)+f(\omega+\mu-\eta)\big)
		\Big].
	\end{split}
\end{align}

Next, we consider the term
\begin{align}
	I_2
	=&\; C_1 \int_{\omega}^{\infty} \int_{0}^{\omega}\,\mathrm{d}\eta\,\mathrm{d}\mu
	\frac{
		\min\{|k|(\mu), |k|(\eta), |k|(\omega+\mu-\eta), |k|(\omega)\}
	}{
		|k|(\omega)
	}
	|k|'(\mu)\,|k|'(\eta)\,|k|'(\omega+\mu-\eta)
	\notag \\
	&\hspace{1.4cm}
	\times |k|(\mu)\,|k|(\eta)\,|k|(\omega+\mu-\eta)
	(\omega\mu\eta(\omega+\mu-\eta))^{\sigma}
	\Big[
	f(\eta) f(\omega+\mu-\eta)
	\notag \\
	&\hspace{1.4cm}
	\times (f(\mu)+f(\omega))
	- f(\omega)f(\mu)\big(f(\eta)+f(\omega+\mu-\eta)\big)
	\Big]
	\notag \\
	&\quad + C_1 \int_{\omega}^{\infty} \int_{\omega}^{\mu}	\,\mathrm{d}\eta\,\mathrm{d}\mu
	\frac{
		\min\{|k|(\mu), |k|(\eta), |k|(\omega+\mu-\eta), |k|(\omega)\}
	}{
		|k|(\omega)
	}
	|k|'(\mu)\,|k|'(\eta)\,|k|'(\omega+\mu-\eta)
	\notag \\
	&\hspace{1.4cm}
	\times |k|(\mu)\,|k|(\eta)\,|k|(\omega+\mu-\eta)
	(\omega\mu\eta(\omega+\mu-\eta))^{\sigma}
	\Big[
	f(\eta) f(\omega+\mu-\eta)
	\notag \\
	&\hspace{1.4cm}
	\times (f(\mu)+f(\omega))
	- f(\omega)f(\mu)\big(f(\eta)+f(\omega+\mu-\eta)\big)
	\Big]
	\notag \\
	&\quad + C_1 \int_{\omega}^{\infty} \int_{\mu}^{\omega+\mu}	\,\mathrm{d}\eta\,\mathrm{d}\mu
	\frac{
		\min\{|k|(\mu), |k|(\eta), |k|(\omega+\mu-\eta), |k|(\omega)\}
	}{
		|k|(\omega)
	}
	|k|'(\mu)\,|k|'(\eta)\,|k|'(\omega+\mu-\eta)
	\notag \\
	&\hspace{1.4cm}
	\times |k|(\mu)\,|k|(\eta)\,|k|(\omega+\mu-\eta)
	(\omega\mu\eta(\omega+\mu-\eta))^{\sigma}
	\Big[
	f(\eta) f(\omega+\mu-\eta)
	\notag \\
	&\hspace{1.4cm}
	\times (f(\mu)+f(\omega))
	- f(\omega)f(\mu)\big(f(\eta)+f(\omega+\mu-\eta)\big)
	\Big]
.
	\notag
\end{align}

Since \(|k|\) is a monotonically increasing function, we observe that
\begin{align*}
	&\text{when } \omega \le \mu < \infty \text{ and } 0 \le \eta \le \omega,
	\quad
	\min\{|k|(\mu), |k|(\eta), |k|(\omega+\mu-\eta), |k|(\omega)\}
	= |k|(\eta),\\[0.4em]
	&\text{when } \omega \le \mu < \infty \text{ and } \omega \le \eta \le \mu,
	\quad
	\min\{|k|(\mu), |k|(\eta), |k|(\omega+\mu-\eta), |k|(\omega)\}
	= |k|(\omega),\\[0.4em]
	&\text{when } \omega \le \mu < \infty \text{ and } \mu \le \eta \le \omega+\mu,
	\quad
	\min\{|k|(\mu), |k|(\eta), |k|(\omega+\mu-\eta), |k|(\omega)\}
	= |k|(\omega+\mu-\eta).
\end{align*}

Using these relations in the integral \(I_2\) yields
\begin{align}\label{1_4}
	\allowdisplaybreaks
	\begin{split}
		I_2
		&= \frac{C_1}{|k|(\omega)}
		\int_{\omega}^{\infty} \int_{0}^{\omega}	\,\mathrm{d}\eta\,\mathrm{d}\mu
		|k|'(\mu)\, |k|'(\eta)\, |k|'(\omega+\mu-\eta)\,
		|k|(\mu)\, |k|^2(\eta)\, |k|(\omega+\mu-\eta)
		\\
		&\hspace{2.2cm}
		\times (\omega\mu\eta(\omega+\mu-\eta))^{\sigma}
		\Big[
		f(\eta) f(\omega+\mu-\eta)\big(f(\mu)+f(\omega)\big)
		\\
		&\hspace{2.2cm}
		- f(\omega) f(\mu)\big(f(\eta)+f(\omega+\mu-\eta)\big)
		\Big]
		\\[0.6em]
		&\quad + \frac{C_1}{|k|(\omega)}
		\int_{\omega}^{\infty} \int_{\omega}^{\mu}	\,\mathrm{d}\eta\,\mathrm{d}\mu
		|k|'(\mu)\, |k|'(\eta)\, |k|'(\omega+\mu-\eta)\,
		|k|(\omega)\, |k|(\mu)\, |k|(\eta)\, |k|(\omega+\mu-\eta)
		\\
		&\hspace{2.2cm}
		\times (\omega\mu\eta(\omega+\mu-\eta))^{\sigma}
		\Big[
		f(\eta) f(\omega+\mu-\eta)\big(f(\mu)+f(\omega)\big)
		\\
		&\hspace{2.2cm}
		- f(\omega) f(\mu)\big(f(\eta)+f(\omega+\mu-\eta)\big)
		\Big]
		\\[0.6em]
		&\quad + \frac{C_1}{|k|(\omega)}
		\int_{\omega}^{\infty} \int_{\mu}^{\omega+\mu}\,\mathrm{d}\eta\,\mathrm{d}\mu
		|k|'(\mu)\, |k|'(\eta)\, |k|'(\omega+\mu-\eta)\,
		|k|(\mu)\, |k|(\eta)\, |k|^2(\omega+\mu-\eta)
		\\
		&\hspace{2.2cm}
		\times (\omega\mu\eta(\omega+\mu-\eta))^{\sigma}
		\Big[
		f(\eta) f(\omega+\mu-\eta)\big(f(\mu)+f(\omega)\big)
		\\
		&\hspace{2.2cm}
		- f(\omega) f(\mu)\big(f(\eta)+f(\omega+\mu-\eta)\big)
		\Big]
		.
	\end{split}
\end{align}

Now, combining \eqref{1_3} and \eqref{1_4} in equation \eqref{1_2}, the equation \eqref{4wave} can be written in the following form:
\begin{align}\label{1_5}\allowdisplaybreaks
	\begin{split}
		\partial_t f
		=&\; \frac{C_1}{|k|(\omega)}
		\int_{0}^{\omega} \int_{0}^{\mu}
		\,\mathrm{d}\eta\,\mathrm{d}\mu
		|k|'(\mu)\,|k|'(\eta)\,|k|'(\omega+\mu-\eta)\,
		|k|(\mu)\,|k|^2(\eta)\,|k|(\omega+\mu-\eta)
		\\
		&\hspace{2.2cm}\times(\omega\mu\eta(\omega+\mu-\eta))^{\sigma}
		\Big[
		f(\eta)f(\omega+\mu-\eta)(f(\mu)+f(\omega))
		\\
		&\hspace{2.2cm}
		- f(\omega)f(\mu)(f(\eta)+f(\omega+\mu-\eta))
		\Big]
		\\[0.6em]
		&\;+\frac{C_1}{|k|(\omega)}
		\int_{0}^{\omega} \int_{\mu}^{\omega}
		\,\mathrm{d}\eta\,\mathrm{d}\mu
		|k|'(\mu)\,|k|'(\eta)\,|k|'(\omega+\mu-\eta)\,
		|k|^2(\mu)\,|k|(\eta)\,|k|(\omega+\mu-\eta)
		\\
		&\hspace{2.2cm}\times(\omega\mu\eta(\omega+\mu-\eta))^{\sigma}
		\Big[
		f(\eta)f(\omega+\mu-\eta)(f(\mu)+f(\omega))
		\\
		&\hspace{2.2cm}
		- f(\omega)f(\mu)(f(\eta)+f(\omega+\mu-\eta))
		\Big]
		\\[0.6em]
		&\;+\frac{C_1}{|k|(\omega)}
		\int_{0}^{\omega} \int_{\omega}^{\omega+\mu}
		\,\mathrm{d}\eta\,\mathrm{d}\mu
		|k|'(\mu)\,|k|'(\eta)\,|k|'(\omega+\mu-\eta)\,
		|k|(\mu)\,|k|(\eta)\,|k|^2(\omega+\mu-\eta)
		\\
		&\hspace{2.2cm}\times(\omega\mu\eta(\omega+\mu-\eta))^{\sigma-1} 
		\Big[
		f(\eta)f(\omega+\mu-\eta)(f(\mu)+f(\omega))
		\\
		&\hspace{2.2cm}
		- f(\omega)f(\mu)(f(\eta)+f(\omega+\mu-\eta))
		\Big]
		\\[0.6em]
		&\;+\frac{C_1}{|k|(\omega)}
		\int_{\omega}^{\infty}\int_{0}^{\omega}	\,\mathrm{d}\eta\,\mathrm{d}\mu
		|k|'(\mu)\,|k|'(\eta)\,|k|'(\omega+\mu-\eta)\,
		|k|(\mu)\,|k|^2(\eta)\,|k|(\omega+\mu-\eta)
		\\
		&\hspace{2.2cm}\times(\omega\mu\eta(\omega+\mu-\eta))^{\sigma}
		\Big[
		f(\eta)f(\omega+\mu-\eta)(f(\mu)+f(\omega))
		\\
		&\hspace{2.2cm}
		- f(\omega)f(\mu)(f(\eta)+f(\omega+\mu-\eta))
		\Big]
			\\[0.6em]
		&\;+\frac{C_1}{|k|(\omega)}
		\int_{\omega}^{\infty}\int_{\omega}^{\mu}
		\,\mathrm{d}\eta\,\mathrm{d}\mu
		|k|'(\mu)\,|k|'(\eta)\,|k|'(\omega+\mu-\eta)\,
		|k|(\omega)\,|k|(\mu)\,|k|(\eta)\,|k|(\omega+\mu-\eta)
		\\
		&\hspace{2.2cm}\times(\omega\mu\eta(\omega+\mu-\eta))^{\sigma}
		\Big[
		f(\eta)f(\omega+\mu-\eta)(f(\mu)+f(\omega))
		\\
		&\hspace{2.2cm}
		- f(\omega)f(\mu)(f(\eta)+f(\omega+\mu-\eta))
		\Big]
		\\[0.6em]
		&\;+\frac{C_1}{|k|(\omega)}
		\int_{\omega}^{\infty}\int_{\mu}^{\omega+\mu}\,\mathrm{d}\eta\,\mathrm{d}\mu
		|k|'(\mu)\,|k|'(\eta)\,|k|'(\omega+\mu-\eta)\,
		|k|(\mu)\,|k|(\eta)\,|k|^2(\omega+\mu-\eta)
		\\
		&\hspace{2.2cm}\times(\omega\mu\eta(\omega+\mu-\eta))^{\sigma}
		\Big[
		f(\eta)f(\omega+\mu-\eta)(f(\mu)+f(\omega))
		\\
		&\hspace{2.2cm}
		- f(\omega)f(\mu)(f(\eta)+f(\omega+\mu-\eta))
		\Big]
				\\[0.6em]
		&\; + C_2 \int_{0}^{\omega}\,\mathrm{d}\mu
		\frac{|k|(\mu)\,|k|(\omega-\mu)}{|k|(\omega)}
		|k|'(\mu)\,|k|'(\omega-\mu)\,
		(\omega\mu(\omega-\mu))^\gamma
		f(\mu)(\omega-\mu)
		\\[0.4em]
		&\; - 2C_2 \int_{0}^{\infty}\,\mathrm{d}\mu
		\frac{|k|(\mu)\,|k|(\omega+\mu)}{|k|(\omega)}
		|k|'(\mu)\,|k|'(\omega+\mu)\,
		(\omega\mu(\omega+\mu))^\gamma
		f(\omega)f(\mu)
		\\[0.4em]
		&\; - 2C_2 \int_{0}^{\omega}\,\mathrm{d}\mu
		\frac{|k|(\mu)\,|k|(\omega-\mu)}{|k|(\omega)}
		|k|'(\mu)\,|k|'(\omega-\mu)\,
		(\omega\mu(\omega-\mu))^\gamma
		f(\omega)\mu
		\\[0.4em]
		&\; + 2C_2 \int_{\omega}^{\infty}\,\mathrm{d}\mu
		\frac{|k|(\mu)\,|k|(\mu-\omega)}{|k|(\omega)}
		|k|'(\mu)\,|k|'(\mu-\omega)\,
		(\omega\mu(\mu-\omega))^\gamma
		\big[f(\omega)f(\mu)+f(\mu)f(\mu-\omega)\big]
		.
	\end{split}
\end{align}

Now substitute $\eta = \omega + \mu - \bar{\eta}$ in the third and sixth integrals, and then replace $\bar{\eta}$ by $\eta$.  
\allowdisplaybreaks
\begin{align}
	\partial_t f
	=&\; \frac{2C_1}{|k|(\omega)}
	\int_{0}^{\omega} \int_{0}^{\mu}
	\,\mathrm{d}\eta\,\mathrm{d}\mu
	|k|'(\mu)\,|k|'(\eta)\,|k|'(\omega+\mu-\eta)\,
	|k|(\mu)\,|k|^2(\eta)\,|k|(\omega+\mu-\eta)
	\notag\\
	&\hspace{2.2cm}\times(\omega\mu\eta(\omega+\mu-\eta))^{\sigma}
	\Big[
	f(\eta)f(\omega+\mu-\eta)(f(\mu)+f(\omega))
	\notag\\
	&\left.\hspace{2.2cm}
	- f(\omega)f(\mu)(f(\eta)+f(\omega+\mu-\eta))
	\right]
	\notag\\[0.5em]
	&+ \frac{C_1}{|k|(\omega)}
	\int_{0}^{\omega} \int_{\mu}^{\omega}
	\,\mathrm{d}\eta\,\mathrm{d}\mu
	|k|'(\mu)\,|k|'(\eta)\,|k|'(\omega+\mu-\eta)\,
	|k|^2(\mu)\,|k|(\eta)\,|k|(\omega+\mu-\eta)
	\notag\\
	&\hspace{2.2cm}\times(\omega\mu\eta(\omega+\mu-\eta))^{\sigma}
	\Big[
	f(\eta)f(\omega+\mu-\eta)(f(\mu)+f(\omega))
	\notag\\
	&\left.\hspace{2.2cm}
	- f(\omega)f(\mu)(f(\eta)+f(\omega+\mu-\eta))
	\right]
	\notag\\[0.5em]
	&+\frac{2C_1}{|k|(\omega)}
	\int_{\omega}^{\infty}\int_{0}^{\omega}
	\,\mathrm{d}\eta\,\mathrm{d}\mu
	|k|'(\mu)\,|k|'(\eta)\,|k|'(\omega+\mu-\eta)\,
	|k|(\mu)\,|k|^2(\eta)\,|k|(\omega+\mu-\eta)
	\notag\\
	&\hspace{2.2cm}\times(\omega\mu\eta(\omega+\mu-\eta))^{\sigma}
	\Big[
	f(\eta)f(\omega+\mu-\eta)(f(\mu)+f(\omega))
	\notag\\
	&\left.\hspace{2.2cm}
	- f(\omega)f(\mu)(f(\eta)+f(\omega+\mu-\eta))
	\right]
	\notag\\[0.5em]
	&+\frac{C_1}{|k|(\omega)}
	\int_{\omega}^{\infty}\int_{\omega}^{\mu}\,\mathrm{d}\eta\,\mathrm{d}\mu
	|k|'(\mu)\,|k|'(\eta)\,|k|'(\omega+\mu-\eta)\,
	|k|(\omega)\,|k|(\mu)\,|k|(\eta)\,|k|(\omega+\mu-\eta)
	\notag\\
	&\hspace{2.2cm}\times(\omega\mu\eta(\omega+\mu-\eta))^{\sigma}
	\Big[
	f(\eta)f(\omega+\mu-\eta)(f(\mu)+f(\omega))
	\notag\\
	&\left.\hspace{2.2cm}
	- f(\omega)f(\mu)(f(\eta)+f(\omega+\mu-\eta))
	\right]
		\notag\\[0.5em]
	&+ C_2 \int_{0}^{\omega}\,\mathrm{d}\mu
	\frac{|k|(\mu)\,|k|(\omega-\mu)}{|k|(\omega)}
	|k|'(\mu)\,|k|'(\omega-\mu)\,
	(\omega\mu(\omega-\mu))^\gamma\,
	f(\mu)\,(\omega-\mu)
	\notag\\
	&- 2C_2 \int_{0}^{\infty}\,\mathrm{d}\mu
	\frac{|k|(\mu)\,|k|(\omega+\mu)}{|k|(\omega)}
	|k|'(\mu)\,|k|'(\omega+\mu)\,
	(\omega\mu(\omega+\mu))^\gamma\,
	f(\omega)f(\mu)
	\notag\\
	&- 2C_2 \int_{0}^{\omega}\,\mathrm{d}\mu
	\frac{|k|(\mu)\,|k|(\omega-\mu)}{|k|(\omega)}
	|k|'(\mu)\,|k|'(\omega-\mu)\,
	(\omega\mu(\omega-\mu))^\gamma\,
	f(\omega)\,\mu
	\notag\\
	&+ 2C_2 \int_{\omega}^{\infty}\,\mathrm{d}\mu
	\frac{|k|(\mu)\,|k|(\mu-\omega)}{|k|(\omega)}
	|k|'(\mu)\,|k|'(\mu-\omega)\,
	(\omega\mu(\mu-\omega))^\gamma\,
	\big[f(\omega)f(\mu)+f(\mu)f(\mu-\omega)\big]
	.
	\notag
\end{align}

The above equation can be rewritten as
\begin{align}\label{1_6}
	\begin{split}
		\partial_t f
		=&\int_{0}^{\omega} \int_{\mu}^{\omega}
		\,\mathrm{d}\eta\,\mathrm{d}\mu
		\mathcal{K}_1(\omega,\mu,\eta)
		\Big[
		f(\eta)f(\omega+\mu-\eta)\big(f(\mu)+f(\omega)\big)
		\\
		&\hspace{2.2cm}
		- f(\omega)f(\mu)\big(f(\eta)+f(\omega+\mu-\eta)\big)
		\Big]
		\\[0.5em]
		&+\int_{0}^{\infty} \int_{0}^{\omega}
		\,\mathrm{d}\eta\,\mathrm{d}\mu
		\mathcal{K}_2(\omega,\mu,\eta)
		\Big[
		f(\eta)f(\omega+\mu-\eta)\big(f(\mu)+f(\omega)\big)
		\\
		&\hspace{2.2cm}
		- f(\omega)f(\mu)\big(f(\eta)+f(\omega+\mu-\eta)\big)
		\Big]
		\\[0.5em]
		&+\int_{\omega}^{\infty} \int_{\omega}^{\mu}
		\,\mathrm{d}\eta\,\mathrm{d}\mu
		\mathcal{K}_3(\omega,\mu,\eta)
		\Big[
		f(\eta)f(\omega+\mu-\eta)\big(f(\mu)+f(\omega)\big)
		\\
		&\hspace{2.2cm}
		- f(\omega)f(\mu)\big(f(\eta)+f(\omega+\mu-\eta)\big)
		\Big]
		\\[0.5em]
		&+\int_{0}^{\omega}\,\mathrm{d}\mu
		\mathcal{K}_4(\omega,\mu)\, f(\mu)\,(\omega-\mu)
		\;-\; \int_{0}^{\infty}\,\mathrm{d}\mu
		\mathcal{K}_5(\omega,\mu)\, f(\omega)f(\mu)
		\\[0.5em]
		&-\int_{0}^{\omega}\,\mathrm{d}\mu
		\mathcal{K}_6(\omega,\mu)\, f(\omega)\,\mu
		+\int_{\omega}^{\infty}	\,\mathrm{d}\mu 
		\mathcal{K}_7(\omega,\mu)\big[f(\omega)f(\mu)+f(\mu)f(\mu-\omega)\big]
	.
	\end{split}
\end{align}

where
\begin{align*}
	\mathcal{K}_1(\omega,\mu,\eta)
	&=
	\frac{C_1}{|k|(\omega)}
	|k|'(\mu)|k|'(\eta)|k|'(\omega+\mu-\eta)
	\,|k|(\mu)|k|(\eta)|k|(\omega+\mu-\eta)
	\\[-0.2em]
	&\qquad \times (|k|(\mu)-2|k|(\eta))\,
	(\omega\mu\eta(\omega+\mu-\eta))^{\sigma},
	\\[0.5em]
	\mathcal{K}_2(\omega,\mu,\eta)
	&=
	\frac{2C_1}{|k|(\omega)}
	|k|'(\mu)|k|'(\eta)|k|'(\omega+\mu-\eta)
	\,|k|(\mu)|k|^2(\eta)|k|(\omega+\mu-\eta)
	\\[-0.2em]
	&\qquad \times (\omega\mu\eta(\omega+\mu-\eta))^{\sigma},
	\\[0.5em]
	\mathcal{K}_3(\omega,\mu,\eta)
	&=
	C_1
	|k|'(\mu)|k|'(\eta)|k|'(\omega+\mu-\eta)
	\,|k|(\mu)|k|(\eta)|k|(\omega+\mu-\eta)
	\\[-0.2em]
	&\qquad \times (\omega\mu\eta(\omega+\mu-\eta))^{\sigma},
	\\[0.5em]
	\mathcal{K}_4(\omega,\mu)
	&=
	\frac{C_2|k|(\mu)|k|(\omega-\mu)}{|k|(\omega)}
	|k|'(\mu)|k|'(\omega-\mu)
	(\omega\mu(\omega-\mu))^{\gamma},
	\\[0.5em]
	\mathcal{K}_5(\omega,\mu)
	&=
	\frac{2C_2|k|(\mu)|k|(\omega+\mu)}{|k|(\omega)}
	|k|'(\mu)|k|'(\omega+\mu)
	(\omega\mu(\omega+\mu))^{\gamma},
	\\[0.5em]
	\mathcal{K}_6(\omega,\mu)
	&=
	\frac{2C_2|k|(\mu)|k|(\omega-\mu)}{|k|(\omega)}
	|k|'(\mu)|k|'(\omega-\mu)
	(\omega\mu(\omega-\mu))^{\gamma},
	\\[0.5em]
	\mathcal{K}_7(\omega,\mu)
	&=
	\frac{2C_2|k|(\mu)|k|(\mu-\omega)}{|k|(\omega)}
	|k|'(\mu)|k|'(\mu-\omega)
	(\omega\mu(\mu-\omega))^{\gamma}.
\end{align*}
And the initial condition is now $f^{\mathrm{in}}(\omega)$.
	
	\section{Formulation of the numerical scheme}
The goal of this section is to construct the numerical scheme \eqref{2_12} for \eqref{1_6}.  To develop an efficient finite-volume scheme, the computational domain is truncated to 
\(\mathcal{D} := (0, R]\), where \(R < \infty\). The mixed wave equation \eqref{1_6} is then written in the following form:
\[
\begin{aligned}
	\partial_t f ={}& \int_{0}^{\omega} \int_{\mu}^{\omega}\,\mathrm{d}\eta\,\mathrm{d}\mu 
	\mathcal{K}_1(\omega,\mu,\eta)\Big[
	f(\eta) f(\omega + \mu - \eta)\big(f(\mu)+f(\omega)\big) \\
	&\qquad\qquad\qquad\qquad 
	- f(\omega) f(\mu)\big(f(\eta)+f(\omega + \mu - \eta)\big)
	\Big] \\
	&+ \int_{0}^{R-\omega} \int_{0}^{\omega} \,\mathrm{d}\eta\,\mathrm{d}\mu 
	\mathcal{K}_2(\omega,\mu,\eta)\Big[
	f(\eta) f(\omega + \mu - \eta)\big(f(\mu)+f(\omega)\big) \\
	&\qquad\qquad\qquad\qquad
	- f(\omega) f(\mu)\big(f(\eta)+f(\omega + \mu - \eta)\big)
	\Big]\\
	&+ \int_{\omega}^{R} \int_{\omega}^{\mu}\,\mathrm{d}\eta\,\mathrm{d}\mu
	\mathcal{K}_3(\omega,\mu,\eta)\Big[
	f(\eta) f(\omega + \mu - \eta)\big(f(\mu)+f(\omega)\big) \\
	&\qquad\qquad\qquad\qquad
	- f(\omega) f(\mu)\big(f(\eta)+f(\omega + \mu - \eta)\big)
	\Big] \\
	&+ \int_{0}^{\omega}\, \mathrm{d}\mu \mathcal{K}_4(\omega,\mu)\, f(\mu) f(\omega-\mu)
	- \int_{0}^{R} \mathcal{K}_5(\omega,\mu)\, f(\omega) f(\mu)\, \mathrm{d}\mu \\
	&- \int_{0}^{\omega}\, \mathrm{d}\mu \mathcal{K}_6(\omega,\mu)\, f(\omega) f(\mu)
	+ \int_{\omega}^{R} \, \mathrm{d}\mu\mathcal{K}_7(\omega,\mu)\, f(\omega) f(\mu) \\
	&+ \int_{\omega}^{R}\, \mathrm{d}\mu  \mathcal{K}_7(\omega,\mu)\, f(\mu) f(\mu-\omega).
\end{aligned}
\]

It is worth noting that taking the limit \(R \to \infty\) in the above truncation recovers the original explicit form of the mixed wave equation \eqref{1_6}, provided that the wave kernels are chosen appropriately. We refer the reader to the work of Stewart \cite{stewart1989global}, where this type of integral convergence has been rigorously established within the framework of coagulation–fragmentation equations.

To formulate the semi-discrete numerical scheme, the computational domain is divided into \(I\) subintervals, denoted by \(\Lambda_i := [\omega_{i-1/2}, \omega_{i+1/2}]\). The representative point of the \(i\)-th cell is defined as
\[
\omega_{1/2} = 0,\quad \omega_{I+1/2} = R,\quad 
\omega_i = \frac{\omega_{i-1/2} + \omega_{i+1/2}}{2}.
\]
The representative point of each cell is sometimes referred to as a \emph{grid point} or \emph{pivot}. The width of the \(i\)-th cell is
\[
\Delta\omega_i = \omega_{i+1/2} - \omega_{i-1/2}.
\]
For the convergence analysis of the proposed scheme, we assume that
\[
\Delta\omega_{\min} \le \Delta\omega_i \le \Delta\omega.
\]

Integrating with respect to \(\omega\) over the subinterval 
\([\omega_{i-1/2}, \omega_{i+1/2}]\) yields the semi-discrete scheme in 
\(\mathbb{R}^I\):
\begin{align}\label{2_1}
	\frac{\mathrm{d} \mathbf{N}}{\mathrm{d} t} = 
	\sum_{k=1}^{8} \mathbf{F}^k(\mathbf{N}), 
	\qquad 
	\mathbf{N}(0) = \mathbf{N}^{\mathrm{in}}.
\end{align}
The \(i\)-th components of the vectors 
\(\mathbf{N} \in \mathbb{R}^I\), 
\(\mathbf{N}^{\mathrm{in}} \in \mathbb{R}^I\), 
and \(\mathbf{F}^k \in \mathbb{R}^I\) \((k=1,\dots,8)\) are given by
\begin{align}
	N_i(t) &= \int_{\omega_{i-1/2}}^{\omega_{i+1/2}} \,\mathrm{d}\omega
	f(\omega,t),
	\qquad
	N^{\mathrm{in}}_i = \int_{\omega_{i-1/2}}^{\omega_{i+1/2}} \,\mathrm{d}\omega
	f^{\mathrm{in}}(\omega), \label{2_2} \\[0.3em]
	F^1_i(t) &= \int_{\omega_{i-1/2}}^{\omega_{i+1/2}}
	\int_{0}^{\omega} \int_{\mu}^{\omega}
	\,\mathrm{d}\eta\,\mathrm{d}\mu\,\mathrm{d}\omega
	\mathcal{K}_1(\omega,\mu,\eta)
	\Big[
	f(\eta) f(\omega+\mu-\eta)\big(f(\mu)+f(\omega)\big)
	\notag\\
	&\hspace{4cm}
	- f(\omega)f(\mu)\big(f(\eta)+f(\omega+\mu-\eta)\big)
	\Big], \label{2_3} \\[0.3em]
	F^2_i(t) &= \int_{\omega_{i-1/2}}^{\omega_{i+1/2}}
	\int_{\omega}^{R} \int_{0}^{\omega}
	\,\mathrm{d}\eta\,\mathrm{d}\mu\,\mathrm{d}\omega
	\mathcal{K}_2(\omega,\mu-\omega,\eta)
	\Big[
	f(\eta) f(\mu-\eta)\big(f(\mu-\omega)+f(\omega)\big)
	\notag\\
	&\hspace{4cm}
	- f(\omega)f(\mu-\omega)\big(f(\eta)+f(\mu-\eta)\big)
	\Big], \label{2_4} \\[0.3em]
	F^3_i(t) &= \int_{\omega_{i-1/2}}^{\omega_{i+1/2}}
	\int_{\omega}^{R} \int_{\omega}^{\mu}\,\mathrm{d}\eta\,\mathrm{d}\mu\,\mathrm{d}\omega
	\mathcal{K}_3(\omega,\mu,\eta)
	\Big[
	f(\eta) f(\omega+\mu-\eta)\big(f(\mu)+f(\omega)\big)
	\notag\\
	&\hspace{4cm}
	- f(\omega)f(\mu)\big(f(\eta)+f(\omega+\mu-\eta)\big)
	\Big]
	, \label{2_5} \\[0.3em]
	F^4_i(t) &= \int_{\omega_{i-1/2}}^{\omega_{i+1/2}}
	\int_{0}^{\omega}\,\mathrm{d}\mu\,\mathrm{d}\omega \mathcal{K}_4(\omega,\mu)\,
	f(\mu) f(\omega-\mu), \label{2_6} \\[0.3em]
	F^5_i(t) &= \int_{\omega_{i-1/2}}^{\omega_{i+1/2}}
	\int_{0}^{R}\,\mathrm{d}\mu\,\mathrm{d}\omega \mathcal{K}_5(\omega,\mu)\,
	f(\omega) f(\mu), \label{2_7} \\[0.3em]
	F^6_i(t) &= \int_{\omega_{i-1/2}}^{\omega_{i+1/2}}
	\int_{0}^{\omega}\,\mathrm{d}\mu\,\mathrm{d}\omega \mathcal{K}_6(\omega,\mu)\,
	f(\omega) f(\mu), \label{2_8} \\[0.3em]
	F^7_i(t) &= \int_{\omega_{i-1/2}}^{\omega_{i+1/2}}
	\int_{\omega}^{R} \,\mathrm{d}\mu\,\mathrm{d}\omega\mathcal{K}_7(\omega,\mu)\,
	f(\omega) f(\mu), \label{2_9} \\[0.3em]
	F^8_i(t) &= \int_{\omega_{i-1/2}}^{\omega_{i+1/2}}
	\int_{\omega}^{R}\,\mathrm{d}\mu\,\mathrm{d}\omega \mathcal{K}_7(\omega,\mu)\,
	f(\mu) f(\mu-\omega). \label{2_10}
\end{align}

The wave density function \(f(t,\omega)\) can be approximated in terms of the
numerical wave densities \(N_i(t)\) as
\begin{align}\label{2_11}
	f(t,\omega) \approx \sum_{i=1}^{I} N_i(t)\,(\omega - \omega_i).
\end{align}

Discretizing the frequency variable \(\omega\) leads to a discrete form of the
wave–kinetic kernels, which can be written as
\begin{align}
	\begin{aligned}\label{kernel_dis}
		&\mathcal{K}_k(\omega,\mu,\eta)
		\approx
		\mathcal{K}^k(\omega_i,\omega_j,\omega_k)
		= \mathcal{K}^k_{i,j,k},
		\qquad k = 1,2,3,4, \\[0.3em]
		&\mathcal{K}_k(\omega,\mu)
		\approx
		\mathcal{K}^k(\omega_i,\omega_j)
		= \mathcal{K}^k_{i,j},
		\qquad k = 4,5,6,7,
	\end{aligned}
\end{align}
whenever \(\omega \in \Lambda_i\), \(\mu \in \Lambda_j\), and \(\eta \in \Lambda_k\).

To formulate the finite volume scheme for the mixed wave kinetic equation, 
we first define the required set of indices:
\begin{align*}
	&\mathcal{I}^i(j,k) := 
	\{(j,k)\in \mathbb{N}\times\mathbb{N}:
	\omega_{i-1/2}\le \omega_j+\omega_k< \omega_{i+1/2}\}, \\[0.3em]
	&\mathcal{J}^i(j,k) := 
	\{(j,k)\in \mathbb{N}\times\mathbb{N}:
	\omega_{i-1/2}\le \omega_j-\omega_k< \omega_{i+1/2}\}, \\[0.3em]
	&\tilde{\theta}^i_{j}(k) := 
	\{k \in \mathbb{N}:
	\omega_{i-1/2}\le \omega_j+\omega_k< \omega_{i+1/2},
	~\text{for some given } (i,j)\in \mathbb{N}\times\mathbb{N}\}, \\[0.3em]
	&\bar{\theta}^i_{jk}(l) := 
	\{l \in \mathbb{N}:
	\omega_{i-1/2}\le \omega_l+\omega_j-\omega_k< \omega_{i+1/2},
	~\text{for some given } (j,k)\in \mathbb{N}\times\mathbb{N}\}, \\[0.3em]
	&\hat{\theta}^i_{jk}(l) := 
	\{l \in \mathbb{N}:
	\omega_{i-1/2}\le \omega_j+\omega_k-\omega_l< \omega_{i+1/2},
	~\text{for some given } (j,k)\in \mathbb{N}\times\mathbb{N}\}.
\end{align*}

Using the previously defined index notation and the approximation \eqref{2_11},
the numerical finite volume scheme (FVS) can be expressed as
\begin{align}\label{2_12}
	\begin{split}
		\allowdisplaybreaks
		\frac{\mathrm{d} N_i}{\mathrm{d} t} &=  
		\sum_{j=1}^{i-2} \sum_{k=j+1}^{i-1} 
		\sum_{\bar{\theta}^i_{kj}(l)}
		\mathcal{K}^1_{l+k-j,j,k} N_j N_k N_l
		+ \sum_{j=2}^{i-1}\sum_{k=1}^{j-1} 
		\sum_{\bar{\theta}^k_{ji}(l)} 
		\mathcal{K}^1_{i,l+j-i,j} N_i N_j N_l \\
		&\quad - \sum_{j=1}^{i-2} \sum_{k=j+1}^{i-1}  
		\mathcal{K}^1_{i,j,k} N_i N_j N_k
		- \sum_{j=1}^{i-2} \sum_{k=j+1}^{i-1} 
		\sum_{\hat{\theta}^k_{ij}(l)} 
		\mathcal{K}^1_{i,j,i+j-l} N_i N_j N_l \\
		&\quad + \sum_{j=i+1}^{I} \sum_{k=1}^{i-1} 
		\sum_{\tilde{\theta}^j_k(l)} 
		\sum_{\hat{\theta}^i_{lk}(m)} 
		\mathcal{K}^2_{l+k-m,m,k} N_k N_l N_m \\
		&\quad + \sum_{j=i+1}^{I} \sum_{k=1}^{i-1} 
		\sum_{\tilde{\theta}^j_k(l)} 
		\mathcal{K}^2_{i,l+k-i,i} N_i N_k N_l \\
		&\quad - \sum_{j=i+1}^{I} \sum_{k=1}^{i-1} 
		\sum_{\tilde{\theta}^j_i(l)} 
		\mathcal{K}^2_{i,l,k} N_i N_k N_l \\
		&\quad - \sum_{j=i+1}^{I} \sum_{k=1}^{i-1} 
		\sum_{\tilde{\theta}^j_i(l)} 
		\sum_{\hat{\theta}^k_{li}(m)}  
		\mathcal{K}^2_{i,l,i+l-m} N_i N_l N_m 	\end{split}
	\end{align}
\begin{align*}
	\begin{split}
		&\quad + \sum_{j=i+2}^{I} \sum_{k=i+1}^{j-1} 
		\sum_{\bar{\theta}^i_{kj}(l)}  
		\mathcal{K}^3_{l+k-j,j,k} N_j N_k N_l \\
		&\quad + \sum_{j=i+1}^{I} \sum_{k=j+1}^{I} 
		\sum_{\bar{\theta}^k_{ji}(m)} 
		\mathcal{K}^3_{i,m+j-i,j} N_i N_j N_m \\
		&\quad - \sum_{j=i+2}^{I} \sum_{k=i+1}^{j-1}  
		\mathcal{K}^3_{i,j,k} N_i N_j N_k \\
		&\quad - \sum_{j=i+2}^{I} \sum_{k=i+1}^{j-1} 
		\sum_{\hat{\theta}^k_{ij}(l)} 
		\mathcal{K}^3_{i,j,i+j-l} N_i N_j N_l \\
		&\quad + \sum_{\mathcal{I}^i(j,k)} 
		\mathcal{K}^4_{k+j,j} N_j N_k
		- \sum_{j=1}^{I} \mathcal{K}^5_{i,j} N_i N_j 
		- \sum_{j=1}^{i-1} \mathcal{K}^6_{i,j} N_i N_j \\
		&\quad + \sum_{j=i+1}^{I} \mathcal{K}^7_{i,j} N_i N_j
		+ \sum_{\mathcal{J}^i(j,k)} \mathcal{K}^7_{j-k,j} N_j N_k.
	\end{split}
\end{align*}

	\section{Convergence analysis of the numerical scheme}
The numerical scheme \eqref{2_12} can first be reformulated in vector form as
\begin{align}\label{3_1}
	\frac{\mathrm{d}\hat{\mathbf{N}}}{\mathrm{d} t} = \hat{\mathbf{J}}(\hat{\mathbf{N}}).
\end{align}
Here, $\hat{\mathbf{N}} := \{\hat{N}_1, \hat{N}_2, \dots, \hat{N}_I\}$ represents the
approximate vector solution that satisfies
\begin{align}\label{3_2}
	f(t,\omega)
	= \sum_{i=1}^{I} N_i(t)\,\delta(\omega - \omega_i)
	+ \mathcal{O}(\Delta\omega^{3}).
\end{align}
	 Similarly, let $\ds\mathbf{ \hat{J}} := \{\hat{J}_1, \hat{J}_2, \cdot\cdot\cdot, \hat{J}_I\}$ be a vector in $\mathbb{R}^I$, whose $i-$th component is defined by
	\begin{align}\label{3_2.1}
		\hat{J}_i\left(\mathbf{\hat{N}}\right) = \sum_{k=1}^{17} \mathcal{Q}_i^k \left(\mathbf{\hat{N}}\right),
	\end{align}
	where
	\begin{align}\label{3_3}
		\begin{split}
			& \mathcal{Q}_i^1 \left(\mathbf{\hat{N}}\right):= \sum_{j=1}^{i-2} \sum_{k=j+1}^{i-1} \sum_{ \bar{\theta}^i_{kj}(l)}\mathcal{K}^1_{l+k-j,j,k} N_j N_k N_l, \quad \mathcal{Q}_i^2 \left(\mathbf{\hat{N}}\right):=\sum_{j=2}^{i-1}\sum_{k=1}^{j-1} \sum_{\bar{\theta}^k_{ji}(l)} \mathcal{K}^1_{i,l+j-i,j} N_i N_j N_l,\\
			&\mathcal{Q}_i^3 \left(\mathbf{\hat{N}}\right):= - \sum_{j=1}^{i-2}\sum_{k=j+1}^{i-1}  \mathcal{K}^1_{i,j,k} N_i N_j N_k, \quad \mathcal{Q}_i^4 \left(\mathbf{\hat{N}}\right):=  - \sum_{j=1}^{i-2} \sum_{k=j+1}^{i-1} \sum_{\hat{\theta}^k_{ij}(l)} \mathcal{K}^1_{i,j,i+j-l}N_i N_j N_l\\
			&  \mathcal{Q}_i^5 \left(\mathbf{\hat{N}}\right):= \sum_{j=i+1}^{I} \sum_{k=1}^{i-1} \sum_{\tilde{\theta}^j_k(l)} \sum_{\hat{\theta}^i_{lk}(m)} \mathcal{K}^2_{l+k-m,m,k} N_k N_l N_m \quad  \mathcal{Q}_i^6 \left(\mathbf{\hat{N}}\right):=  \sum_{j=i+1}^{I} \sum_{k=1}^{i-1} \sum_{\tilde{\theta}^j_k(l)} \mathcal{K}^2_{i,l+k-i,i} N_i N_k N_l,\\
			&\mathcal{Q}_i^7 \left(\mathbf{\hat{N}}\right):= - \sum_{j=i+1}^{I} \sum_{k=1}^{i-1} \sum_{\tilde{\theta}^j_i(l)} \mathcal{K}^2_{i,l,k} N_i N_k N_l, \quad \mathcal{Q}_i^8 \left(\mathbf{\hat{N}}\right):= -  \sum_{j=i+1}^{I} \sum_{k=1}^{i-1} \sum_{\tilde{\theta}^j_i(l)}  \sum_{\hat{\theta}^k_{li}(m)}  \mathcal{K}^2_{i,l,i+l-m} N_i N_l N_m,\\
			&   \mathcal{Q}_i^9 \left(\mathbf{\hat{N}}\right):= \sum_{j=i+2}^{I} \sum_{k=i+1}^{j-1} \sum_{\bar{\theta}^i_{kj}(l)}  \mathcal{K}^3_{l+k-j,j,k}N_j N_k N_l, \quad  \mathcal{Q}_i^{10} \left(\mathbf{\hat{N}}\right):= \sum_{j=i+1}^{I} \sum_{k=j+1}^{I} \sum_{\bar{\theta}^k_{ji}(m)} \mathcal{K}^3_{i,m+j-i,j} N_i N_j N_m,\\
			& \mathcal{Q}_i^{11} \left(\mathbf{\hat{N}}\right):= - \sum_{j=i+2}^{I} \sum_{k=i+1}^{j-1}  \mathcal{K}^3_{i,j,k}N_i N_j N_k,\quad\mathcal{Q}_i^{12} \left(\mathbf{\hat{N}}\right):= -  \sum_{j=i+2}^{I} \sum_{k=i+1}^{j-1} \sum_{\hat{\theta}^k_{ij}(l)} \mathcal{K}^3_{i,j,i+j-l}N_i N_j N_l,\\
			&\mathcal{Q}_i^{13} \left(\mathbf{\hat{N}}\right):= \sum_{ \mathcal{I}_{j,k}^i\left(j,k\right)} \mathcal{K}^4_{k+j,j}N_j N_k, \quad \mathcal{Q}_i^{14} \left(\mathbf{\hat{N}}\right):= -  \sum_{j=1}^{I} \mathcal{K}^5_{i,j}N_i N_j,\quad\mathcal{Q}_i^{15} \left(\mathbf{\hat{N}}\right):= - \sum_{j=1}^{i-1} \mathcal{K}^6_{i,j}N_i N_j \\
			& \mathcal{Q}_i^{16} \left(\mathbf{\hat{N}}\right):= \sum_{j=i+1}^{I} \mathcal{K}^7_{i,j} N_i N_j \quad \mathcal{Q}_i^{17} \left(\mathbf{\hat{N}}\right):=  \sum_{ \mathcal{J}_{j,k}^i\left(j,k\right)} \mathcal{K}^7_{j-k,j} N_j N_k.
		\end{split}
	\end{align}
	
For the convergence analysis of the proposed scheme, the discrete $L^1$ norm on 
$\mathbb{R}^I$ is defined by
\begin{align}\label{3_4}
	\|N(t)\| = \sum_{i=1}^{I} |N_i(t)| .
\end{align}

Moreover, for the purpose of the convergence analysis, we assume that the kinetic
kernels satisfy
\begin{align}\label{3_5}
	\mathcal{K}_i(\omega,\mu,\eta) \in \mathcal{C}^2([0,R]^3)
	\quad \text{for } i = 1,2,3,
	\qquad\text{and}\qquad
	\mathcal{K}_i(\omega,\mu) \in \mathcal{C}^2([0,R]^2)
	\quad \text{for } i = 4,\dots,7.
\end{align}

for any given function $|k|$ with a suitable choice of $\sigma$ and $\gamma$.

In order to prove the convergence of the proposed numerical scheme, we make use 
of several definitions and results from \cite{hundsdorfer2013numerical,linz1975convergence}.

	\begin{defn}\label{def_1}
	The spatial discretization error is defined as the residual obtained by substituting the exact solution 
	\(\mathbf{N} := \{N_1, N_2, \cdots, N_I\}\) into the discrete system:
	\begin{align}\label{3_6}
		\epsilon(t) = \frac{\dd \mathbf{N}}{\dd t} - \hat{\mathbf{J}}(\mathbf{N}).
	\end{align}
	
	We say that the numerical scheme \eqref{3_1} is $p$th--order convergent if, as $\Delta \omega \to 0$,
	\begin{align}\label{3_7}
		\|\epsilon(t)\| = \mathcal{O}(\Delta \omega^p),
		\qquad \text{uniformly for all } 0 \le t \le T.
	\end{align}
	
	\end{defn}
	
We begin the convergence analysis by proving that the mapping $\mathbf{\hat{J}}$ satisfies a Lipschitz condition with a mesh-independent Lipschitz constant. To this end, we first establish the following proposition, which provides a preliminary estimate of the norm in $L^1(0,T)$.

	\begin{Proposition}\label{Prop_1}
	Assume that the wave kinetic kernels $\mathcal{K}_i$ satisfy the condition \eqref{3_5}. 
	Then there exists a positive constant $\mathcal{L}(T)$, independent of the mesh, such that 
	\[
	\|\mathbf{N}(t)\| \le \mathcal{L}(T) \qquad \text{for all } 0 \le t \le T.
	\]
	
	\end{Proposition} 
	
	\begin{proof}
	By summing equation \eqref{3_1} with respect to $i$ from $1$ to $I$, we obtain
	\begin{align}\label{3_8}
		\frac{\dd \|\mathbf{N}(t)\|}{\dd t}
		= \sum_{i=1}^{I} \sum_{k=1}^{17} \mathcal{Q}^k_i(N_i).
	\end{align}
	
	The first term on the right-hand side of the preceding expression can be estimated as follows:
	\begin{align}\label{3_9}
		\begin{split}
			\sum_{i=1}^{I} \mathcal{Q}^1_i(N_i)
			&= \sum_{i=1}^{I}\sum_{j=1}^{i-2} \sum_{k=j+1}^{i-1} 
			\sum_{\bar{\theta}^i_{kj}(l)}
			\mathcal{K}^1_{l+k-j,j,k} N_j N_k N_l \\
			&= \sum_{i=1}^{I}\sum_{j=1}^{i-2} \sum_{k=j+1}^{i-1}
			\mathcal{K}^1_{i+k-j,j,k} N_j N_k N_i.
		\end{split}
	\end{align}
	
	Since each $\mathcal{K}_i \in \mathcal{C}^2$ for $i = 1,\dots,7$, one can find constants $\mathcal{A}_i$ such that
	\begin{align*}
		\sup_{(\omega,\mu,\eta)\in[0,R]^3} \mathcal{K}_i(\omega,\mu,\eta)
		\le \mathcal{A}_i \quad \text{for } i = 1,2,3,
	\end{align*}
	and
	\begin{align*}
		\sup_{(\omega,\mu)\in[0,R]^2} \mathcal{K}_i(\omega,\mu)
		\le \mathcal{A}_i \quad \text{for } i = 4,\dots,7.
	\end{align*}
	
	Applying these bounds together with the kernel discretization \eqref{kernel_dis} to equation \eqref{3_9} yields
	\begin{align*}
		\sum_{i=1}^{I} \mathcal{Q}^1_i(N_i)
		\le \mathcal{A}_1 \|\mathbf{N}(t)\|^3.
	\end{align*}
	
	Proceeding similarly for all remaining terms in equation \eqref{3_8} and substituting the resulting bounds into \eqref{3_8}, we obtain
	\begin{align*}
		\frac{\dd \|\mathbf{N}(t)\|}{\dd t}
		\le \mathcal{A}\,\|\mathbf{N}(t)\|^3
		\qquad \text{for all } 0 \le t \le T,
	\end{align*}
	for some constant $\mathcal{A}$. Consequently, this differential inequality provides the required estimate.
	
	\end{proof}
	
	\begin{Proposition}\label{Prop_2}
	If the wave kinetic kernels $\mathcal{K}_i$ satisfy condition \eqref{3_5}, 
	then there exists a positive constant $\zeta$, independent of the mesh, such that  
	\begin{align}\label{3_10}
		\|\mathbf{\hat{J}}(\mathbf{N}) - \mathbf{\hat{J}}(\mathbf{\hat{N}})\|
		\le \zeta\,\|\mathbf{N} - \mathbf{\hat{N}}\|,
		\qquad \text{for all } \mathbf{N}, \mathbf{\hat{N}} \in \mathbb{R}^I.
	\end{align}
	
	\end{Proposition}
	
	\begin{proof}
	The proof of the proposition reduces to verifying that each term 
	$\mathcal{Q}^k$ $(k = 1,\dots,17)$ satisfies a Lipschitz condition. 
	Consider $\mathbf{N}, \mathbf{\hat{N}} \in \mathbb{R}^I$. Then
	\begin{align}\label{3_11}
		\begin{split}
			\|\mathcal{Q}^1(\mathbf{N}) - \mathcal{Q}^1(\mathbf{\hat{N}})\|
			&= \sum_{i=1}^{I} 
			\left|\mathcal{Q}^1_i(\mathbf{N}) - \mathcal{Q}^1_i(\mathbf{\hat{N}})\right| \\
			&= \sum_{i=1}^{I}\sum_{j=1}^{i-2} \sum_{k=j+1}^{i-1}
			\mathcal{K}^1_{i+k-j,j,k}\,
			\left|N_i N_j N_k - \hat{N}_i \hat{N}_j \hat{N}_k\right|.
		\end{split}
	\end{align}
	
	Observe that
	\[
	\left|N_i N_j N_k - \hat{N}_i \hat{N}_j \hat{N}_k\right|
	= \left| 
	N_i N_j (N_k - \hat{N}_k)
	+ N_i \hat{N}_k (N_j - \hat{N}_j)
	+ \hat{N}_j \hat{N}_k (N_i - \hat{N}_i)
	\right|.
	\]
	
	Substituting this identity into equation \eqref{3_11}, we obtain
	\[
	\|\mathcal{Q}^1(\mathbf{N}) - \mathcal{Q}^1(\mathbf{\hat{N}})\|
	\le 3 \mathcal{L}^2 \|\mathbf{N} - \mathbf{\hat{N}}\|.
	\]
	
	Following a similar procedure, one can show that the remaining terms 
	$\mathcal{Q}^k$ $(k = 2, \ldots, 17)$ also satisfy the Lipschitz condition. 
	Finally, substituting all these estimates into the relation  
	\[
	\|\mathbf{\hat{J}}(\mathbf{N}) - \mathbf{\hat{J}}(\mathbf{\hat{N}})\|
	= \sum_{k=1}^{17} \|\mathcal{Q}^k(\mathbf{N}) - \mathcal{Q}^k(\mathbf{\hat{N}})\|,
	\]
	we obtain a positive constant $\zeta$, depending on $T$ and $\mathcal{L}$ but independent of the mesh, 
	such that the Lipschitz inequality \eqref{3_10} holds.
	
	\end{proof}
	
	At this stage, we proceed to establish the main convergence theorem.
	
	\begin{Theorem}
Assume that all wave kernels $\mathcal{K}_i$ $(i = 1,\ldots,7)$ satisfy condition \eqref{3_5}. 
Then the proposed finite volume scheme \eqref{3_1} is non-negative and consistent with 
first-order accuracy. Consequently, the scheme \eqref{3_1} is convergent, and the order 
of convergence coincides with the order of consistency.

	\end{Theorem}
	\begin{proof}
	To prove the convergence of the proposed finite volume scheme \eqref{3_1}, 
	it is necessary to show that the scheme is both non-negative and consistent.
	
	\emph{Non-negativity:}  
	Consider a non-negative approximate solution 
	\(\mathbf{\hat{N}} := \{\hat{N}_1, \hat{N}_2, \ldots, \hat{N}_I\}\), 
	and suppose that its \(i\)-th component satisfies \(\hat{N}_i = 0\). 
	Then, from the definition of the numerical fluxes \eqref{3_3}, we obtain
	\[
	\mathcal{Q}_i^k \ge 0 \quad \text{for all } k \in S := \{1,5,9,13,17\},
	\qquad 
	\mathcal{Q}_i^k = 0 \quad \text{for all } 
	k \in \{1,\ldots,17\} \setminus S.
	\]
	Substituting this relation into equation \eqref{3_2.1}, we obtain 
	\(\hat{J}_i(\mathbf{\hat{N}}) \ge 0\) for any \(i \in \{1, \ldots, I\}\). 
	Combining this result with Proposition~\ref{Prop_2} and Theorem~4.1 in 
	\cite{hundsdorfer2013numerical}, it follows that the proposed numerical scheme 
	is non-negative.

	\emph{Consistency:}  
	The consistency of the proposed scheme \eqref{3_1} can be established by first examining 
	the discretization of the term \eqref{2_3}. We write
	\begin{align}
		F^1_i(t) &= 
		\int_{\omega_{i-1/2}}^{\omega_{i+1/2}}
		\int_{0}^{\omega_{i-1/2}}
		\int_{\mu}^{\omega}\, \mathrm{d}\eta \, \mathrm{d}\mu \, \mathrm{d}\omega
		\mathcal{K}_1(\omega,\mu,\eta)
		\Big[
		f(\eta) f(\omega + \mu - \eta)\big(f(\mu)+f(\omega)\big)
		\\[-0.3em]
		&\hspace{4.1cm}
		- f(\omega) f(\mu)\big(f(\eta)+f(\omega + \mu - \eta)\big)
		\Big]
		\notag\\
		&\quad +
		\int_{\omega_{i-1/2}}^{\omega_{i+1/2}}
		\int_{\omega_{i-1/2}}^{\omega}
		\int_{\mu}^{\omega}\, \mathrm{d}\eta \, \mathrm{d}\mu \, \mathrm{d}\omega
		\mathcal{K}_1(\omega,\mu,\eta)
		\Big[
		f(\eta) f(\omega + \mu - \eta)\big(f(\mu)+f(\omega)\big)
		\\[-0.3em]
		&\hspace{4.1cm}
		- f(\omega) f(\mu)\big(f(\eta)+f(\omega + \mu - \eta)\big)
		\Big]
		.
		\notag
	\end{align}

By employing the midpoint and left rectangular quadrature rules for the second term, we obtain
\begin{align}\label{3_12}
	F^1_i(t) &=
	\int_{\omega_{i-1/2}}^{\omega_{i+1/2}}
	\int_{0}^{\omega_{i-1/2}}
	\int_{\mu}^{\omega}
	\mathcal{K}_1(\omega,\mu,\eta)
	\left[
	f(\eta) f(\omega+\mu-\eta)\big(f(\mu)+f(\omega)\big)
	\right. \notag\\
	&\left.\hspace{3.2cm}
	- f(\omega) f(\mu)\big(f(\eta)+f(\omega+\mu-\eta)\big)
	\right]
	\, \mathrm{d}\eta \, \mathrm{d}\mu \, \mathrm{d}\omega .
\end{align}

Let us focus on the first term appearing on the right-hand side of $F^1_i$:
\begin{align}\label{3_13}
	\begin{split}
		I^1_1
		&=
		\sum_{j=1}^{i-1}
		\int_{\omega_{i-1/2}}^{\omega_{i+1/2}}
		\int_{\omega_{j-1/2}}^{\omega_{j+1/2}}\, \mathrm{d}\eta \mathrm{d}\mu \mathrm{d}\omega
		\int_{\mu}^{\omega_{i-1/2}}
		\mathcal{K}_1(\omega,\mu,\eta)
		f(\mu) f(\eta) f(\omega+\mu-\eta)
		\\
		&\quad +
		\sum_{j=1}^{i-1}
		\int_{\omega_{i-1/2}}^{\omega_{i+1/2}}
		\int_{\omega_{j-1/2}}^{\omega_{j+1/2}}
		\int_{\omega_{i-1/2}}^{\omega}	\, \mathrm{d}\eta \mathrm{d}\mu \mathrm{d}\omega
		\mathcal{K}_1(\omega,\mu,\eta)
		f(\mu) f(\eta) f(\omega+\mu-\eta)
		\\
		&=
		\sum_{j=1}^{i-1}
		\int_{\omega_{i-1/2}}^{\omega_{i+1/2}}
		\int_{\omega_j}^{\omega_{i-1/2}}
		\, \mathrm{d}\eta \mathrm{d}\omega
		\mathcal{K}_1(\omega,\omega_j,\eta)
		N_j f(\eta) f(\omega+\omega_j-\eta)
		\\
		&\quad +
		\sum_{j=1}^{i-1}
		\int_{\omega_{i-1/2}}^{\omega_i}	\, \mathrm{d}\eta
		\mathcal{K}_1(\omega_i,\omega_j,\eta)
		N_i N_j f(\omega_i+\omega_j-\eta)
		\\
		&=
		\sum_{j=1}^{i-1}
		\int_{\omega_{i-1/2}}^{\omega_{i+1/2}}
		\int_{\omega_j}^{\omega_{j+1/2}}
		\, \mathrm{d}\eta \mathrm{d}\omega
		\mathcal{K}_1(\omega,\omega_j,\eta)
		N_j f(\eta) f(\omega+\omega_j-\eta)
		\\
		&\quad +
		\sum_{j=1}^{i-2}
		\int_{\omega_{i-1/2}}^{\omega_{i+1/2}}
		\int_{\omega_{j+1/2}}^{\omega_{i-1/2}}
		\, \mathrm{d}\eta \mathrm{d}\omega
		\mathcal{K}_1(\omega,\omega_j,\eta)
		N_j f(\eta) f(\omega+\omega_j-\eta)
		\\
		&\quad +
		\frac{1}{2}\sum_{j=1}^{i-1}
		\mathcal{K}^1_{i,j,i} N_i N_j^2
		+
		\mathcal{O}(\Delta \omega^2)
		\\
		&=
		\sum_{j=1}^{i-2}
		\sum_{k=j+1}^{i-1}
		\sum_{\omega_{i-1/2} \le \omega_l + \omega_k - \omega_j < \omega_{i+1/2}}
		\mathcal{K}^1_{l+k-j,j,k}
		N_j N_k N_l
		\\
		&\quad +
		\frac{1}{2}\sum_{j=1}^{i-1}
		\left(\mathcal{K}^1_{i,j,j} + \mathcal{K}^1_{i,j,i}\right)
		N_i N_j^2
		+
		\mathcal{O}(\Delta \omega^2)
		\\
		&=
		\mathcal{Q}_i^1 \left(\mathbf{\hat{N}}\right)
		+
		\frac{1}{2}\sum_{j=1}^{i-1}
		\left(\mathcal{K}^1_{i,j,j} + \mathcal{K}^1_{i,j,i}\right)
		N_i N_j^2
		+
		\mathcal{O}(\Delta \omega^2).
	\end{split}
\end{align}

Next, we consider the second term on the right-hand side of $F^1_i$ and change the order 
of integration to simplify the expression 
\begin{align}\label{3_14}
	I_2^1
	&=
	\sum_{j=1}^{i-1}
	\int_{\omega_{i-1/2}}^{\omega_{i+1/2}}
	\int_{\omega_{j-1/2}}^{\omega_{j+1/2}}
	\int_{0}^{\eta}
	\, \mathrm{d}\mu\,\mathrm{d}\eta\,\mathrm{d}\omega
	\mathcal{K}_1(\omega,\mu,\eta)
	f(\omega) f(\eta) f(\omega+\mu-\eta)
	\\
	&\quad+
	\sum_{j=1}^{i-1}
	\int_{\omega_{i-1/2}}^{\omega_{i+1/2}}
	\int_{\omega_{j-1/2}}^{\omega_{j+1/2}}
	\int_{\omega_{i-1/2}}^{\omega}
	\, \mathrm{d}\eta\,\mathrm{d}\mu\,\mathrm{d}\omega
	\mathcal{K}_1(\omega,\mu,\eta)
	f(\omega) f(\eta) f(\omega+\mu-\eta) \notag
	\\
	&=
	\sum_{j=1}^{i-1}
	\int_{0}^{\omega_j}	\, \mathrm{d}\mu
	\mathcal{K}_1(\omega_i, \mu, \omega_j)\,
	N_i N_j\, f(\omega_i+\mu-\omega_j)
	+
	\frac{1}{2} \sum_{j=1}^{i-1}
	\mathcal{K}^1_{i,j,i} N_i^{2} N_j
	+
	\mathcal{O}(\Delta\omega^{2}) \notag
	\\
	&=
	\sum_{j=2}^{i-1}
	\sum_{k=1}^{j-1}
	\sum_{\omega_{k-1/2} \le \omega_l + \omega_j - \omega_i < \omega_{k+1/2}}
	\mathcal{K}^1_{i,l+j-i,j}\,
	N_i N_j N_l
	+
	\frac{1}{2} \sum_{j=1}^{i-1}
	\left(\mathcal{K}^1_{i,j,j} + \mathcal{K}^1_{i,j,i}\right)
	N_i^{2} N_j
	+
	\mathcal{O}(\Delta\omega^{2}) \notag
	\\
	&=
	\mathcal{Q}_i^{2}(\mathbf{\hat{N}})
	+
	\frac{1}{2} \sum_{j=1}^{i-1}
	\left(\mathcal{K}^1_{i,j,j} + \mathcal{K}^1_{i,j,i}\right)
	N_i^{2} N_j
	+
	\mathcal{O}(\Delta\omega^{2}). \notag
\end{align}

Consider the third term on the right-hand side of \(F^1_i\) and apply the midpoint rule for approximation:
\begin{align}\label{3_15}
	\begin{split}
		I_3^1
		&= -\int_{0}^{\omega_{i-1/2}} \int_{\mu}^{\omega_i}\, \mathrm{d}\eta\, \mathrm{d}\mu
		\mathcal{K}_1(\omega_i,\mu,\eta)\, N_i\, f(\mu)\, f(\eta) \\[0.3em]
		&= -\sum_{j=1}^{i-1}\int_{\omega_j}^{\omega_{i-1/2}}\, \mathrm{d}\eta
		\mathcal{K}_1(\omega_i,\omega_j,\eta)\, N_i N_j\, f(\eta)
		\;-\;
		\sum_{j=1}^{i-1}\int_{\omega_{i-1/2}}^{\omega_i}\, \mathrm{d}\eta 
		\mathcal{K}_1(\omega_i,\omega_j,\eta)\, N_i N_j\, f(\eta)\\[0.3em]
		&= -\sum_{j=1}^{i-2}\sum_{k=j+1}^{i-1}
		\int_{\omega_{k-1/2}}^{\omega_{k+1/2}}\, \mathrm{d}\eta
		\mathcal{K}_1(\omega_i,\omega_j,\eta)\, N_i N_j\, f(\eta)
		\;-\;
		\sum_{j=1}^{i-1}\int_{\omega_j}^{\omega_{j+1/2}}\, \mathrm{d}\eta
		\mathcal{K}_1(\omega_i,\omega_j,\eta)\, N_i N_j\, f(\eta) \\[0.3em]
		&\quad
		-\frac{1}{2}\sum_{j=1}^{i-1}\mathcal{K}^1_{i,j,i}\, N_i^2 N_j
		+\mathcal{O}(\Delta\omega^2) \\[0.4em]
		&= \mathcal{Q}_i^3(\mathbf{\hat{N}})
		-\frac{1}{2} \sum_{j=1}^{i-1}\mathcal{K}^1_{i,j,j}\, N_i N_j^2
		-\frac{1}{2} \sum_{j=1}^{i-1}\mathcal{K}^1_{i,j,i}\, N_i^2 N_j
		+\mathcal{O}(\Delta\omega^2).
	\end{split}
\end{align}

Consider the last term on the right-hand side of \(F^1_i\):
\begin{align}\label{3_16}
	I_4^1
	&= -\sum_{j=1}^{i-1}
	\int_{\omega_{i-1/2}}^{\omega_{i+1/2}}
	\int_{\omega_{j-1/2}}^{\omega_{j+1/2}}
	\int_{\mu}^{\omega_{i-1/2}}\,
	\mathrm{d}\eta\, \mathrm{d}\mu\, \mathrm{d}\omega 
	\mathcal{K}_1(\omega,\mu,\eta)\,
	f(\omega)\, f(\mu)\, f(\omega+\mu-\eta)
	\\
	&\quad
	-\sum_{j=1}^{i-1}
	\int_{\omega_{i-1/2}}^{\omega_{i+1/2}}
	\int_{\omega_{j-1/2}}^{\omega_{j+1/2}}
	\int_{\omega_{i-1/2}}^{\omega}\,
	\mathrm{d}\eta\, \mathrm{d}\mu\, \mathrm{d}\omega
	\mathcal{K}_1(\omega,\mu,\eta)\,
	f(\omega)\, f(\mu)\, f(\omega+\mu-\eta)
	\notag\\[0.4em]
	&=
	-\sum_{j=1}^{i-2}
	\int_{\omega_{i-1/2}}^{\omega_{i+1/2}}
	\int_{\omega_j}^{\omega_{i-1/2}}\,
	\mathrm{d}\eta\, \mathrm{d}\omega
	\mathcal{K}_1(\omega,\omega_j,\eta)\,
	f(\omega)\, N_j\, f(\omega+\omega_j-\eta)
	\notag\\
	&\quad
	-\sum_{j=1}^{i-1}
	\int_{\omega_{i-1/2}}^{\omega_{i}}\,
	\mathrm{d}\eta
	\mathcal{K}_1(\omega_i,\omega_j,\eta)\,
	N_i N_j\, f(\omega_i+\omega_j-\eta)
	\notag\\[0.4em]
	&=
	-\sum_{j=1}^{i-2}\sum_{k=j+1}^{i-1}
	\sum_{\omega_{k-1/2} \le \,\omega_i-\omega_j-\omega_l < \omega_{k+1/2}}
	\mathcal{K}_1\!\left(\omega_i,\omega_j,\omega_i+\omega_j-\omega_l\right)
	N_i N_j N_l
	\notag\\
	&\quad
	-\frac{1}{2}\sum_{j=1}^{i-1} \mathcal{K}^1_{i,j,i}\, N_i N_j^2
	-\frac{1}{2}\sum_{j=1}^{i-1} F(\omega_i,\omega_j,\omega_j)\, N_i^2 N_j
	+\mathcal{O}(\Delta\omega^2)
	\notag\\[0.4em]
	&=
	\mathcal{Q}_i^4
	-\frac{1}{2}\sum_{j=1}^{i-1} \mathcal{K}^1_{i,j,i}\, N_i N_j^2
	-\frac{1}{2}\sum_{j=1}^{i-1} F(\omega_i,\omega_j,\omega_j)\, N_i^2 N_j
	+\mathcal{O}(\Delta\omega^2). \notag
\end{align}

	By substituting equations \eqref{3_13}–\eqref{3_14} into equation \eqref{3_12}, we obtain
	\begin{align}\label{3_17}
		F^1_i(t)  = \sum_{k=1}^{4}\mathcal{Q}_i^k +\mathcal{O} \left(\Delta \omega^2\right):= \hat{\mathcal{Q}}_i^1 +\mathcal{O} \left(\Delta \omega^2\right).
	\end{align}
	
	
Next, we compute the discretization error corresponding to term \eqref{2_4}. 
The first term on the right-hand side of \(F^2_i\) is discretized as follows 
\begin{align}\label{3_18}
	I^2_1 &= \int_{\omega_{i-1/2}}^{\omega_{i+1/2}}
	\int_{\omega}^{R}
	\int_{0}^{\omega} 
	\,\mathrm{d}\eta\,\mathrm{d}\mu\,\mathrm{d}\omega 
	\mathcal{K}_2(\omega,\mu-\omega,\eta)\,
	f(\eta)\,f(\mu - \eta)\,f(\mu-\omega)\\[0.2cm]
	&= \int_{\omega_{i-1/2}}^{\omega_{i+1/2}}
	\int_{\omega_{i+1/2}}^{R}
	\int_{0}^{\omega_{i-1/2}}
	\,\mathrm{d}\eta\,\mathrm{d}\mu\,\mathrm{d}\omega
	\mathcal{K}_2(\omega,\mu-\omega,\eta)\,
	f(\eta)\,f(\mu - \eta)\,f(\mu-\omega) \notag\\
	&\quad + 
	\int_{\omega_{i-1/2}}^{\omega_{i+1/2}}
	\int_{\omega_{i+1/2}}^{R}
	\int_{\omega_{i-1/2}}^{\omega}
	\,\mathrm{d}\eta\,\mathrm{d}\mu\,\mathrm{d}\omega
	\mathcal{K}_2(\omega,\mu-\omega,\eta)\,
	f(\eta)\,f(\mu - \eta)\,f(\mu-\omega) \notag\\[0.2cm]
	&= \sum_{j=i+1}^{I} \sum_{k=1}^{i-1}
	\int_{\omega_{i-1/2}}^{\omega_{i+1/2}}
	\int_{\omega_{j-1/2}}^{\omega_{j+1/2}}
	\,\mathrm{d}\mu\,\mathrm{d}\omega 
	\mathcal{K}_2(\omega,\mu-\omega,\omega_k)\,
	N_k\,f(\mu - \omega_k)\,f(\mu-\omega)\notag\\
	&\quad + \frac{1}{2}
	\sum_{j=i+1}^{I}
	\int_{\omega_{j-1/2}}^{\omega_{j+1/2}}\,\mathrm{d}\mu
	\mathcal{K}_2(\omega_i,\mu-\omega_i,\omega_i)\,
	N_i\,f(\mu - \omega_i)\,f(\mu-\omega_i)
	+ \mathcal{O}(\Delta\omega^2) \notag\\[0.2cm]
	&=
	\sum_{j=i+1}^{I} \sum_{k=1}^{i-1}
	\sum_{\omega_{j-1/2} \le \omega_k + \omega_l < \omega_{j+1/2}}
	\int_{\omega_{i-1/2}}^{\omega_{i+1/2}}\,
	\mathrm{d}\omega
	\mathcal{K}_2(\omega,\omega_l+\omega_k-\omega,\omega_k)\,
	N_k N_l\,f(\omega_l+\omega_k-\omega) \notag\\
	&\quad + \frac{1}{2}
	\sum_{j=i+1}^{I}
	\sum_{\omega_{j-1/2} \le \omega_k + \omega_i < \omega_{j+1/2}}
	\mathcal{K}^2_{i,k,i}\,N_i N_k^2
	+ \mathcal{O}(\Delta\omega^2) \notag\\[0.2cm]
	&= \mathcal{Q}_i^5
	+ \frac{1}{2}
	\sum_{j=i+1}^{I}
	\sum_{\omega_{j-1/2} \le \omega_k + \omega_i < \omega_{j+1/2}}
	\mathcal{K}^2_{i,k,i}\,N_i N_k^2
	+ \mathcal{O}(\Delta\omega^2). \notag
\end{align}

Consider the second term on the right-hand side of $F^2_i$ 
\begin{align}\label{3_19}
	I^2_2 &= 
	\int_{\omega_{i-1/2}}^{\omega_{i+1/2}}
	\int_{\omega}^{R}
	\int_{0}^{\omega} 
	\,\mathrm{d}\eta\,\mathrm{d}\mu\,\mathrm{d}\omega
	\mathcal{K}_2(\omega,\mu-\omega,\eta)\,
	f(\omega)\,f(\eta)\,f(\mu - \eta) \\[0.2cm]
	&=
	\int_{\omega_{i-1/2}}^{\omega_{i+1/2}}
	\int_{\omega_{i+1/2}}^{R}
	\int_{0}^{\omega_{i-1/2}} 
	\,\mathrm{d}\eta\,\mathrm{d}\mu\,\mathrm{d}\omega
	\mathcal{K}_2(\omega,\mu-\omega,\eta)\,
	f(\omega)\,f(\eta)\,f(\mu - \eta) \notag \\
	&\quad +
	\int_{\omega_{i-1/2}}^{\omega_{i+1/2}}
	\int_{\omega_{i+1/2}}^{R}
	\int_{\omega_{i-1/2}}^{\omega} 
	\,\mathrm{d}\eta\,\mathrm{d}\mu\,\mathrm{d}\omega
	\mathcal{K}_2(\omega,\mu-\omega,\eta)\,
	f(\omega)\,f(\eta)\,f(\mu - \eta) \notag \\[0.2cm]
	&=
	\sum_{j=i+1}^{I} 
	\sum_{k=1}^{i-1}
	\int_{\omega_{j-1/2}}^{\omega_{j+1/2}}
	\mathcal{K}_2(\omega_i,\mu-\omega_i,\omega_k)\,
	N_i N_k\,f(\mu - \omega_k)
	\,\mathrm{d}\mu \notag \\
	&\quad +
	\frac{1}{2}
	\sum_{j=i+1}^{I}  
	\int_{\omega_{j-1/2}}^{\omega_{j+1/2}}\,\mathrm{d}\mu
	\mathcal{K}_2(\omega_i,\mu-\omega_i,\omega_i)\,
	N_i^2\,f(\mu - \omega_i)
	+ \mathcal{O}(\Delta \omega^2) \notag \\[0.2cm]
	&=
	\mathcal{Q}_i^6
	+
	\frac{1}{2}
	\sum_{j=i+1}^{I}
	\sum_{\omega_{j-1/2}\le \omega_l + \omega_i < \omega_{j+1/2}}
	\mathcal{K}^2_{i,l,i}\, N_i^2 N_l
	+
	\mathcal{O}(\Delta \omega^2). \notag
\end{align}

Consider the third term on the right-hand side of $F_i^2$:
\begin{align}\label{3_20}
	I^2_3 &= - \int_{\omega_{i-1/2}}^{\omega_{i+1/2}}
	\int_{\omega_{i+1/2}}^{R}
	\int_{0}^{\omega_{i-1/2}}
	\,\mathrm{d}\eta\,\mathrm{d}\mu\,\mathrm{d}\omega
	\mathcal{K}_2(\omega,\mu-\omega,\eta)\,
	f(\omega) f(\eta) f(\mu - \omega) \\
	&\quad
	- \int_{\omega_{i-1/2}}^{\omega_{i+1/2}}
	\int_{\omega_{i+1/2}}^{R}
	\int_{\omega_{i-1/2}}^{\omega}
	\,\mathrm{d}\eta\,\mathrm{d}\mu\,\mathrm{d}\omega
	\mathcal{K}_2(\omega,\mu-\omega,\eta)\,
	f(\omega) f(\eta) f(\mu - \omega) \notag \\
	&= -\sum_{j=i+1}^{I}\sum_{k=1}^{i-1}
	\int_{\omega_{j-1/2}}^{\omega_{j+1/2}}
	\int_{\omega_{k-1/2}}^{\omega_{k+1/2}}
	\,\mathrm{d}\eta\,\mathrm{d}\mu
	\mathcal{K}_2(\omega_i,\mu-\omega_i,\eta)\,
	N_i f(\eta) f(\mu - \omega_i) \notag \\
	&\quad
	- \frac{1}{2}\sum_{j=i+1}^{I}
	\int_{\omega_{j-1/2}}^{\omega_{j+1/2}}\,\mathrm{d}\mu
	\mathcal{K}_2(\omega_i,\mu-\omega_i,\omega_i)\,
	N_i^2 f(\mu - \omega_i)
	+ \mathcal{O}(\Delta\omega^2) \notag \\
	&= \mathcal{Q}_i^7
	- \frac{1}{2}\sum_{j=i+1}^{I}
	\sum_{\omega_{j-1/2} \le \omega_l + \omega_i < \omega_{j+1/2}}
	\mathcal{K}^2_{i,l,i}\, N_i^2 N_l
	+ \mathcal{O}(\Delta\omega^2). \notag
\end{align}

We focus on the fourth term on the right-hand side of $F_i^2$ for its discretization 
\begin{align}\label{3_21}
	I^2_4 
	&= - \int_{\omega_{i-1/2}}^{\omega_{i+1/2}}
	\int_{\omega_{i+1/2}}^{R}
	\int_{0}^{\omega_{i-1/2}}\,
	\mathrm{d}\eta\,\mathrm{d}\mu\,\mathrm{d}\omega 
	\mathcal{K}_2(\omega,\mu-\omega,\eta)\,
	f(\omega)\, f(\mu-\omega)\, f(\mu-\eta) \\
	&\quad
	- \int_{\omega_{i-1/2}}^{\omega_{i+1/2}}
	\int_{\omega_{i+1/2}}^{R}
	\int_{\omega_{i-1/2}}^{\omega}\,
	\mathrm{d}\eta\,\mathrm{d}\mu\,\mathrm{d}\omega
	\mathcal{K}_2(\omega,\mu-\omega,\eta)\,
	f(\omega)\, f(\mu-\omega)\, f(\mu-\eta) \notag \\
	&= - \sum_{j=i+1}^{I}\sum_{k=1}^{i-1}
	\int_{\omega_{j-1/2}}^{\omega_{j+1/2}}
	\int_{\omega_{k-1/2}}^{\omega_{k+1/2}}\,
	\mathrm{d}\eta\,\mathrm{d}\mu 
	\mathcal{K}_2\!\left(\omega_i,\mu-\omega_i,\eta\right)
	N_i\, f(\mu-\omega_i)\, f(\mu-\eta)\notag \\
	&\quad
	- \frac{1}{2}\sum_{j=i+1}^{I}
	\int_{\omega_{j-1/2}}^{\omega_{j+1/2}}\,
	\mathrm{d}\mu 
	\mathcal{K}_2\!\left(\omega_i,\mu-\omega_i,\omega_i\right)
	N_i\, f(\mu-\omega_i)\, f(\mu-\omega_i)\notag \\
	&= \mathcal{Q}_i^8
	- \frac{1}{2}\sum_{j=i+1}^{I}
	\sum_{\omega_{j-1/2}\le \omega_l+\omega_i<\omega_{j+1/2}}
	\mathcal{K}^2_{i,l,i}\, N_i N_l^{2}
	+ \mathcal{O}(\Delta\omega^2). \notag
\end{align}

Combining equations \eqref{3_18}--\eqref{3_21} in equation \eqref{2_4}, we obtain
\begin{align}\label{3_22}
	F^2_i(t)
	= \sum_{k=5}^{8} \mathcal{Q}_i^k
	+ \mathcal{O}(\Delta\omega^2)
	:= \hat{\mathcal{Q}}_i^2
	+ \mathcal{O}(\Delta\omega^2).
\end{align}

To compute the discretization error associated with $F_i^3$, we begin by discretizing the first term on the right-hand side of \eqref{2_5}   
\begin{align}\label{3_23}
	I^3_1 
	&:= \int_{\omega_{i-1/2}}^{\omega_{i+1/2}} \int_{\omega}^{R} \int_{\omega}^{\mu}
	\, \mathrm{d}\eta\, \mathrm{d}\mu\, \mathrm{d}\omega
	\mathcal{K}_3(\omega,\mu,\eta)\,
	f(\mu)\, f(\eta)\, f(\omega + \mu - \eta) \\[0.3em]
	&= \int_{\omega_{i-1/2}}^{\omega_{i+1/2}} 
	\int_{\omega_{i+1/2}}^{R}
	\int_{\omega}^{\mu}
	\, \mathrm{d}\eta\, \mathrm{d}\mu\, \mathrm{d}\omega
	\mathcal{K}_3(\omega,\mu,\eta)\,
	f(\mu)\, f(\eta)\, f(\omega + \mu - \eta) \notag \\[0.3em]
	&= \sum_{j=i+1}^{I} 
	\int_{\omega_{i-1/2}}^{\omega_{i+1/2}}
	\int_{\omega}^{\omega_j}
	\, \mathrm{d}\eta\, \mathrm{d}\omega
	\mathcal{K}_3(\omega,\omega_j,\eta)\,
	N_j\, f(\eta)\, f(\omega + \omega_j - \eta)
	+ \mathcal{O}(\Delta \omega^3) \notag \\[0.3em]
	&= \frac{1}{2} 
	\sum_{j=i+1}^{I}
	\left( \mathcal{K}^3_{i,j,i} + \mathcal{K}^3_{i,j,j} \right)
	N_i\, N_j^2
	+ \mathcal{Q}_i^9
	+ \mathcal{O}(\Delta \omega^2). \notag
\end{align}

	Next, we focus on the second term on the right-hand side of $F_i^3$ 
	\begin{align}\label{3_24}
		I^3_2 &= \int_{\omega_{i-1/2}}^{\omega_{i+1/2}}\int_{\omega_{i+1/2}}^R \int_{\omega}^{\mu} \dd \eta \dd \mu \dd \omega \mathcal{K}_3(\omega,\mu,\eta) f(\omega) f(\eta) f(\omega + \mu - \eta)\\
		& = \sum_{j=i+1}^{I} \int_{\omega_i}^{\omega_{i+1/2}}\dd \mu \dd \eta \int_{\omega_{j-1/2}}^{\omega_{j+1/2}} \mathcal{K}_3(\omega_i,\mu,\eta) N_i f(\eta) f(\omega_i + \mu - \eta)  \notag\\
		& \quad + \sum_{j=i+1}^{I}\int_{\omega_{j-1/2}}^{\omega_{j+1/2}} \int_{\eta}^{R}\dd \mu \dd \eta\mathcal{K}_3(\omega_i,\mu,\eta) N_i f(\eta) f(\omega_i + \mu - \eta)  +\mathcal{O} \left(\Delta \omega^3\right) \notag\\
		&= \frac{1}{2}\sum_{j=i+1}^{I} \int_{\omega_{j-1/2}}^{\omega_{j+1/2}}   \mathcal{K}_3(\omega_i,\mu,\omega_i) N_i^2 f( \mu ) \dd \mu + \sum_{j=i+1}^{I} \int_{\omega_j}^{R}\dd \mu\mathcal{K}_3(\omega_i,\mu,\omega_j) N_i N_j f(\omega_i + \mu - \omega_j) +\mathcal{O} \left(\Delta \omega^2\right) \notag\\
		& = \frac{1}{2}\sum_{j=i+1}^{I}  \left(\mathcal{K}^3_{i,j,i} + \mathcal{K}^3_{i,j,j} \right) N_i^2 N_j+ \mathcal{Q}_i^{10} +\mathcal{O} \left(\Delta \omega^2\right). \notag
	\end{align}
We now examine the third term on the right-hand side of $F_i^3$ 
\begin{align}\label{3_25}
	I^3_3
	&= - \int_{\omega_{i-1/2}}^{\omega_{i+1/2}}
	\int_{\omega}^{R}
	\int_{\omega}^{\mu}
	\, \mathrm{d}\eta\, \mathrm{d}\mu\, \mathrm{d}\omega
	\mathcal{K}_3(\omega,\mu,\eta)\,
	f(\omega)\, f(\mu)\, f(\eta) \\[0.3em]
	&= - \sum_{j=i+1}^{I} 
	\int_{\omega_i}^{\omega_j}\,
	\mathrm{d}\eta
	\mathcal{K}_3(\omega_i,\omega_j,\eta)\,
	N_i N_j\, f(\eta) \notag \\[0.3em]
	&= -\frac{1}{2} \sum_{j=i+1}^{I} 
	\mathcal{K}^3_{i,j,i}\, N_i^2 N_j
	\;-\; \frac{1}{2} \sum_{j=i+1}^{I} 
	\mathcal{K}^3_{i,j,j}\, N_i N_j^2  \notag \\
	&\qquad 
	-\sum_{j=i+2}^{I} 
	\sum_{k=i+1}^{j-1}
	\int_{\omega_{k-1/2}}^{\omega_{k+1/2}}
	H(\omega_i,\omega_j,\eta)\,
	N_i N_j\, f(\eta)\,
	\mathrm{d}\eta
	+ \mathcal{O}(\Delta\omega^2) \notag \\[0.3em]
	&= -\frac{1}{2} \sum_{j=i+1}^{I} 
	\mathcal{K}^3_{i,j,i}\, N_i^2 N_j
	\;-\; \frac{1}{2} \sum_{j=i+1}^{I} 
	\mathcal{K}^3_{i,j,j}\, N_i N_j^2
	+ \mathcal{Q}_i^{11}
	+\mathcal{O}(\Delta\omega^2). \notag
\end{align}

Next, we consider the final term in $F_i^3$ 
\begin{align}\label{3_26}
	I^3_4
	&= - \int_{\omega_{i-1/2}}^{\omega_{i+1/2}}
	\int_{\omega_{i+1/2}}^{R}
	\int_{\omega}^{\omega_{i+1/2}}\,
	\mathrm{d}\eta\, \mathrm{d}\mu\, \mathrm{d}\omega
	\mathcal{K}_3(\omega,\mu,\eta)\,
	f(\omega)\, f(\mu)\, f(\omega+\mu-\eta) \\[0.3em]
	&\quad 
	- \int_{\omega_{i-1/2}}^{\omega_{i+1/2}}
	\int_{\omega_{i+1/2}}^{R}
	\int_{\omega_{i+1/2}}^{R}\,
	\mathrm{d}\eta\, \mathrm{d}\mu\, \mathrm{d}\omega
	\mathcal{K}_3(\omega,\mu,\eta)\,
	f(\omega)\, f(\mu)\, f(\omega+\mu-\eta) \notag \\[0.3em]
	&= \frac{1}{2} \sum_{j=i+1}^{I}
	\mathcal{K}^3_{i,j,i}\, N_i N_j^2
	+ \int_{\omega_{i+1/2}}^{\omega_{i+1}}\,
	\mathrm{d}\eta
	\mathcal{K}_3(\omega_i,\omega_j,\eta)\,
	N_i N_j\, f(\omega_i+\omega_j-\eta) \notag \\
	&\quad
	+\sum_{j=i+2}^{I}
	\sum_{k=i+1}^{j-1}
	\int_{\omega_{k-1/2}}^{\omega_{k+1/2}}\,
	\mathrm{d}\eta
	\mathcal{K}_3(\omega_i,\omega_j,\eta)\,
	N_i N_j\, f(\omega_i+\omega_j-\eta) \notag \\
	&\quad
	+\sum_{j=i+2}^{I}
	\int_{\omega_{j-1/2}}^{\omega_j}
	\mathcal{K}_3(\omega_i,\omega_j,\eta)\,
	N_i N_j\, f(\omega_i+\omega_j-\eta)\,
	\mathrm{d}\eta
	+ \mathcal{O}(\Delta\omega^2) \notag \\[0.3em]
	&= \frac{1}{2} \sum_{j=i+1}^{I}
	\mathcal{K}^3_{i,j,i}\, N_i N_j^2
	+ \frac{1}{2} \sum_{j=i+1}^{I}
	\mathcal{K}^3_{i,j,j}\, N_i^2 N_j
	+ \mathcal{Q}_i^{12}
	+ \mathcal{O}(\Delta\omega^2). \notag
\end{align}

Combining equations \eqref{3_23}--\eqref{3_26} in equation \eqref{2_5}, we obtain  
\begin{align}\label{3_27}
	F^3_i(t)  
	= \sum_{k=9}^{12} \mathcal{Q}_i^k 
	+ \mathcal{O}\!\left(\Delta \omega^{2}\right)
	= \hat{\mathcal{Q}}_i^3 
	+ \mathcal{O}\!\left(\Delta \omega^{2}\right).
\end{align}

To evaluate the discretization error of the terms \(F_i^k\) for \(k = 4, \ldots, 8\), we follow the procedure described in \cite{das2024numerical}. Specifically,
\begin{align}
	F^4_i 
	&= \sum_{(j,k)\in \mathcal{I}_{j,k}^i}
	\mathcal{K}_1(\omega_k+\omega_j,\,\omega_j)\,
	N_j N_k
	+ \mathcal{O}\!\left(\Delta \omega^3\right)
	= \mathcal{Q}_i^{13}
	+ \mathcal{O}\!\left(\Delta \omega^3\right),\\[0.3em]
	F^5_i 
	&= \sum_{j=1}^{I}
	\mathcal{K}_2(\omega_i,\,\omega_j)\,
	N_i N_j
	+ \mathcal{O}\!\left(\Delta \omega^3\right)
	= \mathcal{Q}_i^{14}
	+ \mathcal{O}\!\left(\Delta \omega^3\right),\\[0.3em]
	F^6_i 
	&= \sum_{j=1}^{i-1}
	\mathcal{K}_3(\omega_i,\,\omega_j)\,
	N_i N_j
	+ \mathcal{O}\!\left(\Delta \omega^3\right)
	= \mathcal{Q}_i^{15}
	+ \mathcal{O}\!\left(\Delta \omega^2\right),\\[0.3em]
	F^7_i 
	&= \sum_{j=i+1}^{I}
	\mathcal{K}_4(\omega_i,\,\omega_j)\,
	N_i N_j
	+ \mathcal{O}\!\left(\Delta \omega^2\right)
	= \mathcal{Q}_i^{16}
	+ \mathcal{O}\!\left(\Delta \omega^2\right),\\[0.3em]
	F^8_i 
	&= \sum_{(j,k)\in \mathcal{J}_{j,k}^i}
	\mathcal{K}_4(\omega_j - \omega_k,\,\omega_j)\,
	N_j N_k
	+ \mathcal{O}\!\left(\Delta \omega^3\right)
	= \mathcal{Q}_i^{18}
	+ \mathcal{O}\!\left(\Delta \omega^3\right).
\end{align}

Now, by substituting all the obtained discretization error estimates into 
Definition~\ref{def_1}, we can evaluate the overall discretization error:
\begin{align*}
	\epsilon_i 
	= \left| J_i(\mathbf{N}) - \hat{J}_i(\mathbf{N}) \right|
	= \sum_{k=1}^{3} \left| F_i^k(\mathbf{N}) - \hat{\mathcal{Q}}_i^k(\mathbf{N}) \right|
	+ \sum_{k=4}^{8} \left| F_i^k(\mathbf{N}) - \mathcal{Q}_i^k(\mathbf{N}) \right|
	= \mathcal{O}(\Delta \omega^2).
\end{align*}
Consequently, we obtain  
\[
\|\epsilon\| = \sum_{i=1}^{I} |\epsilon_i|
= \mathcal{O}(\Delta \omega).
\]

\emph{Convergence:}  
To complete the convergence proof, we combine the above estimates with 
Proposition~\ref{Prop_2}, which ensures that all the required conditions for the 
convergence analysis in \cite{linz1975convergence} are satisfied. Hence, the proposed 
numerical scheme \eqref{3_1} is convergent. Moreover, the scheme exhibits 
first-order convergence, matching the order of consistency.

\end{proof}

	\section{Numerical Tests}
In this section, we present several test problems to assess the performance of the newly proposed
finite volume scheme (FVS) given in~\eqref{2_12}, and compare the numerical results with the corresponding
theoretical predictions. The mixed $3$-- and $4$--wave kinetic equation~\eqref{4wave} is solved using the
FVS~\eqref{2_12} under a variety of initial conditions.

For these numerical experiments, we consider two different wave--frequency relations,
\[
\omega(|k|) = |k|^\rho, \qquad \rho = 2 \ \text{or} \ 3.
\]
Our investigation proceeds in two directions. 
First, for each choice of the frequency function $\omega(|k|)$, we examine how the solution behaves
for different values of the parameters $C_1$ and $C_2$.
Second, we analyze the effect of varying the homogeneity degrees $\sigma$ and $\gamma$ while keeping
the frequency function $\omega(|k|)$ fixed.

To study the qualitative behavior of the numerical solution, we monitor the evolution of the total
third moment and total energy of the system. Their discrete approximations are given by
\begin{align}
	M_3(t) &= \sum_{i=1}^{I} \omega_i^3\,|k|^2(\omega_i)\, |k|'(\omega_i)\, N_i, \label{4_1}\\[4pt]
	E(t) &= \sum_{i=1}^{I} \omega_i\, |k|^2(\omega_i)\, |k|'(\omega_i)\, N_i. \label{4_2}
\end{align}

In each test problem, the computation is performed over the time interval \([0, 30]\) using a time step of \(\Delta t = 0.1\). All simulations are carried out on a uniform grid. It is observed that the choice of the spatial step size \(\Delta \omega\) depends on the initial condition \(f^{\mathrm{in}}(\omega)\). For the first two test cases, we set \(\Delta \omega = 0.33\), while for the final test case we use \(\Delta \omega = 0.1\).

	\subsection{Test Case I}\label{test_1}
For the first test case, we consider the analytic initial data
\begin{align}\label{4_3}
	f^{\mathrm{in}}(\omega) = \omega e^{-\omega}, \qquad \text{for } \omega \ge 0.
\end{align}

To implement the FVS \eqref{2_12}, we begin by choosing equal values for the degrees of homogeneity $\sigma$ and $\gamma$, and we define the computational domain as $[10^{-9},\, 10]$ with a step size $\Delta \omega = 0.33$. Under this setup, we first plot the initial distribution $f^{\mathrm{in}}(\omega)$ in Figure~\ref{test1.1}, and the corresponding final state $f(T,\omega)$ at $T=30$ in Figure~\ref{test1.2}, for the frequency function $\omega(|k|) = |k|^2$ and $C_1=C_2=1$.

	\begin{figure}[H]
		\begin{subfigure}{.45\textwidth}
			\centering
			\includegraphics[width=1.0\textwidth]{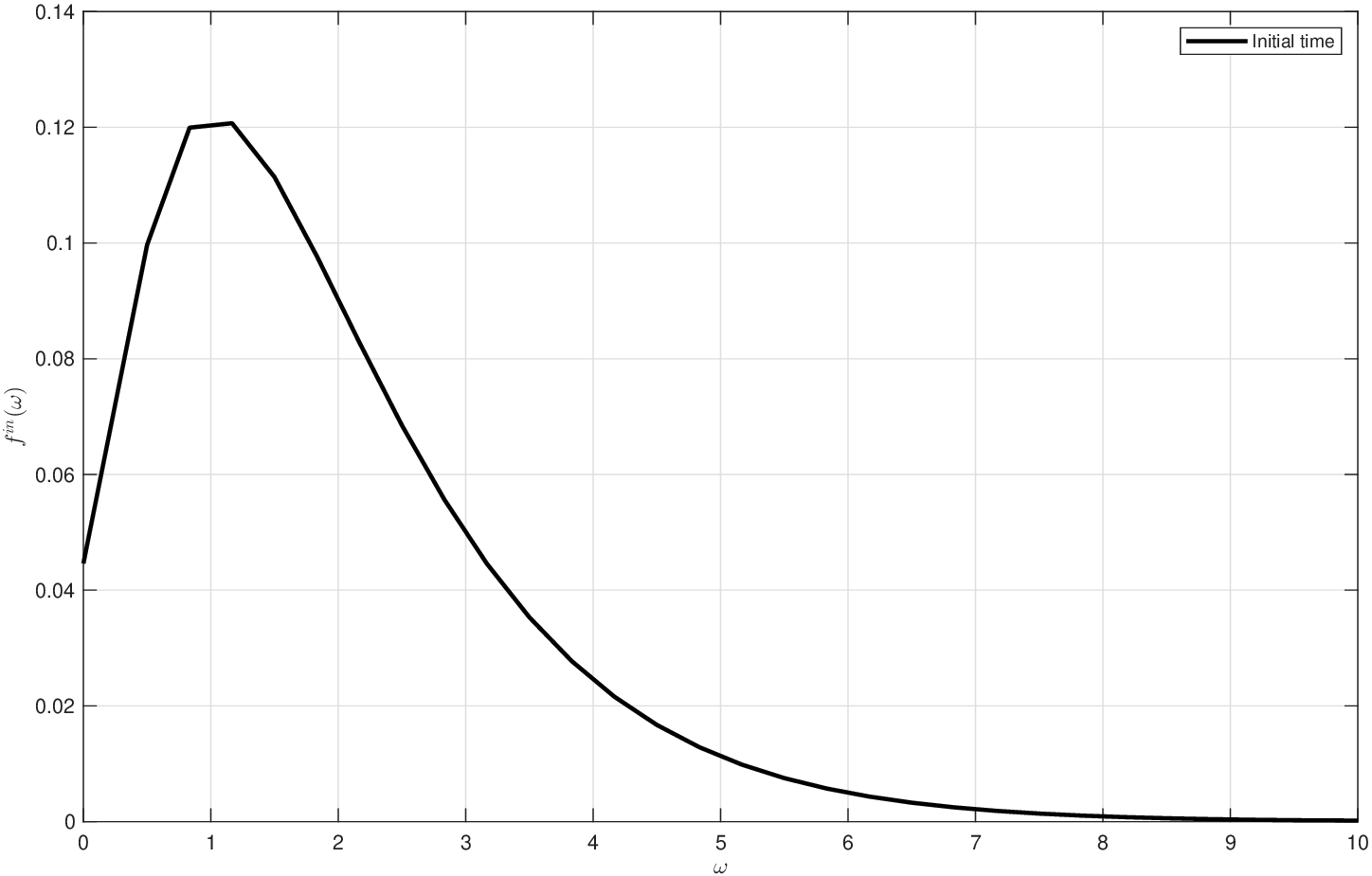}
			\caption{Initial Time}
			\label{test1.1}
		\end{subfigure}
		\begin{subfigure}{.45\textwidth}
			\centering
			\includegraphics[width=1.0\textwidth]{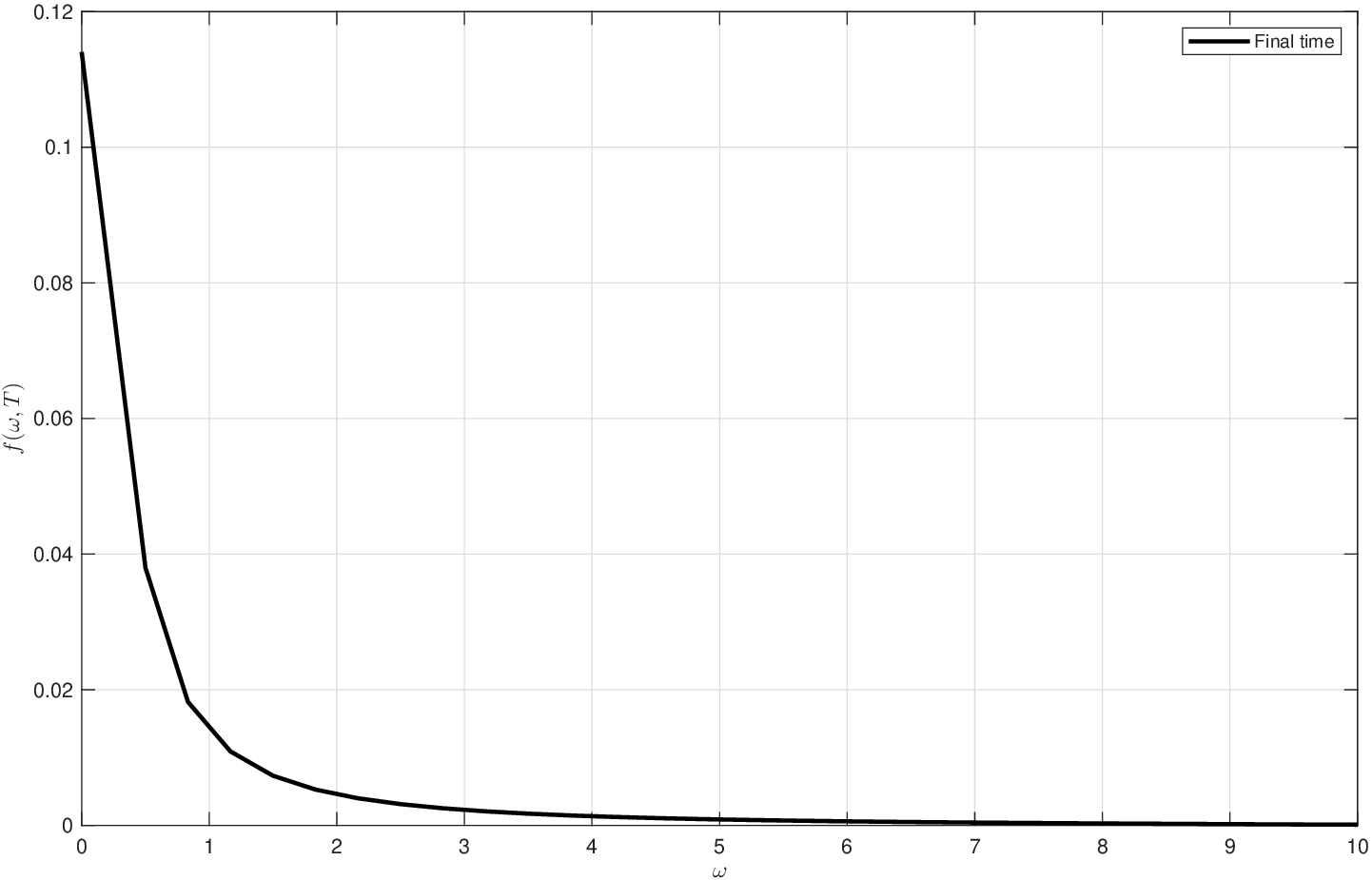}
			\caption{Final Time}
			\label{test1.2}
		\end{subfigure}
		\caption{Evolution of wave density $f(\omega)$ at initial and final time.}
		\label{test1}
	\end{figure}
	
Considering the initial data given in \eqref{4_3} and the frequency function 
$|k|(\omega) = \omega^{1/2}$, the time evolution of the total energy $E(t)$  and the third-order moment $M_3(t)$ is illustrated in 
Figure~\ref{f1} for different values of the parameters $C_1$ and $C_2$, while 
the values of $\sigma$ and $\gamma$ are kept fixed.

	\begin{figure}[H]
		\begin{subfigure}{.45\textwidth}
			\centering
			\includegraphics[width=1.0\textwidth]{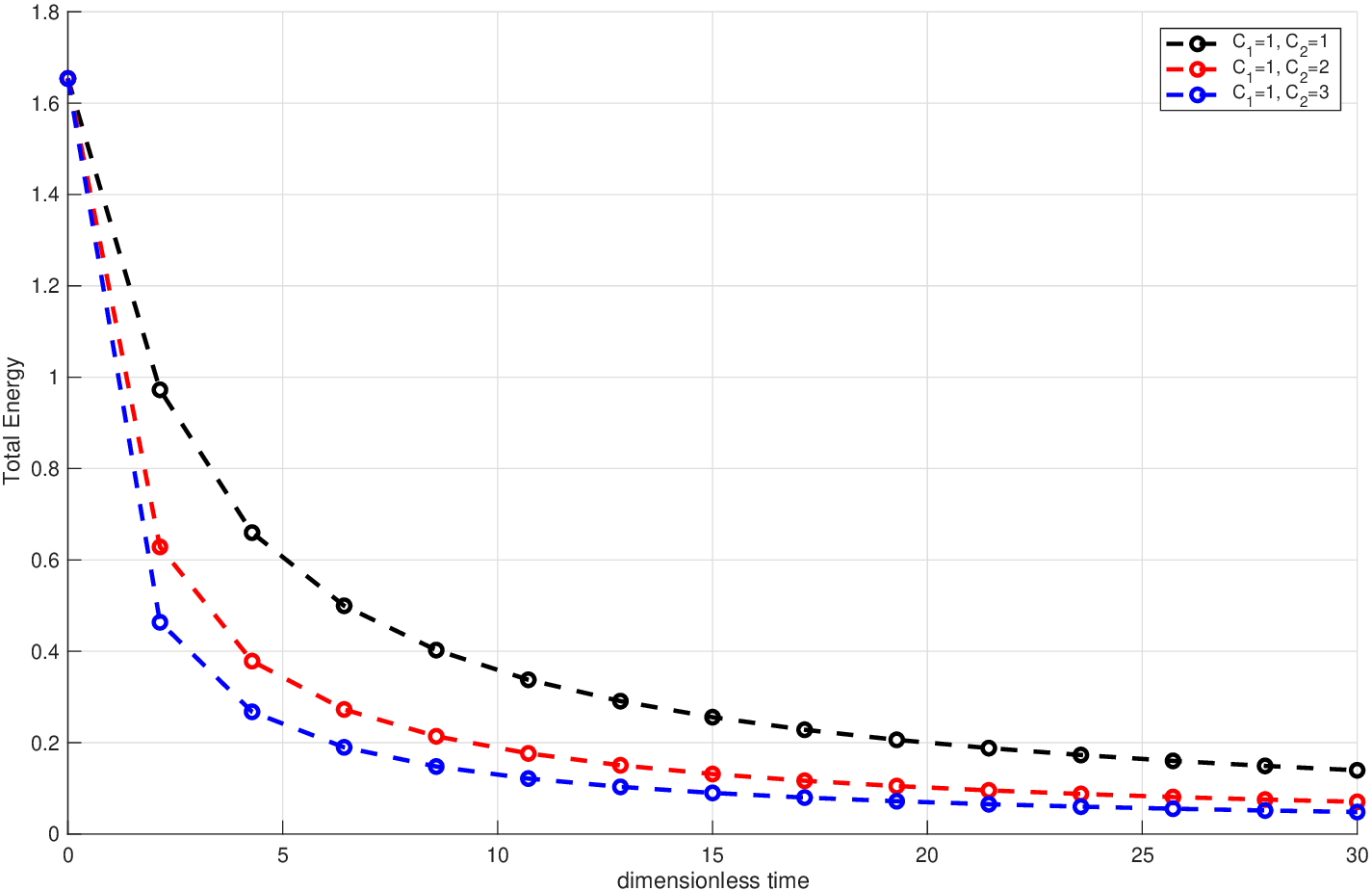}
			\caption{}
			\label{f1.1}
		\end{subfigure}
		\begin{subfigure}{.45\textwidth}
			\centering
			\includegraphics[width=1.0\textwidth]{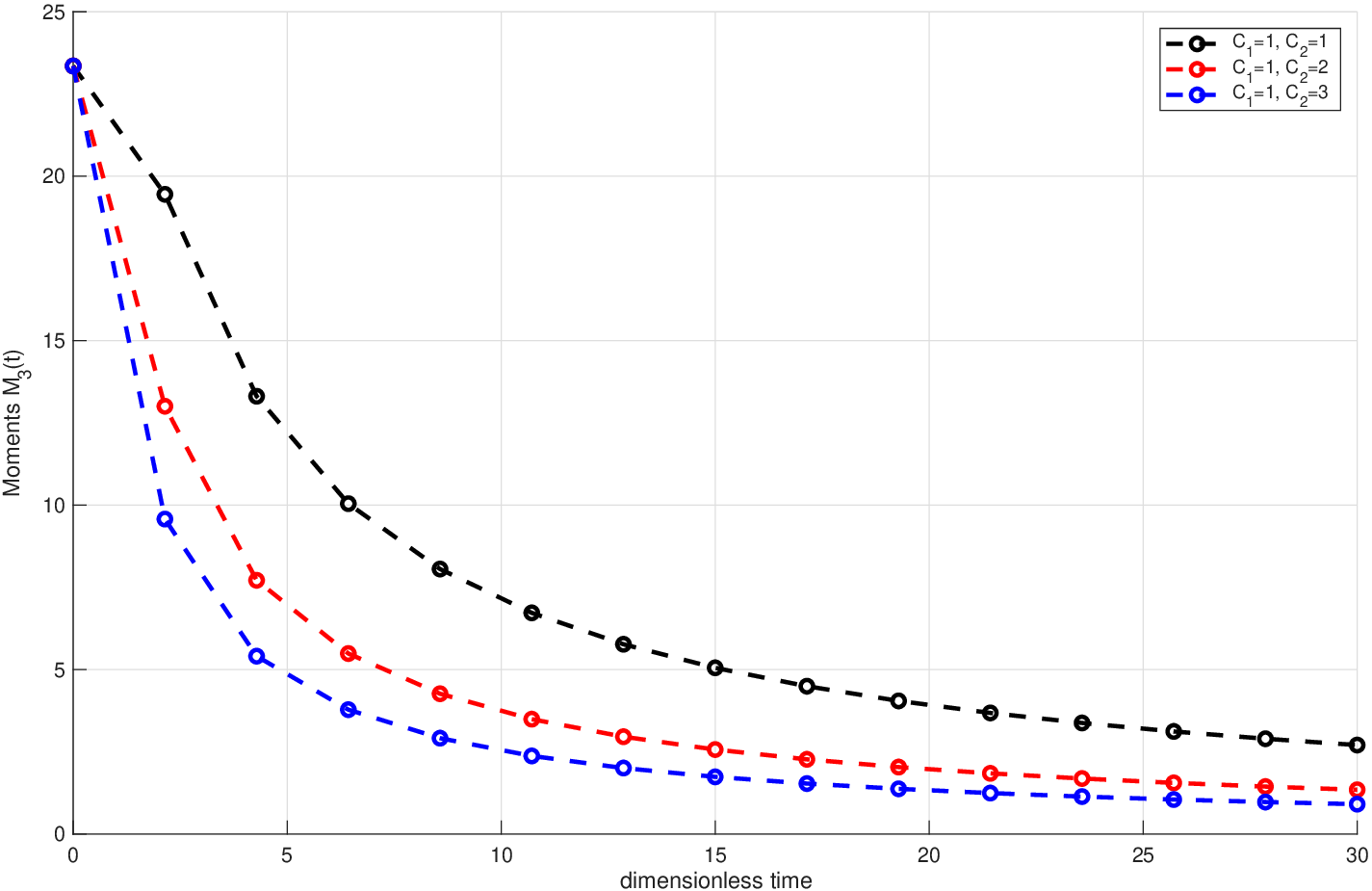}
			\caption{}
			\label{f1.3}
		\end{subfigure}
		\caption{Time evolution of the total energy (A) and third order moment (B) for different values of $C_1$ and $C_2$ when $\sigma=0.50$, $\gamma = 0.50$, and  $\ds \left|k\right|\left(\omega\right) = \sqrt{\omega}$.}
		\label{f1}
	\end{figure}
	
Using a setup similar to that in Figure~\ref{f1}, we plot the total energy, 
mass, and third-order moment in Figure~\ref{f2} for the frequency function 
$|k|(\omega) = \omega^{1/3}$.

		\begin{figure}[H]
		\begin{subfigure}{.45\textwidth}
			\centering
			\includegraphics[width=1.0\textwidth]{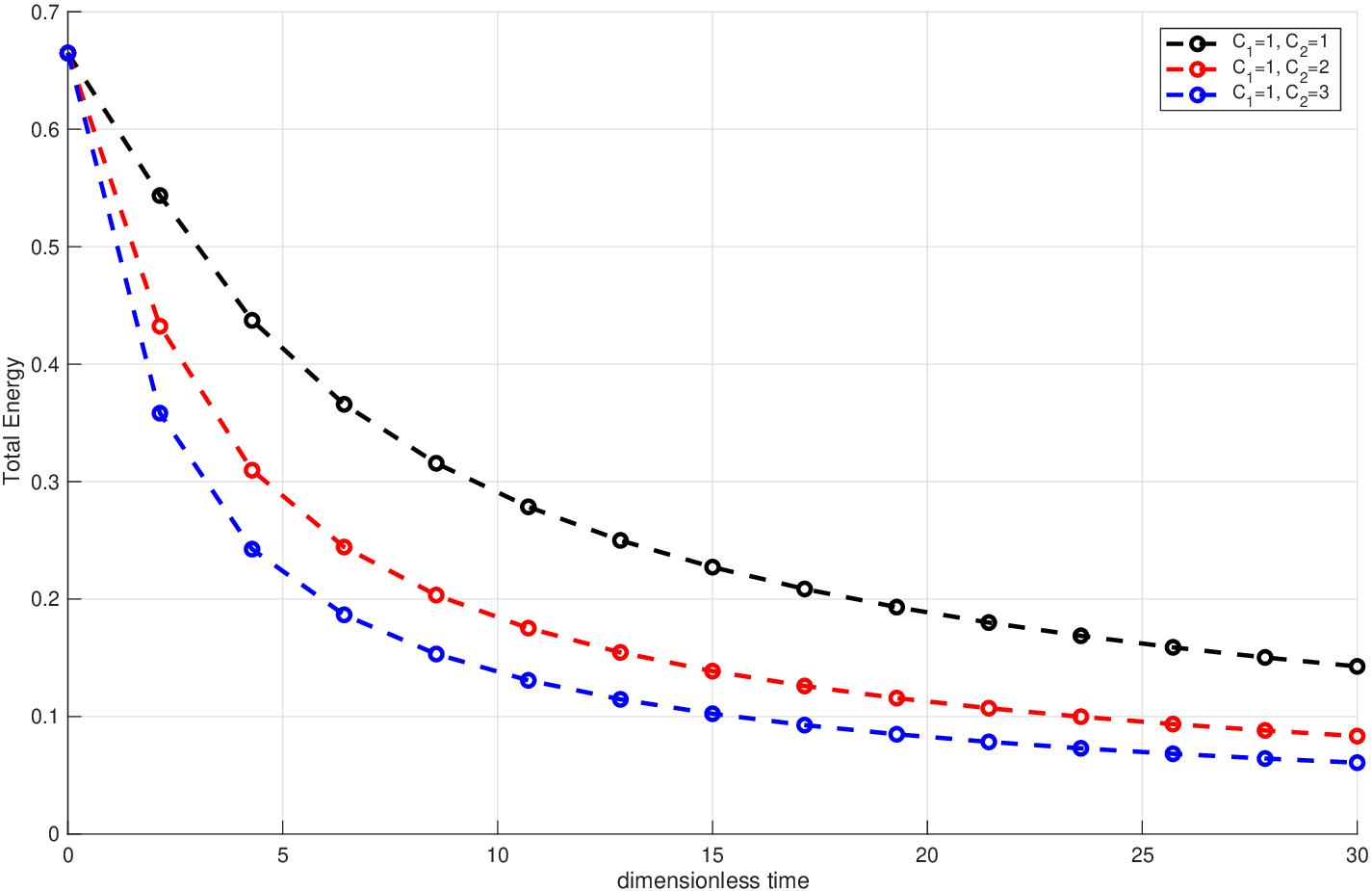}
			\caption{}
			\label{f2.1}
		\end{subfigure}
		\begin{subfigure}{.45\textwidth}
			\centering
			\includegraphics[width=1.0\textwidth]{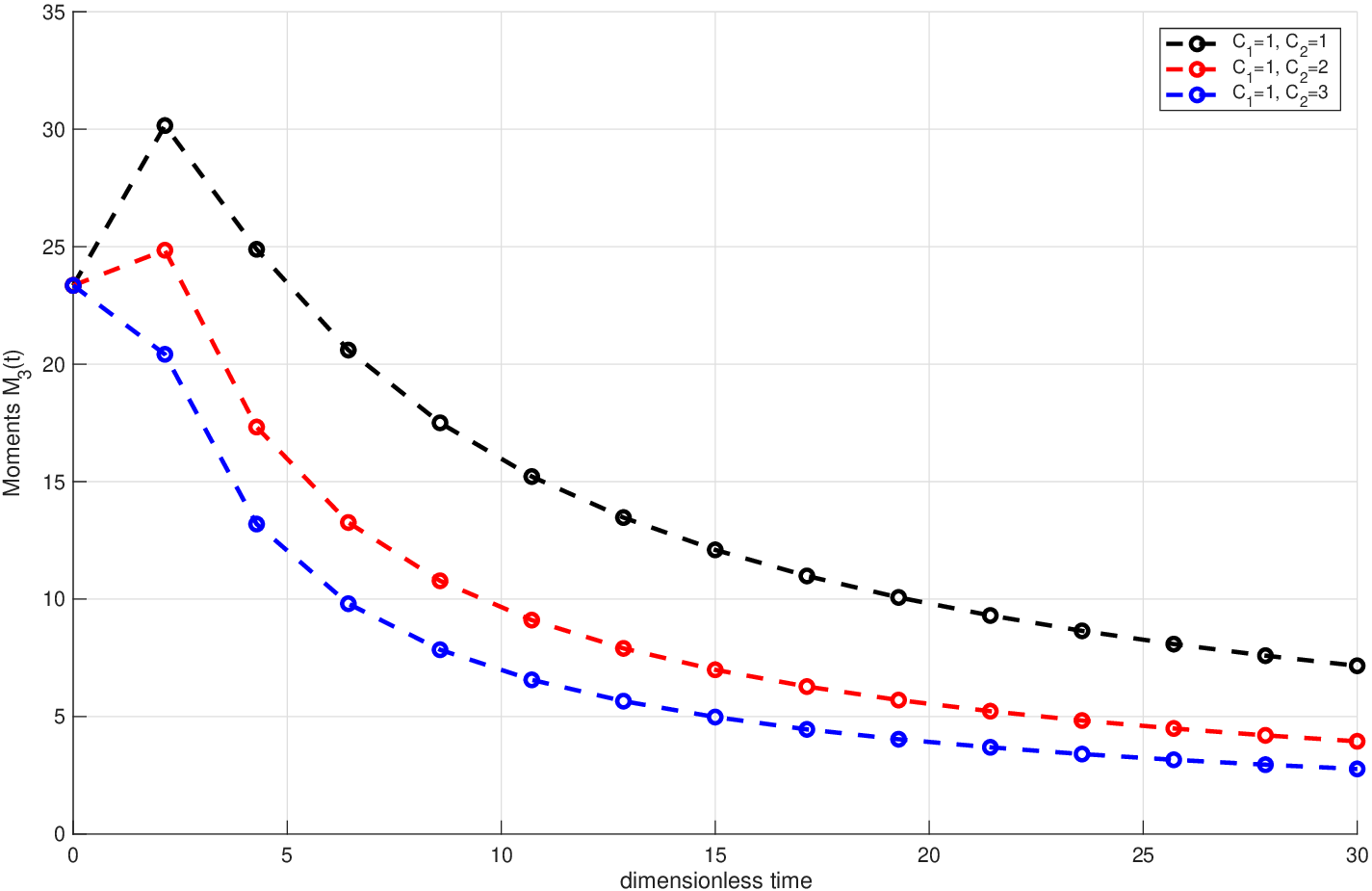}
			\caption{}
			\label{f2.3}
		\end{subfigure}
		\caption{Time evolution of the total energy (A) and third order moment (B) for different values of $C_1$ and $C_2$ when $\sigma=0.50$, $\gamma = 0.50$, and  $\ds \left|k\right|\left(\omega\right) = {\omega}^{1/3}$.}
		\label{f2}
	\end{figure}

We conclude the first test case by fixing the parameter values 
$C_1 = C_2 = 1$ and varying the degrees of homogeneity. 
The total energy and third-order moment are plotted in 
Figure~\ref{f1_sigma} for different values of $\sigma$ and $\gamma$, 
corresponding to the wave frequency function $|k|(\omega) = \omega^{1/2}$. The plots demonstrate that the energy and third-order moment vanish as time evolves in the considered bounded domain.

\begin{figure}[htp]
	\begin{subfigure}{.45\textwidth}
		\centering
		\includegraphics[width=1.0\textwidth]{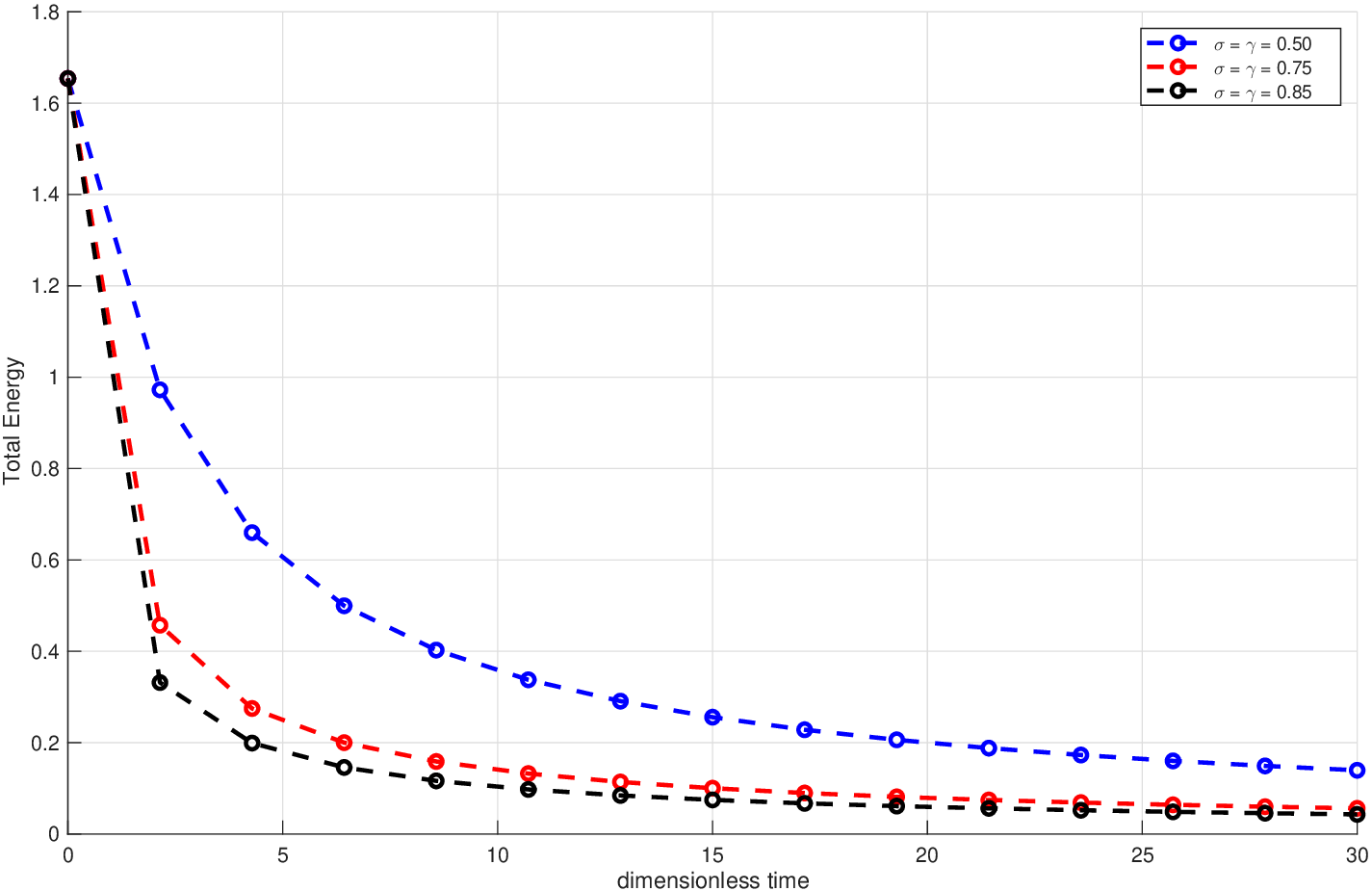}
		\caption{}
		\label{f1_sigma.1}
	\end{subfigure}
	\begin{subfigure}{.45\textwidth}
		\centering
		\includegraphics[width=1.0\textwidth]{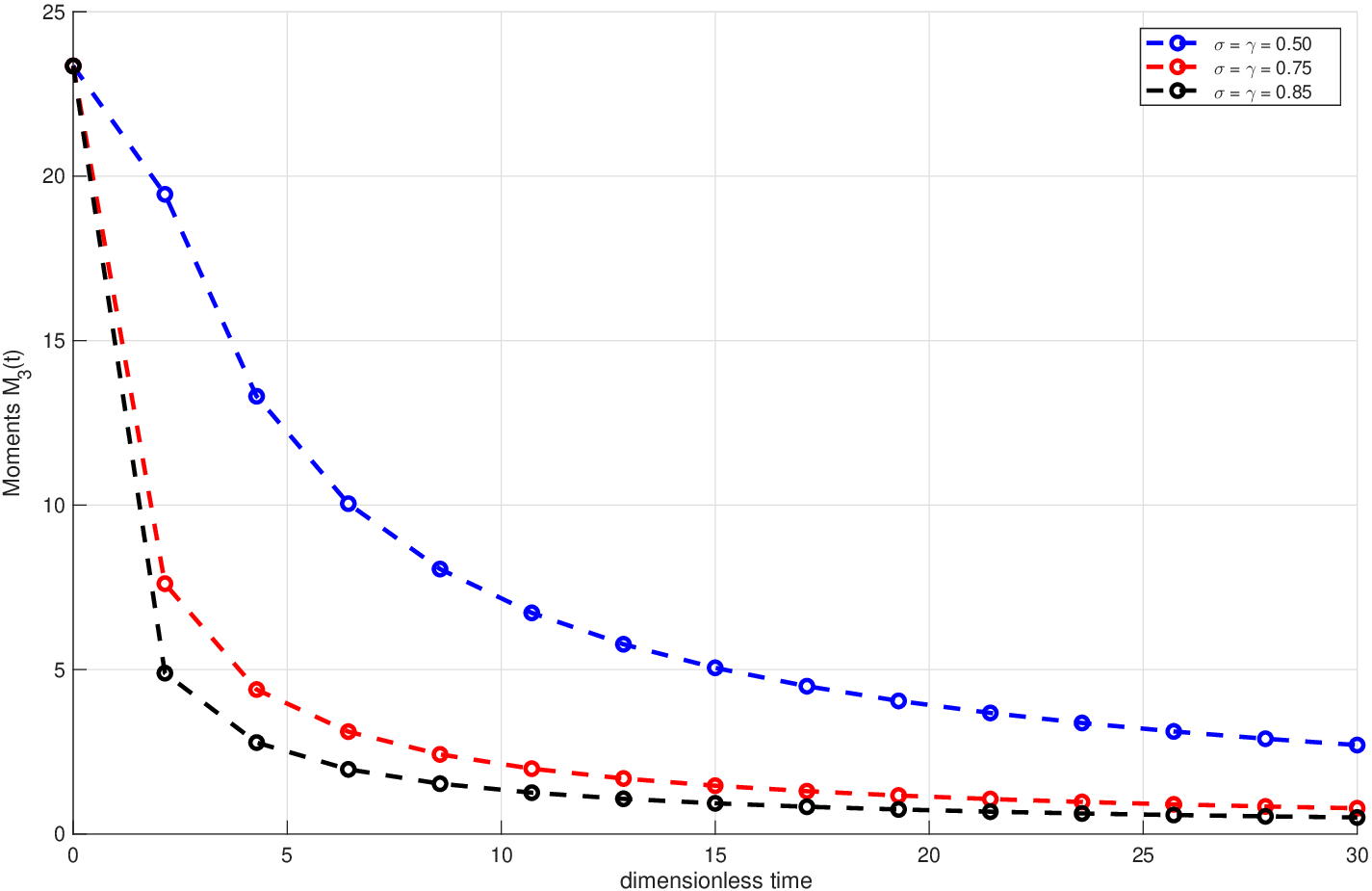}
		\caption{}
		\label{f1_sigma.3}
	\end{subfigure}
	\caption{Time evolution of the total energy (A) and third order moment (B) for different values of $\sigma$ and $\gamma$ when  $\ds \left|k\right|\left(\omega\right) = \sqrt{\omega}$.}
	\label{f1_sigma}
\end{figure}

	\subsection{Test Case II}
In the second numerical test, we consider an initial condition that is compactly supported in the low-frequency region and defined as follows:
\[
f^{\text{in}}(\omega) = 
\begin{cases}
	\exp\!\left(\dfrac{5}{|\omega - 5|^2 - 1}\right), & \text{if } |\omega - 5| \le 1, \\[0.4cm]
	0, & \text{if } |\omega - 5| > 1.
\end{cases}
\]

For the numerical computation, we use the same computational domain as in Test Case~\ref{test1}.

In Figures~\ref{f3} and \ref{f4}, we present the evolution of the total energy and the third-order moments for the wave frequency functions $\lvert k \rvert(\omega) = \omega^{1/2}$ and $\lvert k \rvert(\omega) = \omega^{1/3}$, respectively. From the plots, we observe that, for both frequency functions, the total energy exhibits asymptotic decay over time for different values of the parameters $C_1$ and $C_2$.

	\begin{figure}[htp]
		\begin{subfigure}{.45\textwidth}
			\centering
			\includegraphics[width=1.0\textwidth]{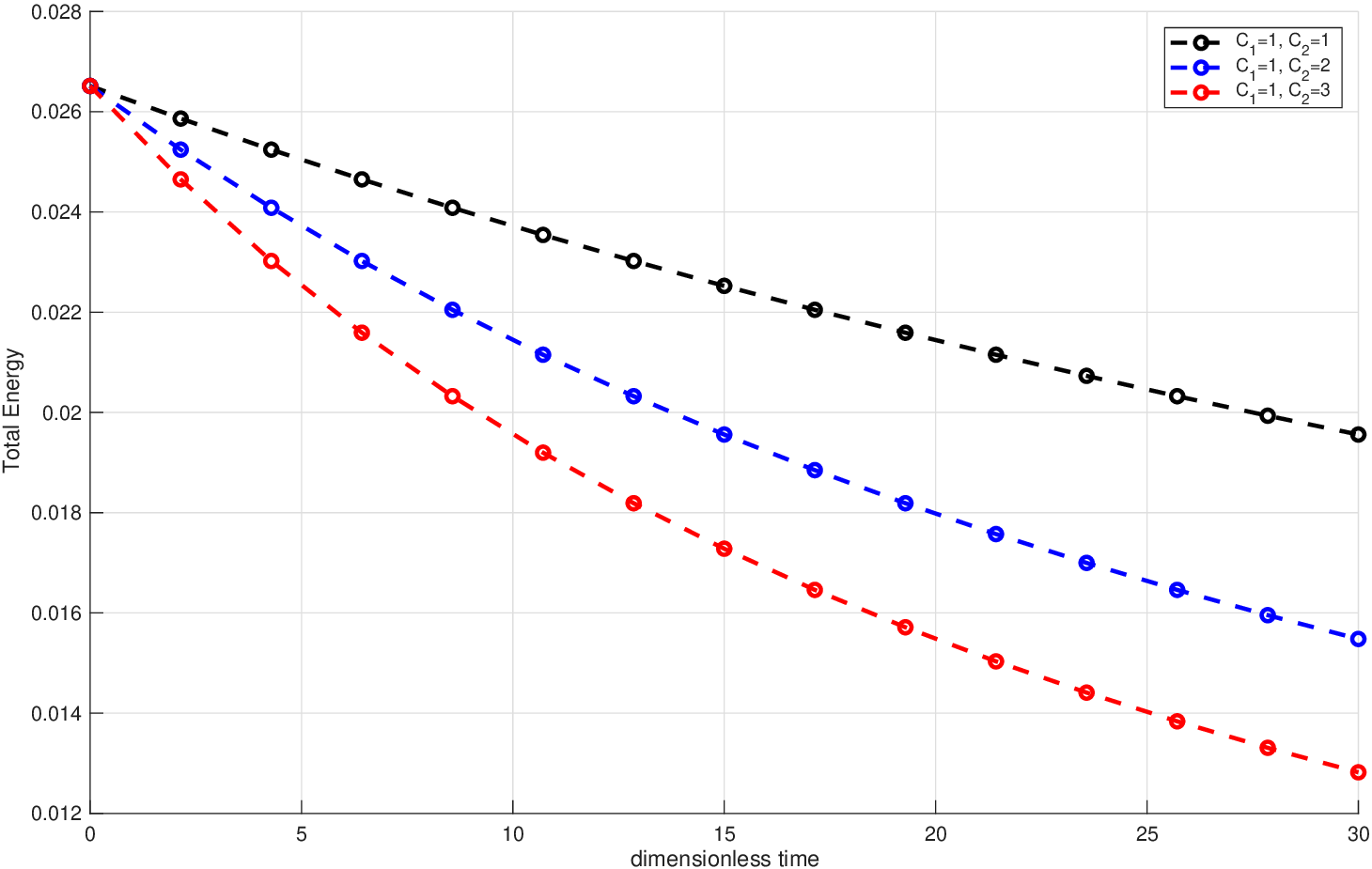}
			\caption{}
			\label{f3.1}
		\end{subfigure}
		\begin{subfigure}{.45\textwidth}
			\centering
			\includegraphics[width=1.0\textwidth]{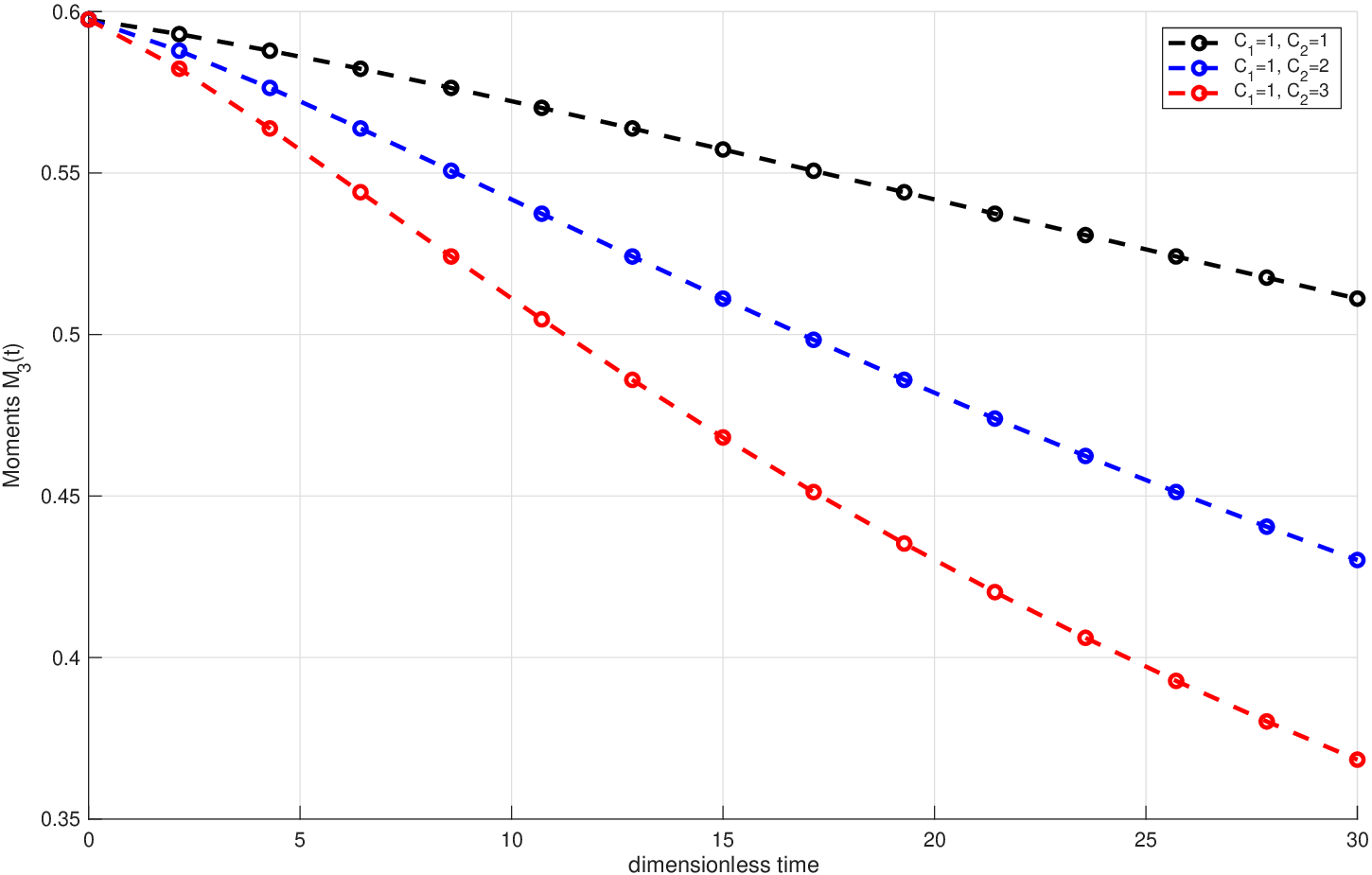}
			\caption{}
			\label{f3.3}
		\end{subfigure}
		\caption{Time evolution of the total energy (A) and third order moment (B) for different values of $C_1$ and $C_2$ when $\sigma=0.50$, $\gamma = 0.50$, and  $\ds \left|k\right|\left(\omega\right) = \sqrt{\omega}$.}
		\label{f3}
	\end{figure}
	
	\begin{figure}[htp]
		\begin{subfigure}{.45\textwidth}
			\centering
			\includegraphics[width=1.0\textwidth]{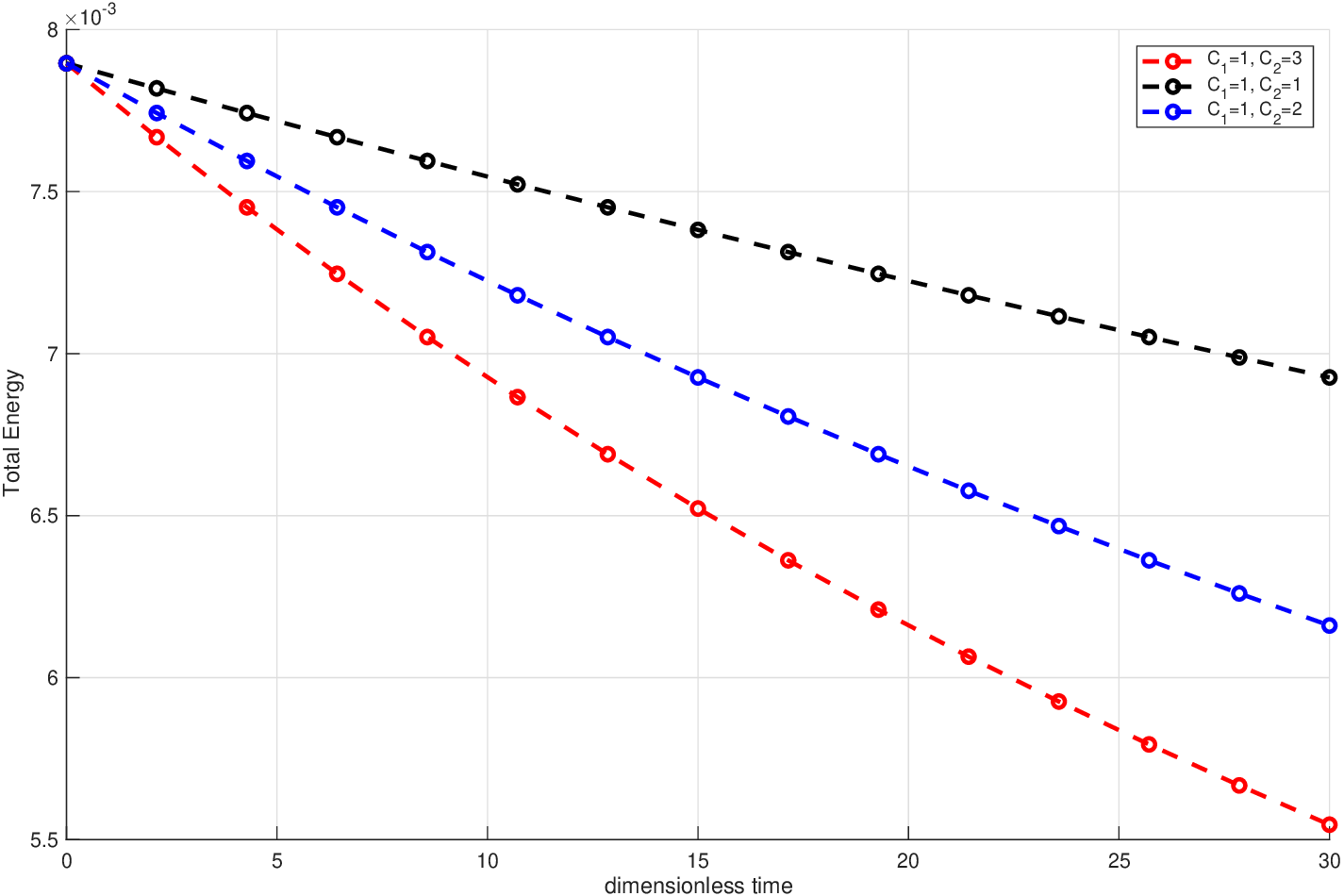}
			\caption{}
			\label{f4.1}
		\end{subfigure}
		\begin{subfigure}{.45\textwidth}
			\centering
			\includegraphics[width=1.0\textwidth]{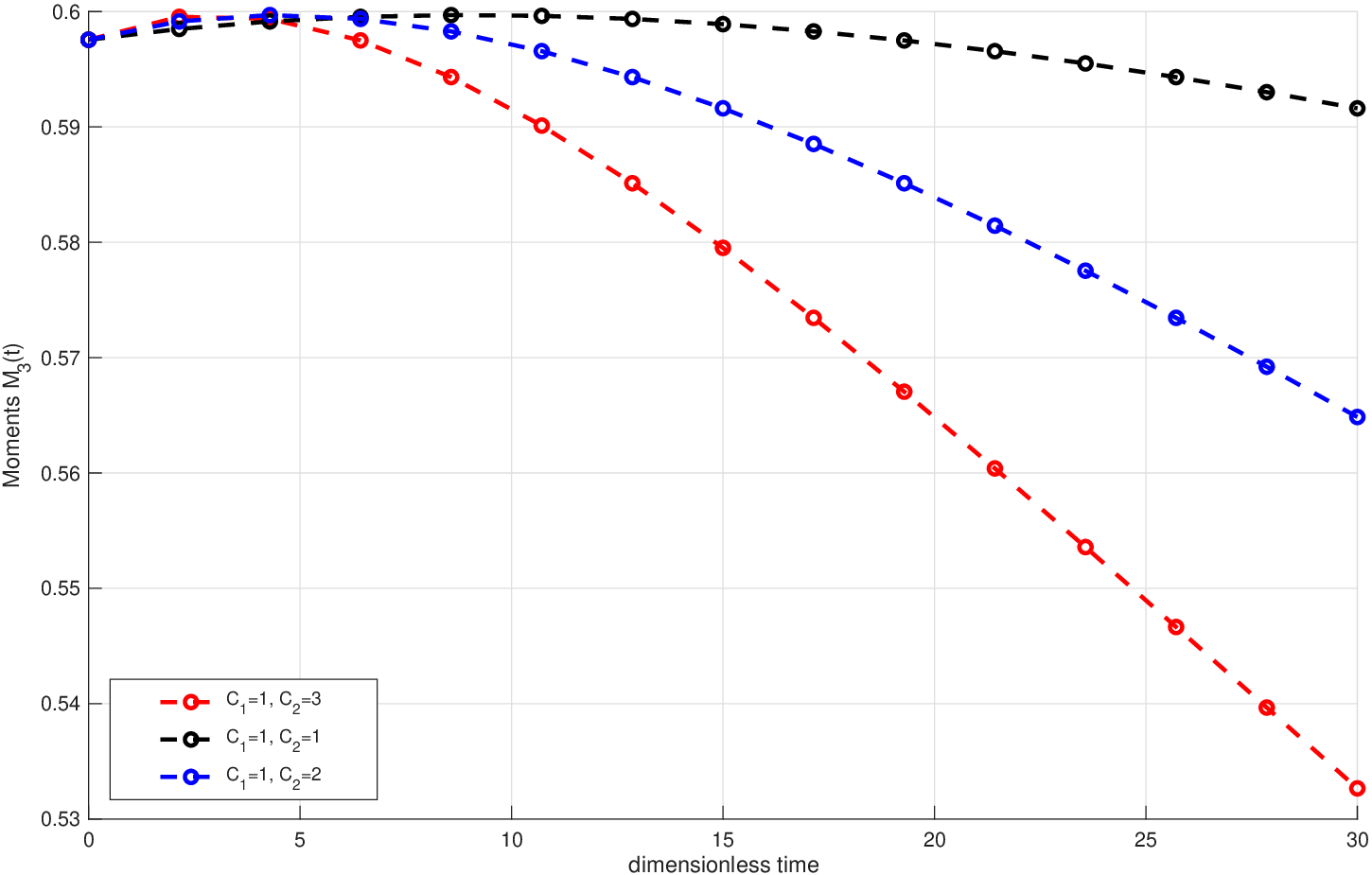}
			\caption{}
			\label{f4.3}
		\end{subfigure}
		\caption{Time evolution of the total energy  (A) and third order moment (B) for different values of $C_1$ and $C_2$ when $\sigma=0.50$, $\gamma = 0.50$, and  $\ds \left|k\right|\left(\omega\right) = {\omega}^{1/3}$.}
		\label{f4}
	\end{figure}
	
Similar to the previous test case, this numerical experiment concludes by examining the temporal evolution of the total energy and the third-order moment, as shown in Figure~\ref{f2_sigma}, for different values of the degrees of homogeneity $\sigma$ and $\gamma$, while keeping $C_1 = C_2 = 1$ fixed.

	\begin{figure}[htp]
		\begin{subfigure}{.45\textwidth}
			\centering
			\includegraphics[width=1.0\textwidth]{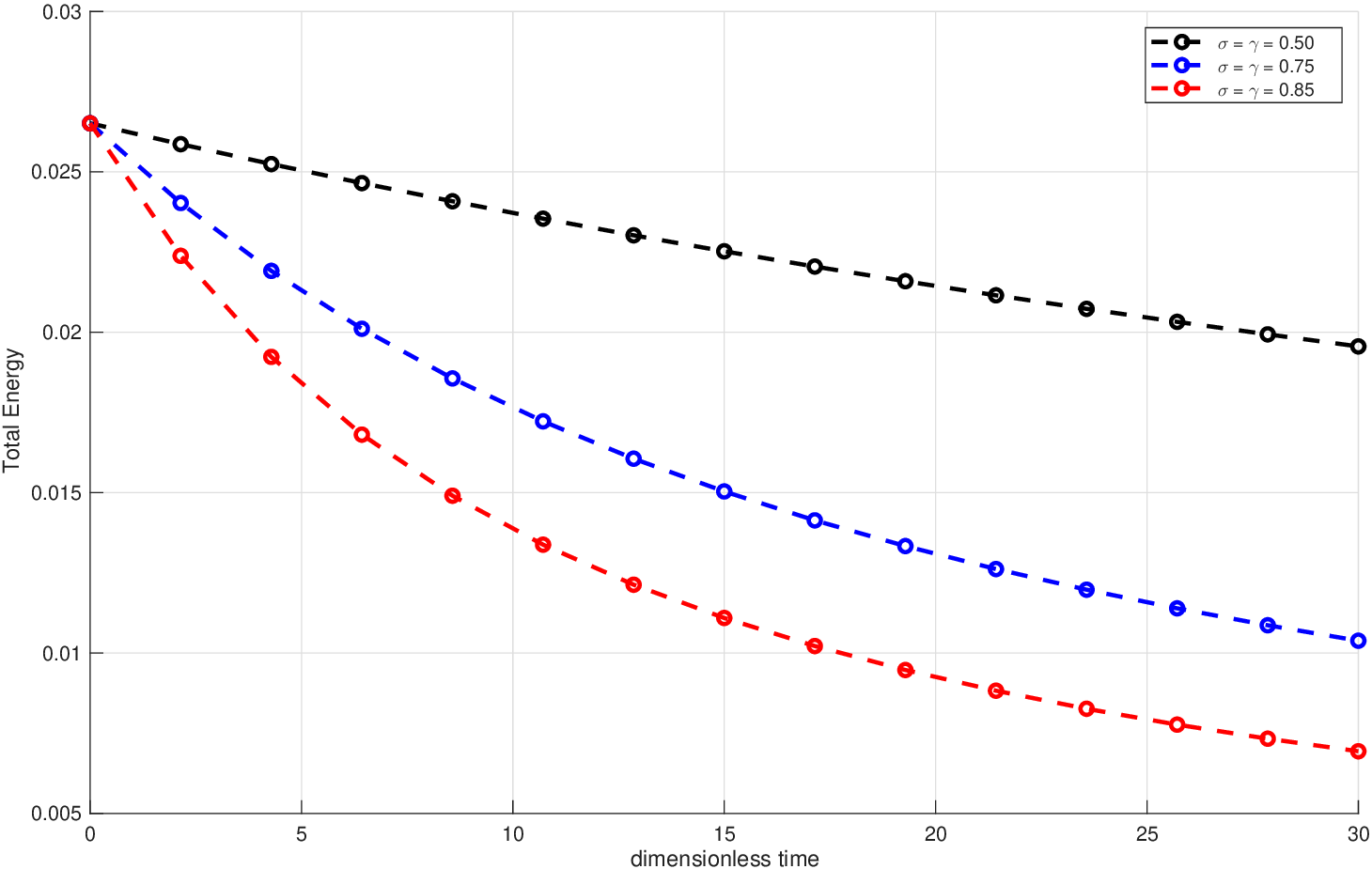}
			\caption{}
			\label{f2_sigma.1}
		\end{subfigure}
		\begin{subfigure}{.45\textwidth}
			\centering
			\includegraphics[width=1.0\textwidth]{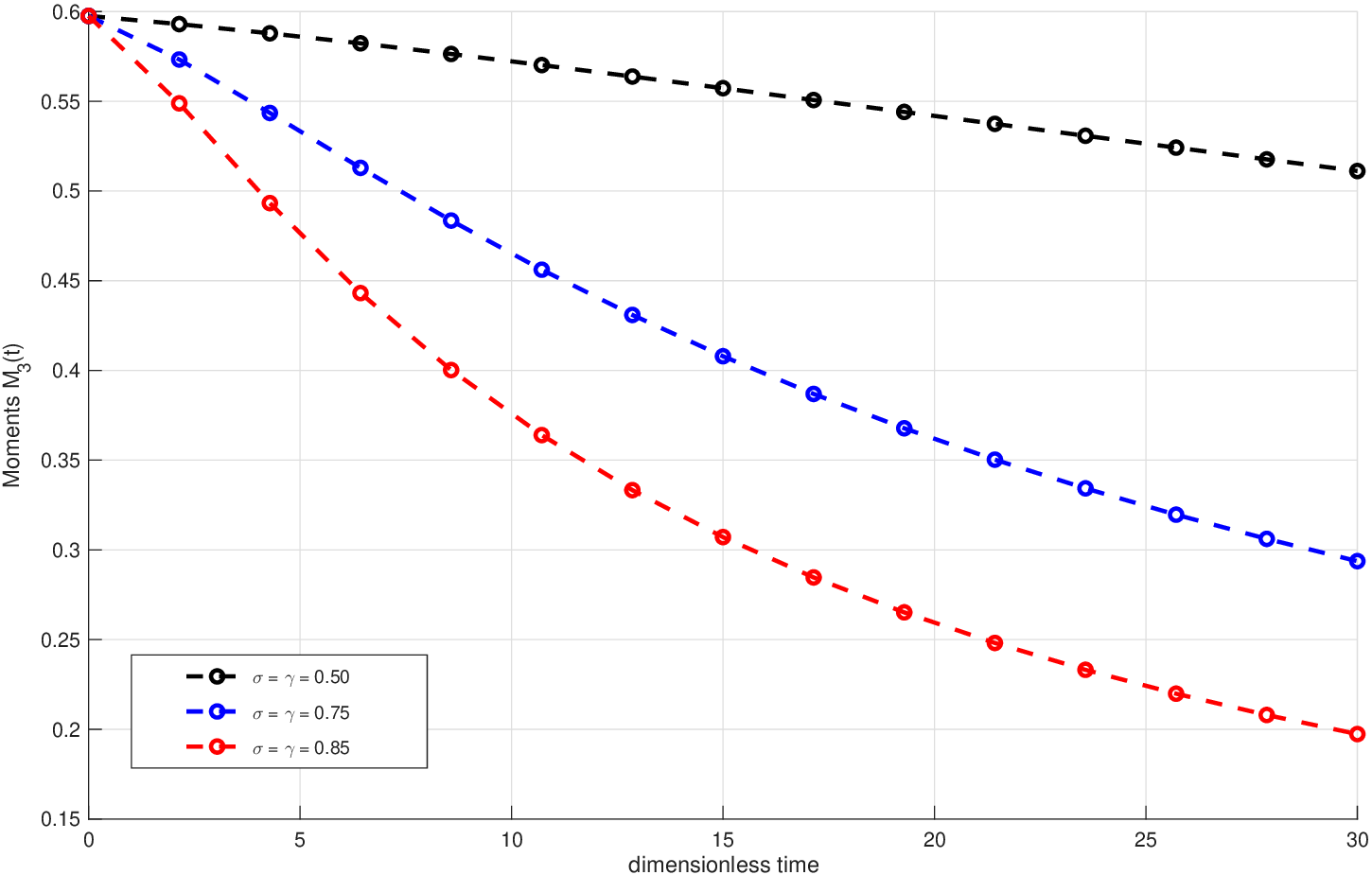}
			\caption{}
			\label{f2_sigma.3}
		\end{subfigure}
		\caption{Time evolution of the total energy (A) and third order moment (B) for different values of $\sigma$ and $\gamma$ when  $\ds \left|k\right|\left(\omega\right) = \sqrt{\omega}$.}
		\label{f2_sigma}
	\end{figure}

	\subsection{Test Case III}
	In our last test we consider mono disperse type of initial data:
	\begin{align*}
		f^{in}(\omega)  = \left\{\begin{array}{ll}
			1, &\mbox{if}\quad|\omega-1|\le \frac{1}{2} \vspace{0.2cm}\\
			0 , &\mbox{elsewhere}.
		\end{array}\right.
	\end{align*}
	
	To compute the solution, we consider the computational domain $\ds [10^{-9}, 2]$, which is divided into twenty uniform sub-intervals. Figure \ref{test3} illustrates the initial and final states ($T=30$) of the wave density function $f(t,\omega)$ for the specific parameter values $C_1 = C_2 = 1$ and $\sigma = \gamma = 0.5$, corresponding to the frequency function $\ds |k|(\omega) = \omega^{1/2}$.
	\begin{figure}[H]
		\begin{subfigure}{.45\textwidth}
			\centering
			\includegraphics[width=1.0\textwidth]{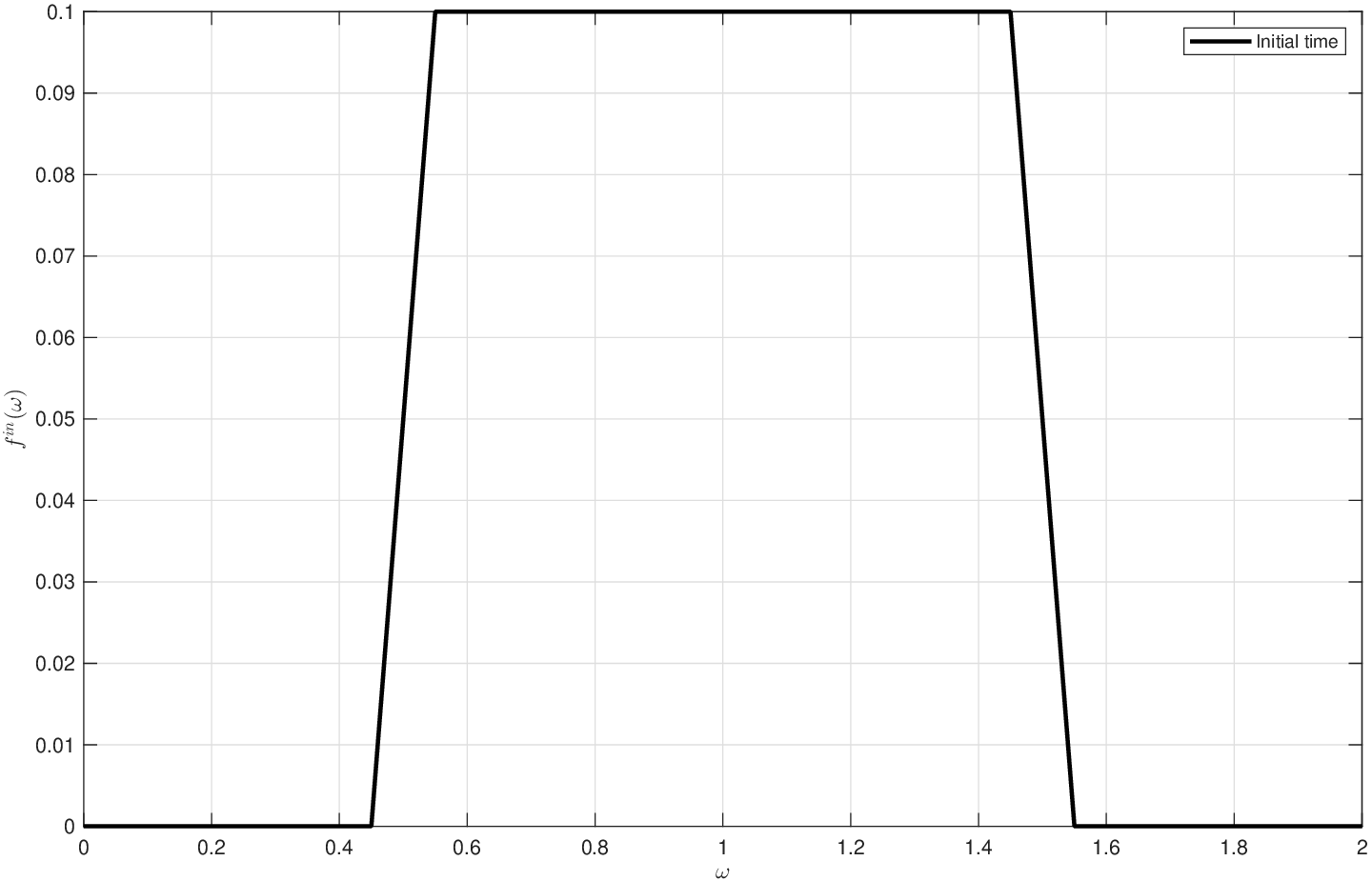}
			\caption{Initial Time}
			\label{test3.1}
		\end{subfigure}
		\begin{subfigure}{.45\textwidth}
			\centering
			\includegraphics[width=1.0\textwidth]{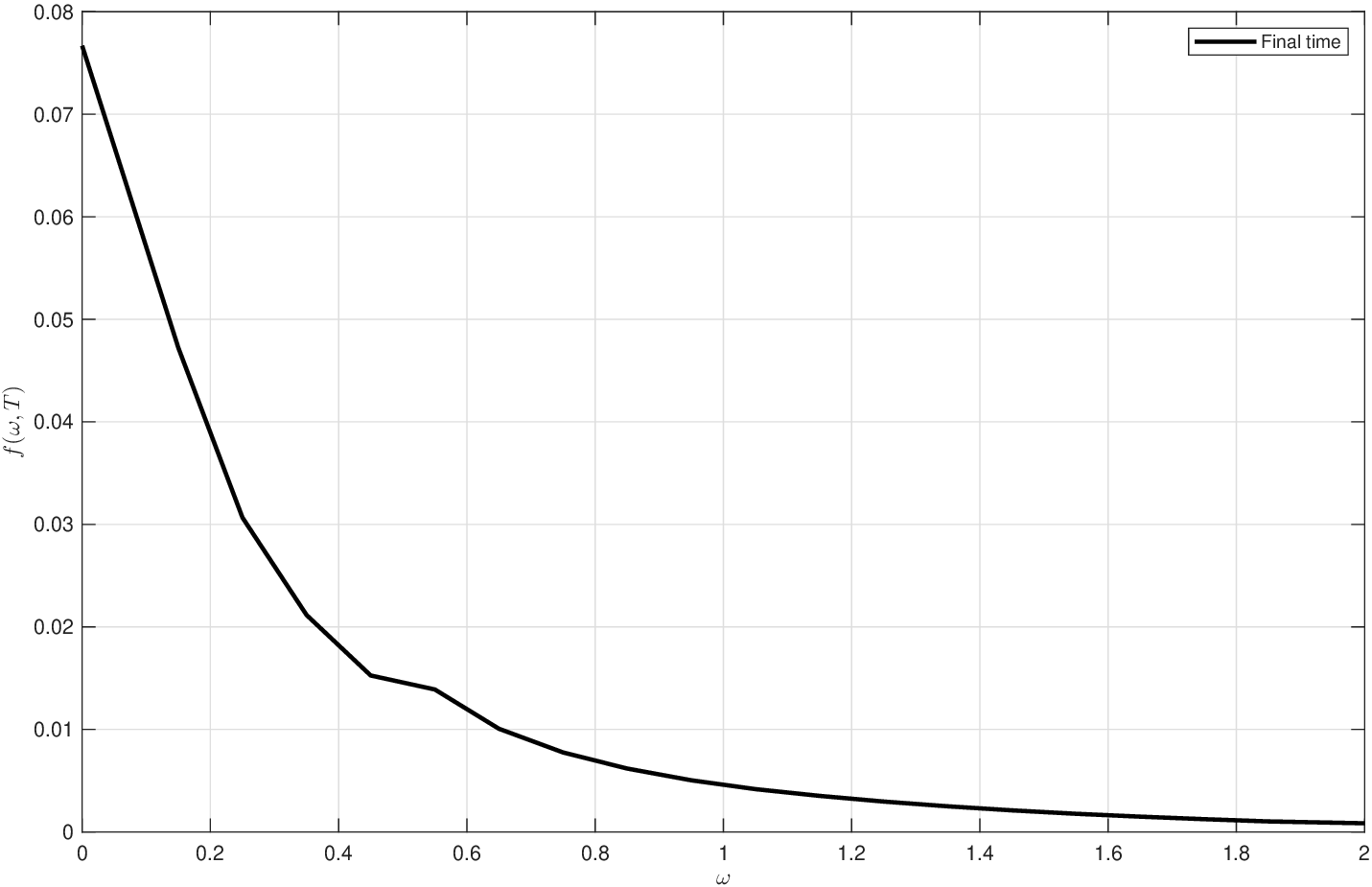}
			\caption{Final Time}
			\label{test3.2}
		\end{subfigure}
		\caption{Evolution of wave density $f(\omega)$ at initial and final time.}
		\label{test3}
	\end{figure}
	 
In Figure~\ref{f5}, we present the time evolution of the total energy (A) and the third-order moment (B) for different values of $C_1$ and $C_2$. The plots exhibit a clear decay over time, which aligns well with the theoretical predictions.

	\begin{figure}[H]
		\begin{subfigure}{.45\textwidth}
			\centering
			\includegraphics[width=1.0\textwidth]{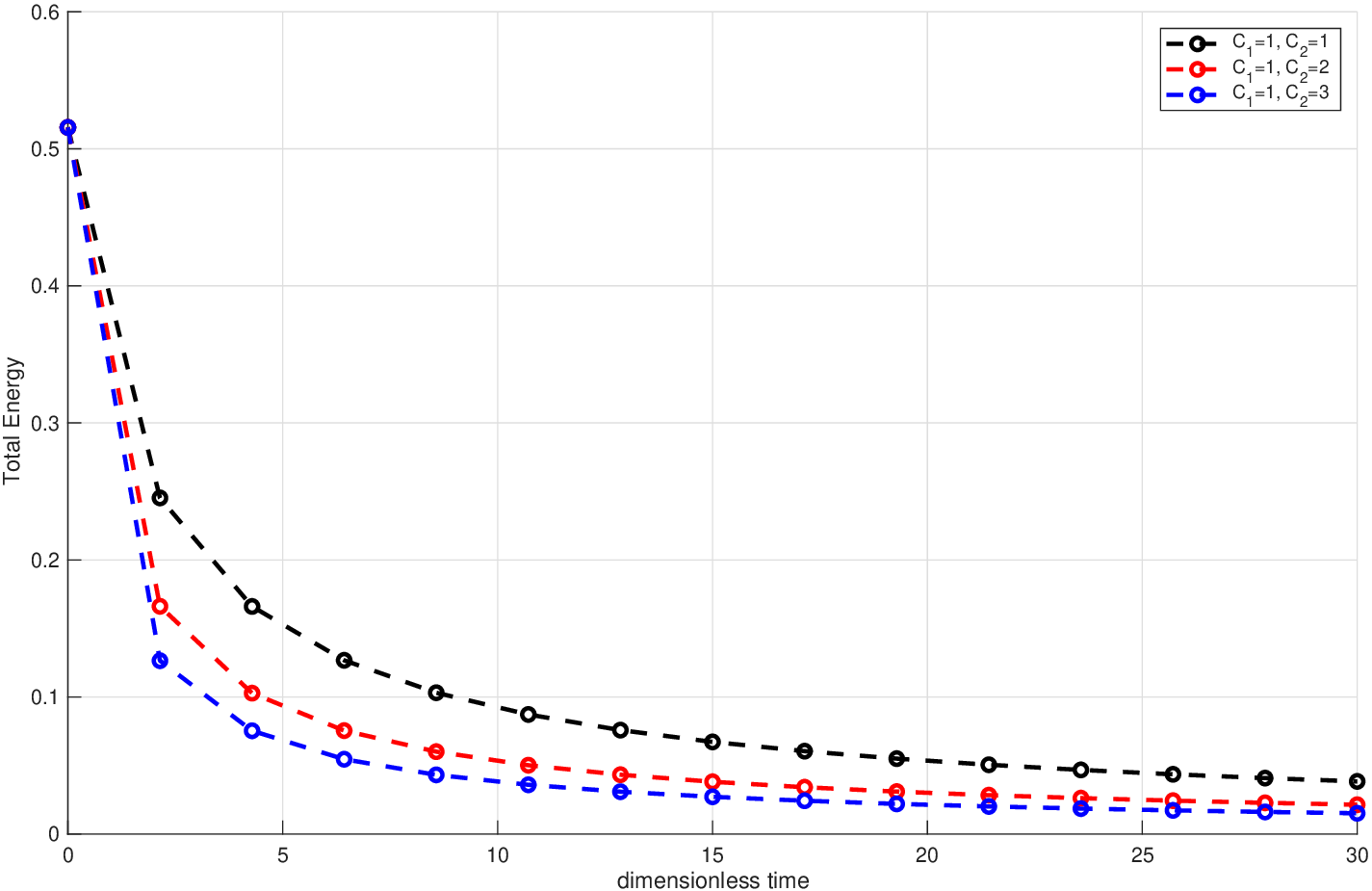}
			\caption{}
			\label{f5.1}
		\end{subfigure}
		\begin{subfigure}{.45\textwidth}
			\centering
			\includegraphics[width=1.0\textwidth]{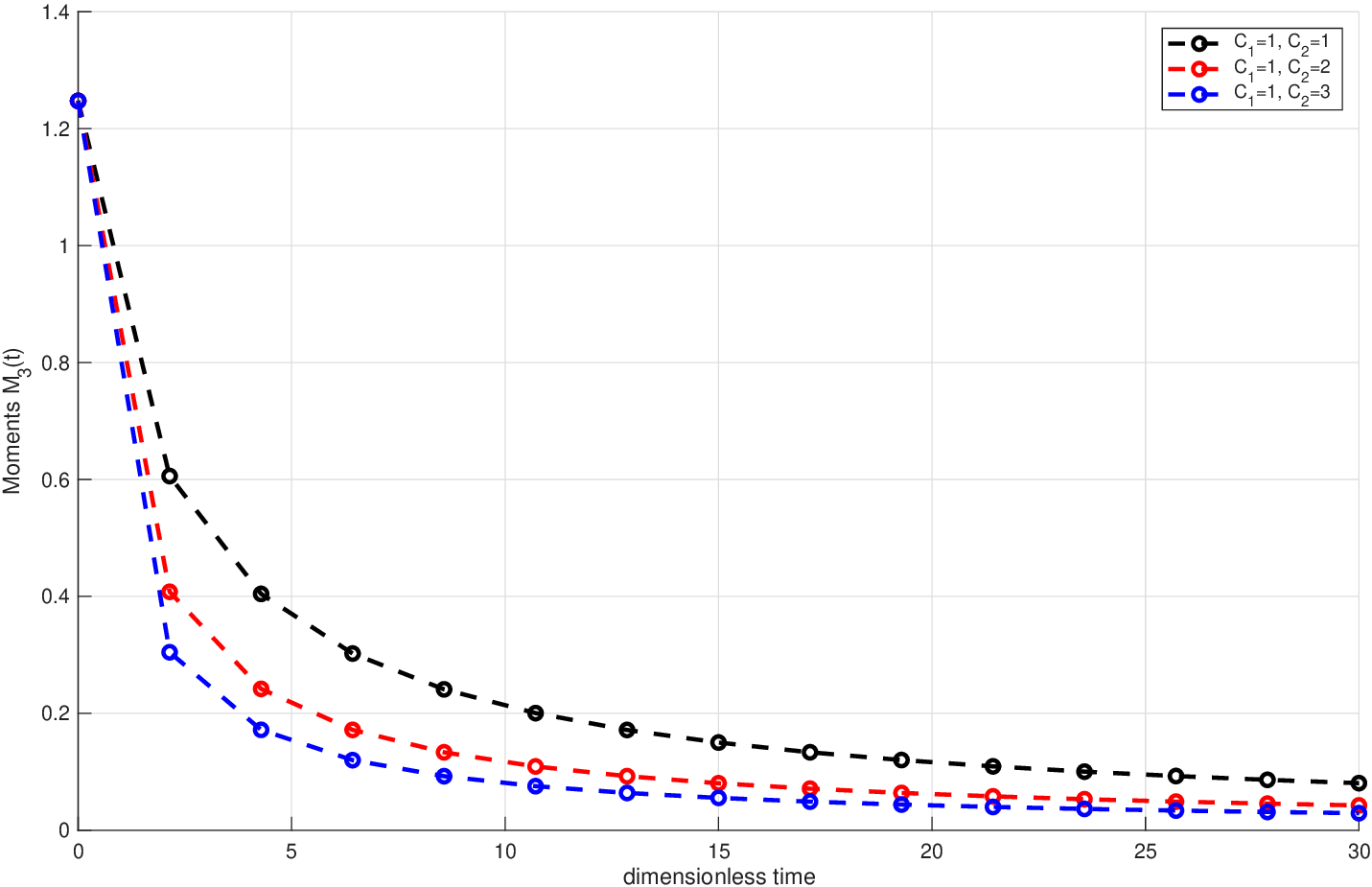}
			\caption{}
			\label{f5.3}
		\end{subfigure}
		\caption{Time evolution of the total energy (A) and third order moment (B) for different values of $C_1$ and $C_2$ when $\sigma=0.50$, $\gamma = 0.50$, and  $\ds \left|k\right|\left(\omega\right) = \sqrt{\omega}$.}
		\label{f5}
	\end{figure}
	
In Figure~\ref{f6}, we present the temporal evolution of the same quantities for the frequency function $|k|(\omega) = \omega^{1/3}$. It is observed that, in this case as well, the proposed scheme exhibits a similar decay of the total energy over time.

	\begin{figure}[H]
		\begin{subfigure}{.45\textwidth}
			\centering
			\includegraphics[width=1.0\textwidth]{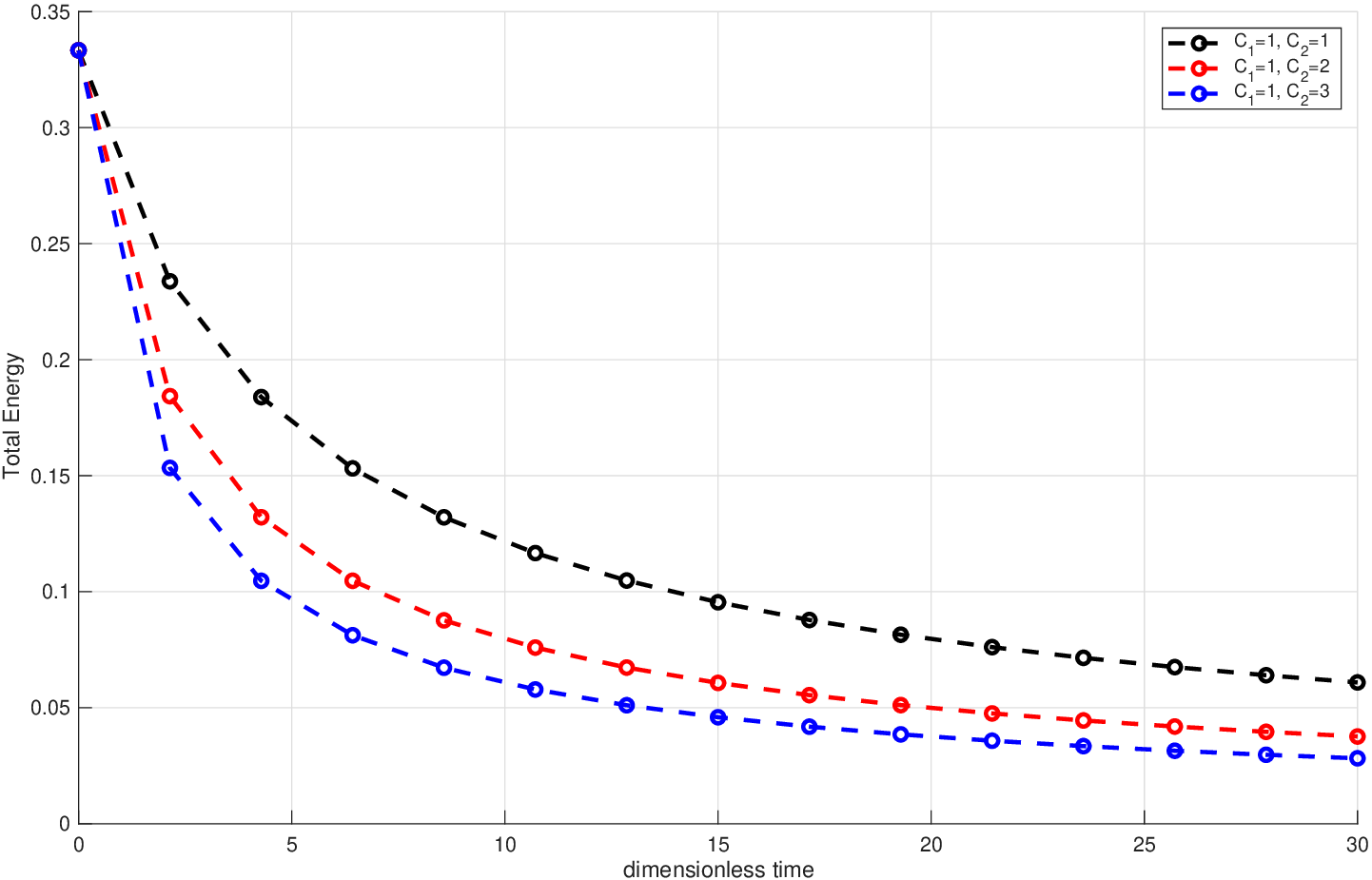}
			\caption{}
			\label{f6.1}
		\end{subfigure}
		\begin{subfigure}{.45\textwidth}
			\centering
			\includegraphics[width=1.0\textwidth]{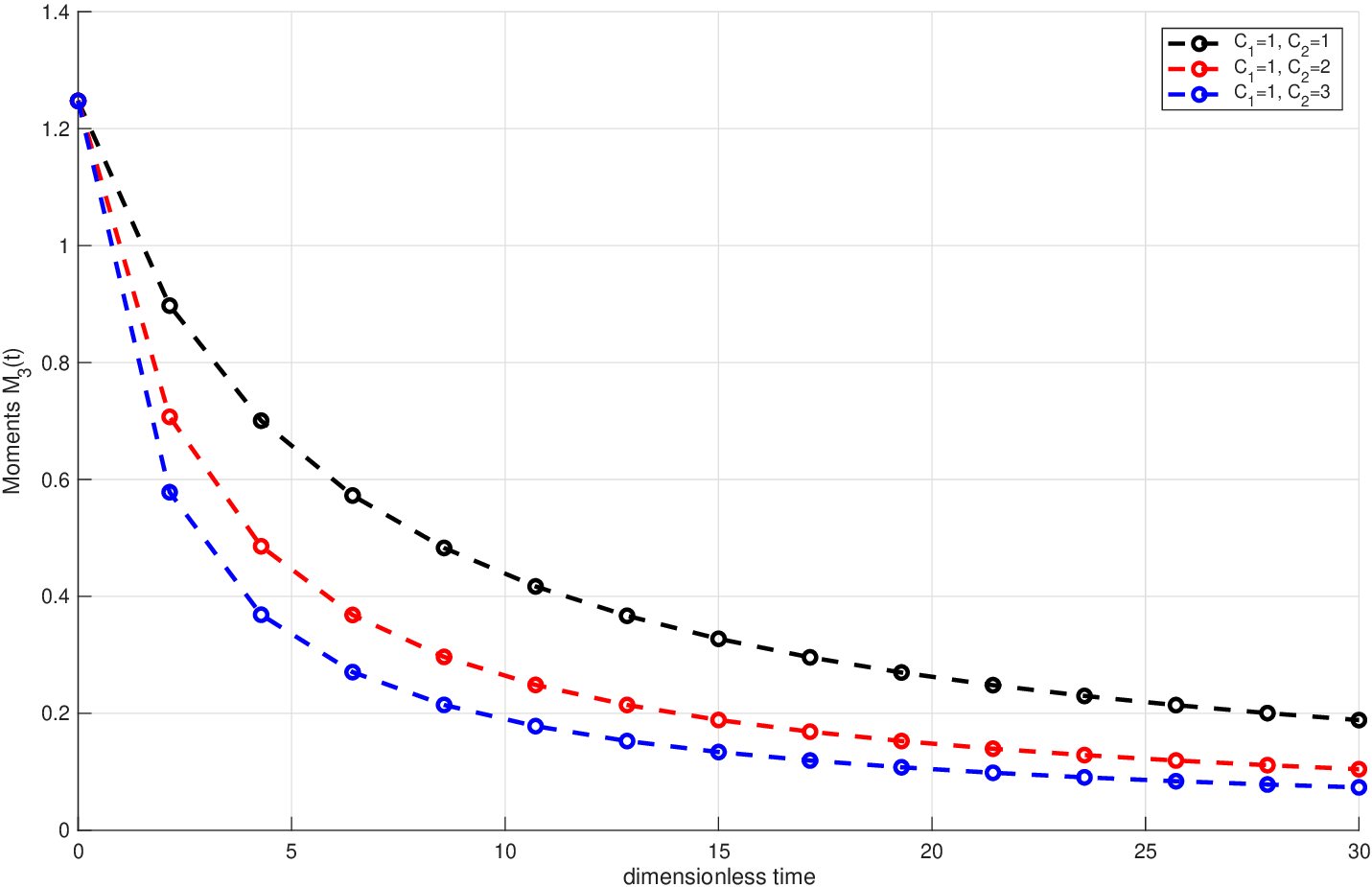}
			\caption{}
			\label{f6.3}
		\end{subfigure}
		\caption{Time evolution of the total energy  and third order moment for different values of $C_1$ and $C_2$ when $\sigma=0.50$, $\gamma = 0.50$, and  $\ds \left|k\right|\left(\omega\right) = {\omega}^{1/3}$.}
		\label{f6}
	\end{figure}
	
Finally, we conclude our numerical experiments by fixing the values of $C_1$ and $C_2$ while varying the degrees of homogeneity $\sigma$ and $\gamma$. In this case as well, the approximate solutions exhibit behavior consistent with the theoretical results.

	\begin{figure}[H]
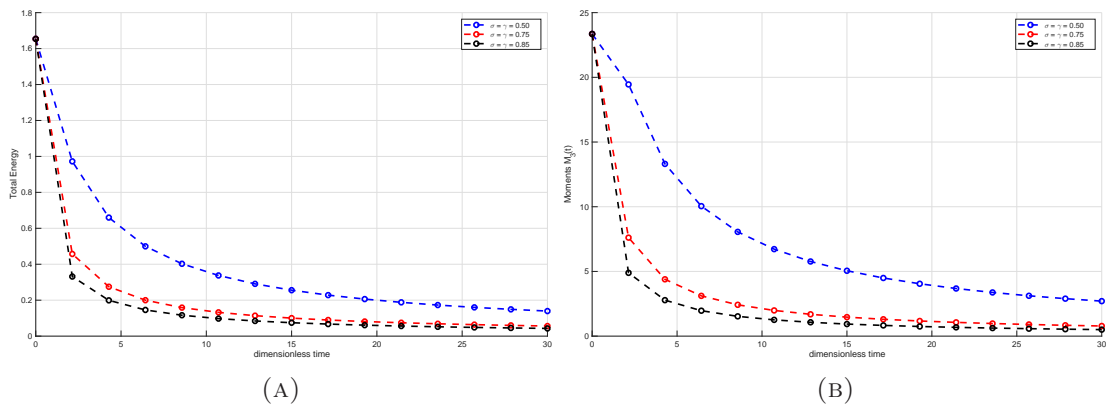

		\begin{subfigure}{.45\textwidth}
			\centering
			\includegraphics[width=1.0\textwidth]{EPS/Energy_Exp_SigmaGamma}
			\caption{}
			\label{f3_sigma.1}
		\end{subfigure}
		\begin{subfigure}{.45\textwidth}
			\centering
			\includegraphics[width=1.0\textwidth]{EPS/ThirdMom_Exp_SigmaGamma}
			\caption{}
			\label{f3_sigma.3}
		\end{subfigure}
		\caption{Time evolution of the total energy (A) and third order moment (B) for different values of $\sigma$ and $\gamma$ when  $\ds \left|k\right|\left(\omega\right) = \sqrt{\omega}$.}
		\label{f3_sigma}
	\end{figure}

	\bibliographystyle{unsrt}
\bibliographystyle{plain}

\bibliography{WaveTurbulenc}
\end{document}